\documentclass[12pt]{amsart}
\usepackage{amsmath, amssymb,latexsym }
\usepackage{verbatim}

\usepackage{graphicx}

\newcommand{\tcal}{\mathcal{T}}

\newcommand{\kahler}{K\"ahler}
\def\e{{\rm e}}
\def\sq2{\sqrt{2}}

\def\t2{{\mathbb T}^2}
\def\s2{{\mathbb S}^2}

\def\N{\mathbb{N}}

\def\R{\mathbb{R}}
\def\Z{\mathbb{Z}}
\def\C{\mathbb{C}}

\def\e{{\rm e}}
\def\sq2{\sqrt{2}}

\def\t2{{\mathbb T}^2}
\def\s2{{\mathbb S}^2}

\def\N{\mathbb{N}}

\def\R{\mathbb{R}}
\def\Z{\mathbb{Z}}
\def\C{\mathbb{C}}

\def\hto0{\xrightarrow{\hbar\to 0}}

\def\rto0{\xrightarrow{r\to 0}}

\newcommand{\HH}{\mathcal H}

 \newcommand{\nj}{n_{\mbox{crit}}(\lambda_j)}

 \newcommand{\szego}{Szeg\"o}
\renewcommand{\le}{\leqslant}
\renewcommand{\ge}{\geqslant}

\newcommand{\E}{\mathbb E}
\newcommand{\PP}{\mathbb P}
\newcommand{\Ss}{{\mathbb S}}
\newcommand{\Hh}{{\mathbb H}}

\newcommand{\dbar}{\bar\partial}
\newcommand{\ddbar}{\partial\dbar}

\renewcommand{\Re}{{\operatorname{Re\,}}}
\renewcommand{\Im}{{\operatorname{Im\,}}}
\newcommand{\half}{{\frac{1}{2}}}
\renewcommand{\phi}{\varphi}

\newcommand{\ccal}{\mathcal{C}}
\newcommand{\dcal}{\mathcal{D}}

\newcommand{\gcal}{\mathcal{G}}

\newcommand{\hcal}{\mathcal{H}}
\newcommand{\lcal}{\mathcal{L}}

\newcommand{\ncal}{\mathcal{N}}
\newcommand{\pcal}{\mathcal{P}}
\newcommand{\ocal}{\mathcal{O}}
\newcommand{\scal}{\mathcal{S}}

\renewcommand{\H}{{\mathbf H}}
\newcommand{\acal}{\mathcal{A}}
\newcommand{\sgn}{\mbox{sgn}}

\newcommand{\La}{\Lambda}
\newcommand{\la}{\lambda}
\newcommand{\de}{\delta}
\newcommand{\De}{\Delta}

\newtheorem{theo}{{\sc Theorem}}[section]
\newtheorem{cor}[theo]{{\sc Corollary}}
\newtheorem{conj}[theo]{{\sc Conjecture}}
\newtheorem{lem}[theo]{{\sc Lemma}}

\newtheorem{prop}[theo]{{\sc Proposition}}

\newenvironment{rem}{\medskip\noindent{\it Remark:\/} }{\medskip}
\newenvironment{defin}{\medskip\noindent{\it Definition:\/} }{\medskip}

\newtheorem{maintheo}{{\sc Theorem}}
\newtheorem{maincor}[maintheo]{{\sc Corollary}}
\newtheorem{mainprop}[maintheo]{{\sc Proposition}}
\newtheorem{mainlem}[maintheo]{{\sc Lemma}}

\author{Steve Zelditch}

\title[Eigenfunctions and nodal sets  ]
{Eigenfunctions   and nodal sets }

\address{Department of Mathematics, Northwestern  University, Evanston, IL 60208, USA}

\email{zelditch@math.northwestern.edu}

\thanks{Research partially supported by NSF grant  \# DMS-0904252.}

\date{\today}

\begin{document}


\maketitle

\begin{abstract} This is a survey of recent results on  nodal sets
of eigenfunctions of the Laplacian on Riemannian manifolds. The emphasis is on complex nodal
sets of analytic continuations of eigenfunctions. 
\bigskip



\end{abstract}
\bigskip

\bigskip


Let $(M, g)$ be a (usually compact) Riemannian manifold of dimension $n$, and let $\{\phi_j\}$ denote an orthonormal
basis of eigenfunctions of its  Laplacian,
\begin{equation}\label{EP}  \Delta_g\; \phi_j= - \lambda_j^2\; \phi_j \;\;\; \langle \phi_j, \phi_k \rangle = \delta_{jk}. \end{equation} Here $\langle u, v \rangle = \int_M u v dV_g$ where $dV_g$ is the volume form of $(M, g)$.  If  $\partial M \not=0$ we
impose Dirichlet or Neumann boundary conditions. When $(M, g)$ is
compact, the spectrum of $\Delta$  is discrete and can be put in  non-decreasing order $\lambda_0 < \lambda_1 \leq \lambda_2
\uparrow \infty$.  The eigenvalues $\lambda_j^2$ are often termed energies while their square roots
$\lambda_j$ are often termed the frequencies. 
The nodal set of an eigenfunction $\phi_{\lambda}$  is the zero
set
\begin{equation} Z_{\phi_{\lambda}} \footnote{In difference references we use either the notation $Z$ or $\ncal$
for the nodal set. Sometimes we use the subscript $\phi_{\lambda}$ and sometimes only $\lambda$.}= \{x \in M: \phi_{\lambda}(x) = 0\}. \end{equation} The aim of this
survey is to review some recent results on the
$\hcal^{n-1}$-surface measure and on the yet more difficult problem of  the spatial
distribution
of the nodal sets, i.e. the behavior of the integrals \begin{equation} \label{INT} 
 \frac{1}{\lambda_j} \int_{Z_{\phi_{\lambda_j}}} f
dS_{\lambda_j} , \;\; (f \in C(M)) \end{equation}  as $\lambda \to \infty$. Here, 
 $dS_{\lambda} = d \hcal^{n-1}$ denotes the Riemannian
hypersurface volume form on $Z_{\phi_{\lambda}}$.More
generally, we consider the same problems for any level set
\begin{equation} \ncal_{\phi_{\lambda}}^c : = \{\phi_{\lambda} = c\},\end{equation}
where $c$ is a constant (which in general may depend on
$\lambda$). Nodal sets are special level sets and much more
attention has been devoted to them than other level sets, but it
is often of interest to study general level sets and in particular
`high level' sets or excursion sets. 

We have recently written surveys \cite{Z5,Z6} on the global harmonic analysis of eigenfunctions, which include
some discussion of nodal sets and critical point sets. To the extent possible, we hope to avoid repeating
what is written there, but  inevitably  there will be some overlap. We refer there and \cite{H} for background on well-established results. We also decided to cover some results of  research in progress  (especially from \cite{Z3}, but also 
on $L^{\infty}$ quantum ergodic theory).  We generally refer to the results as `Conjectures' even when detailed
arguments exist, since they have not yet been carefully examined by others.

There are two basic intuitions underlying many of the conjectures and results on eigenfunctions:

\begin{itemize}

\item Eigenfunctions of $\Delta_g$-eigenvalue $- \lambda^2$ are similar to polynomials of degree $\lambda$. 
In particular, $Z_{\lambda}$ is similar to a real algebraic variety of degree $\lambda$.

 Of course, this intuition
is most reliable when $(M, g)$ is real analytic. It is quite unclear at this time how reliable it is for general $C^{\infty}$
metrics, although there are some recent improvements on volumes and equidistribution in the smooth case. 

\item High frequency behavior of eigenfunctions reflects the dynamics of the geodesic flow $G^t: S^*M \to S^*M$
of $M$. Here, $S^*M$ is the unit co-sphere bundle of $(M, g)$. 

When the dynamics is ``chaotic" (highly ergodic), then eigenfunctions
are de-localized and 
behave like Gaussian random waves of almost fixed frequency. This motivates the study of Gaussian random wave models for eigenfunctions, and suggests that in the `chaotic case'    nodal sets should be asympotically  uniformly
distributed.

 When $G^t$ is completely integrable, model eigenfunctions are highly
localized and their nodal sets are often exhibit quite regular patterns. 
The latter heuristic is not necessarily expected
when there exist high multiplicities, as for rational flat tori, and then some weaker randomness can enter.

\end{itemize}

Both of these general intuitions lead to predictions about nodal sets and critical point sets. Most of the predictions
are well beyond current or forseeable techniques to settle.  A principal theme of this survey is that the analogues of such `wild'
predictions can sometimes be proved for real analytic $(M, g)$  if one analytically continues eigenfunctions to the complexification
of $M$ and studies complex nodal sets instead of real ones. 
\bigskip

{\it As with algebraic varieties, nodal sets in the real analytic case  are better behaved in the complex domain than the real domain. That is, zero sets of analytic continuations of eigenfunctions to the complexification of $M$ behave
like complex algebraic varieties and also  reflect
the dynamics of the geodesic flow.  }
\medskip

It is well-known that the  complexification of $M$ can be identified with a neighborhood of the zero-section 
of  the phase space $T^* M$.  That is one reason why dynamics of the geodesic flow has greater impact on the
complex nodal set. 

We will exhibit a number of relatively recent results (some unpublished elsewhere) which justify this viewpoint:

\begin{itemize}

\item Theorem \ref{BNP}, which shows that complex methods can be used to give upper bounds
on the number of nodal components of Dirichlet or Neumann eigenfunctions which ``touch the boundary"
of a real analytic plane domain. 

\item Theorem \ref{ZERO} on the limit distribution of the normalized currents of integration
$$\frac{1}{\lambda_{j_k}} [Z_{\phi_{j_k}^{\C}}] $$
over the complex zero sets of ``ergodic eigenfunctions'' in the complex domain.

\item Theorem \ref{ZMEASURE} and  Corollary \ref{ZDSM}, which show that the similar currents for
analytic continuations of ``Riemannian random waves'' tend to the same limit almost surely. Thus,
the prediction that zero sets of ergodic eigenfunctions agrees with that of random waves is correct
in the complex domain. 

\item Sharper results on the distribution of intersections points of nodal sets and geodesics on complexified
real analytic surfaces (Theorem \ref{MAINCOR}).
\end{itemize}

Our analysis of  nodal sets in the complex domain is based on the use of complex Fourier
integral
techniques (i.e. generalized Paley-Wiener theory). The principal tools are the analytic continuation of the Poisson-wave kernel
and the Szeg\"o kernel in the complex domain. They become Fourier integral operators with complex phase and
with wave fronts along the complexified geodesic flow. One can read off the growth properties of complexified
eigenfunctions from mapping properties of such operators.  Log moduli of complexified spectral projectors  are
asymptotically extremal  pluri-subharmonic functions for all $(M, g)$.  These ideas are the basis 
of the articles \cite{Z2, TZ,  Z3,Z4,Z8, Z9,He}.  Such ideas have antecedents in work of S. Bernstein,  Baouendi-
Goulaouic, and Donnelly-Fefferman, Guillemin,  F.H. Lin  (among others) .

We note that the focus on complex nodal sets only makes sense for real analytic $(M, g)$. It is possible that
one can study ``almost analytic extensions'' of eigenfunctions for  general $C^{\infty}$ metrics in a similar
spirit,  but this is
just a speculation and certain key methods  break down when $g$ is not real analytic.  Hence the
results in the $C^{\infty}$ case are much less precise than in the real analytic case.

It should also be mentioned that much work on eigenfunctions concerns ground states, i.e. the first and second
eigenfunctions.  Unfortunately, we do not have the space or expertise to review the results on ground states in this
survey. For a sample we refer to \cite{Me}. Further, many if not all of the techniques and results surveyed here
have generalizations to Schr\"odinger operators $-\hbar^2 \Delta + V$. For the sake of brevity we confine the
discussion to the Laplacian.
\bigskip

\subsection{Notation}

The first  notational issue is whether to choose $\Delta_g$ to be the positive or negative Laplacian. The traditional
choice 
\begin{equation} \label{DELTADEF} \Delta_g =   \frac{1}{\sqrt{g}}\sum_{i,j=1}^n
\frac{\partial}{\partial x_i} \left( g^{ij} \sqrt{g}
\frac{\partial}{\partial x_j} \right).  \end{equation}
makes $\Delta_g$ is negative, but many authors call $- \Delta_g$ the Laplacian to avoid the minus signs. 
Also, the metric $g$ is often fixed and is dropped from the notation. 

A less traditional choice is to denote eigenvalues by $\lambda^2$ rather than $\lambda$. It is a common
convention in microlocal analysis and so we adopt it here. But we warn that $\lambda$ is often used to
denote $\Delta$-eigenvalues as is \cite{DF,H}.

We sometimes denote eigenfunctions of eigenvalue $- \lambda^2$ by $\phi_{\lambda}$ when we only  wish
to emphasize the corresponding  eigenvalue and do not need $\phi_{\lambda}$ to be part of an orthonormal
basis. For instance, when $\Delta_g$ has multiplicities as on the standard sphere or rational torus, there
are many possible orthonormal bases. But estimates on $\hcal^{n-1}(Z_{\phi_{\lambda}})$ do not depend
on whether $\phi_{\lambda}$ is included in the orthonormal basis. 

\subsection{Acknowledgements}

Thanks to C. D. Sogge and B. Shiffman for helpful comments on the exposition, and to S. Dyatlov for
a stimulating discussion of $L^{\infty}$ quantum ergodicity.

\section{\label{ESTIMATES} Basic estimates of eigenfunctions}

We start by collecting some classical elliptic  estimates and their applications to eigenfunctions. 

First,  the general Sobolev estimate:
Let $w \in C_0^{\infty}(\Omega)$ where $\Omega \subset \R^n$ with $n \geq 3$. Then there exists $C > 0$: 
$$\left(\int_{\Omega} |w|^{\frac{2n}{n-2}} \right)^{\frac{n-2}{n}} \leq C \int_{\Omega} |\nabla w|^2. $$


Next, we recall the Bernstein gradient estimates:

\begin{theo} \cite{DF3} Local eigenfunctions of a Riemannian manifold satisfy:

\begin{enumerate}

\item $L^{2}$ Bernstein estimate:
\begin{equation} \left( \int_{B(p,r)} |\nabla \phi_{\lambda}|^2 dV
\right)^{1/2} \leq \frac{C \lambda}{r} \left( \int_{B(p,r)} |
\phi_{\lambda}|^2 dV \right)^{1/2}. \end{equation}

\item $L^{\infty}$ Bernstein estimate: There exists $K > 0$ so that
\begin{equation}  \max_{x \in B(p, r)} |\nabla \phi_{\lambda}(x)|
\leq \frac{C \lambda^K}{r}  \max_{x \in B(p, r)}
|\phi_{\lambda}(x)|. \end{equation}

\item Dong's improved bound:
$$\max_{B_r(p)}|\nabla \phi_{\lambda}|\leq{C_1\sqrt{\lambda}\over r}\max_{B_r(p)}|\phi_{\lambda}| $$ for $ r\leq
C_2\lambda^{-1/4}.$

\end{enumerate}

\end{theo}

Another well-known estimate is the doubling estimate:

\begin{theo} \label{DOUBLE} (Donnelly-Fefferman, Lin)  and \cite{H} (Lemma 6.1.1) Let $\phi_{\lambda}$ be a global
eigenfunction of  a $C^{\infty}$ $(M, g)$
there exists $C = C(M, g)$ and $r_0$  such that for $0 < r < r_0$,
$$\frac{1}{Vol(B_{2r}(a))} \int_{B_{2r}(a)} |\phi_{\lambda}|^2
dV_g \leq e^{C \lambda} \frac{1}{Vol(B_{r}(a))} \int_{B_{r}(a)}
|\phi_{\lambda}|^2 dV_g. $$

Further,
\begin{equation} \max_{B(p, r)} |\phi_{\lambda}(x)| \leq
\left(\frac{r}{r'} \right)^{C \lambda} \max_{x \in B(p, r')}
|\phi_{\lambda}(x)|, \;\; (0 < r' < r). \end{equation}

\end{theo}

The doubling estimates imply  the vanishing order estimates. Let
$a \in M$ and suppose that $u(a) = 0$. By the vanishing order
$\nu(u, a)$ of $u$ at $a$ is meant the largest positive integer
such that $D^{\alpha} u(a) = 0$ for all $|\alpha| \leq \nu$. 

\begin{theo}
 Suppose that $M$ is compact and of dimension $n$. Then there exist constants $C(n), C_2(n)$ depending only on the dimension such that
the  the vanishing order $\nu(u, a)$ of $u$ at $a \in M$ satisfies
$\nu(u, a) \leq C(n) \; N(0, 1) + C_2(n)$ for all $a \in
B_{1/4}(0)$. In the case of a global  eigenfunction, $\nu(\phi_{\lambda},
a) \leq C(M, g) \lambda.$
\end{theo}

We now recall quantitative lower bound estimates. They follow from doubling estimates and also
from Carleman inequalities.

\begin{theo}\label{DFED}   Suppose that $M$ is compact and that $\phi_{\lambda}$ is a
global eigenfunction,  $\Delta \phi_{\lambda} = \lambda^2 \phi_{\lambda}$.
 Then for all $p, r$,  there exist $C, C' > 0$ so that $$ \max_{x \in B(p, r)} |\phi_{\lambda}(x)| \geq C'
e^{- C \lambda}.$$ \end{theo}

Local lower bounds on $\frac{1}{\lambda} \log |\phi_{\lambda}^{\C}|$  follow from doubling estimates. 
They imply that there exists $A, \delta > 0$ so that,  for any $\zeta_0 \in \overline{M_{\tau/2}}$,
\begin{equation} \label{DE} \sup_{\zeta \in B_{\delta}(\zeta_0)}  |\phi_{\lambda}(\zeta) | \geq C e^{- A \lambda}.  \end{equation}
Indeed, there of course exists a point $x_0 \in M$ so that $|\phi_{\lambda}(x_0)| \geq 1$. Any point
of $\overline{M}_{\tau/2}$ can be linked to this point by a smooth curve ofuniformly  bounded length. We then
choose $\delta$ sufficiently small so that the $\delta$-tube around the curve lies in $M_{\tau}$ and link
$B_{\delta}(\zeta)$ to $B_{\delta}(x_0)$ by a  chain of $\delta$-balls in $M_{\tau}$ where the number of links
in the chain is uniformly bounded above as $\zeta $ varies in $M_{\tau}$. If 
the balls are denoted $B_j$ we have $\sup_{B_{j + 1}} |\phi_{\lambda}| \leq e^{\beta \lambda}
\sup_{B_{j}} |\phi_{\lambda}| $ since $B_{j + 1} \subset 2 B_j$.  The growth
estimate implies that for any ball $B$,  $\sup_{2 B}  |\phi_{\lambda} | \leq e^{C \lambda} \sup_{B} |\phi_{\lambda}|$. 
Since the number of balls is uniformly
bounded,  $$1 \leq \sup_{B_{\delta}(x_0)} |\phi_{\lambda} | \leq e^{A \lambda} \sup_{B_{\delta}(\zeta)} |\phi_{\lambda}|$$
and  we get
a contradiction if no such $A$ exists. 

As an illustration, Gaussian beams such as highest weight
spherical harmonics decay at a rate $e^{- C \lambda d(x, \gamma)}$
away from a stable elliptic orbit $\gamma$. Hence if the closure
of an open set is disjoint from $\gamma$, one has a uniform
exponential decay rate which saturate the lower bounds.

We now recall sup-norm estimates of eigenfunctions which follow from the local Weyl law:
$$\begin{array}{l}
\Pi_{\lambda}(x,x)  := \sum_{\lambda_{\nu} \leq \lambda}
|\phi_{\nu}(x)|^2  = (2\pi)^{-n} \int_{p(x,\xi)\le \lambda}
d\xi + R(\lambda,x) \end{array}
$$
with uniform remainder bounds
$$|R(\lambda,x)|\le
C\lambda^{n - 1}, \quad x\in M.
$$
Since the integral in the local Weyl law is a continuous function
of $\lambda$ and since the spectrum of the Laplacian is discrete,
this immediately gives $$\sum_{\lambda_\nu=\lambda}|\phi_\nu(x)|^2
\le 2C\lambda^{n-1}$$  which in turn yields
\begin{equation}\label{03}
||\phi_{\lambda}||_{C^0}  = O(\lambda^{\frac{n-1}{2}})
\end{equation}
on any compact Riemannian manifold.

\subsection{$L^p$ estimates}

The classical Sogge estimates state that,  for any compact Riemannian manifold of dimension $n$, we have
\begin{equation}\label{i.8}
\frac{\|\phi_{\lambda}\|_p}{\|\phi_{\lambda}\|_2}=O(\lambda^{\delta(p)}), \quad 2\le p\le \infty,
\end{equation}
where 
\begin{equation} \delta(p)=
\begin{cases}
n(\tfrac12-\tfrac1p)-\tfrac12, \quad \tfrac{2(n+1)}{n-1}\le p\le \infty
\\
\tfrac{n-1}2(\tfrac12-\tfrac1p),\quad 2\le p\le \tfrac{2(n+1)}{n-1}.
\end{cases}
\end{equation}
Since we often use surfaces as an illustrantion, we note that in dimension $2$ one has for $\la\ge1$,
\begin{equation}\label{1.3}
\|\phi_\la\|_{L^p(M)}\le C\la^{\frac12 (\frac12 -\frac1p)} \|\phi_\la\|_{L^2(M)},
\quad 2\le p\le 6,
\end{equation}
and
\begin{equation}\label{1.4}
\|\phi_\la\|_{L^p(M)}\le C\la^{2(\frac12-\frac1p)-\frac12}\|e_\la\|_{L^2(M)},
\quad  6\le p\le \infty.
\end{equation}
These estimates are also sharp for the round sphere $S^2$. The first 
estimate, \eqref{1.3},  is saturated by  highest weight spherical harmonics. 
  The second estimate, \eqref{1.4}, is sharp due to the
zonal functions on $S^2$, which concentrate at points.   We go over these examples in \S \ref{S2}.

\section{Volume and equidistribution  problems on nodal sets and level sets}

We begin the survey by stating some of the principal problems an
results regarding nodal sets and more general level sets. Some of
the problems are intentionally stated in vague terms that admit a
number of rigorous formulations.

\subsection{\label{AREA} Hypersurface areas of nodal sets}

 One of the  principal problems  on nodal sets is to measure their
hypersurface volume. In the real analytic case, 
 Donnelly-Fefferman ( \cite{DF} (see also \cite{Lin}) ) proved:

\begin{theo}  
\label{DF}

Let $(M, g)$ be a compact real analytic  Riemannian
manifold, with or without boundary. Then there exist $c_1, C_2$
depending only on $(M, g)$ such that
$$c_1 \lambda \leq {\mathcal
H}^{m-1}(Z_{\phi_{\lambda}}) \leq C_2 \lambda, \;\;\;\;\;\;(\Delta
\phi_{\lambda} = \lambda^2 \phi_{\lambda}; c_1, C_2 > 0).
$$
\end{theo}

  The bounds  were  conjectured by S. T. Yau
 \cite{Y1,Y2} for all $C^{\infty}$ $(M,g)$, but this remains an open problem.
The lower bound was proved for all $C^{\infty}$ metrics for
surfaces, i.e. for $n = 2$ by Br\"uning \cite{Br}. For general
$C^{\infty}$ metrics the sharp upper and lower bounds are not
known, although there has been some recent progress that we
consider below.

The nodal hypersurface bounds are consistent with the heuristic that  $\phi_{\lambda}$ is the
analogue on a Riemannian manifold of a polynomial of degree
$\lambda$, since  the hypersurface volume of a real
algebraic variety is bounded by its degree.

\subsection{Equidistribution of nodal sets in the real domain}

The equidistribution problem for nodal sets is to study the behavior of the integrals \eqref{INT}
of general continuous functions $f$ over the nodal set.  Here, we   normalize the delta-function on the
nodal set by the conjectured surface volume of \S \ref{AREA}.  More precisely:
\bigskip

\noindent{\bf Problem} Find the weak* limits of the family of measures
$\{\frac{1}{\lambda_j} d S_{\lambda_j}\}$.
\medskip

Note that in the $C^{\infty}$ case we do not even know if this family has uniformly bounded mass. 
The high-frequency limit is the
semi-classical limit and generally signals increasing complexity
in the `topography' of eigenfunctions.

Heuristics from quantum  chaos suggests that eigenfunctions of quantum
chaotic systems should  behave like random waves. The random wave model
is defined and studied in  \cite{Z4} (see \S \ref{RWONB}) , and it is proved (see Theorem \ref{ZDSM})  that if one picks a random sequence
$\{\psi_{\lambda_j}\}$ of random waves of increasing frequency,  then almost surely
\begin{equation} \label{RANDOM} \frac{1}{\lambda_j} \int_{\hcal_{\psi_{\lambda_j}}} f
dS_{\lambda_j} \to \frac{1}{Vol(M)} \int_M f dV_g, \end{equation}
i.e. their nodal sets become equidistributed with respect to the
volume form on $M$. Hence the heuristic principle leads to the
conjecture that nodal sets of  eigenfunctions of quantum chaotic
systems should become equidistributed according to the volume
form.

The conjecture for eigenfunctions (rather than random waves)  is far beyond any current techniques and serves
mainly as inspiration for studies of equidistribution of nodal
sets. 

A yet more speculative conjecture in quantum chaosis that the
nodal sets should tend to $CLE_6$ curves in critical percolation. CLE refers to 
conformal loop ensembles, which are closed curves  related to $SLE$ curves. 
As above, this problem is motivated by a comparision to random waves, but for
these the problem is also completely open.   In  \S \ref{PERC} we  review the  heuristic principles
which started in condensed matter physics \cite{KH,KHS,Isi, IsiK, Wei} before migrating
to quantum chaos 
\cite{BS,BS2, FGS,BGS,SS,EGJS}.  It is dubious that such speculative conjectures can
be studied rigorously in the forseeable future, but we include them to expose the reader
to the
 questions   that are relevant to
physicists. 

\subsection{$L^1$ norms and nodal sets}

Besides nodal sets it is of much current interest to study $L^p$ norms of eigenfunctions globally
on $(M, g)$ and also of their restrictions to submanifolds. In fact, recent results show that nodal
sets and $L^p$ norms are related. For instance, in \S \ref{LB} we will use  the identity
\begin{equation} \label{ID}   ||\phi_{\lambda}||_{L^1} = \frac{1}{\lambda^2}   \int_{Z_{\phi_{\lambda}}}   |\nabla
\phi_{\lambda}| dS  \end{equation} relating the $L^1$ norm of $\phi_{\lambda}$ to a weighted integral over $Z_{\phi_{\lambda}}$ to obtain lower bounds on $\hcal^{n-1}(Z_{\phi_{\lambda}}). $ See \eqref{1}. 

Obtaining lower bounds on $L^1$ norms of eigenfunctions is closely related to finding upper bounds on $L^4$
norms. The current bounds are nowhere near sharp enough to improve nodal set bounds.

\subsection{Critical points and values}

A closely related   problem in the `topography' of Laplace
eigenfunctions $\phi_{\lambda}$ is to determine the asymptotic
distribution of their critical points
$$C(\phi_{\lambda}) = \{x: \nabla \phi_{\lambda}(x) = 0\}. $$ This
problem is analogous to that of measuring the  hypersurface area
$\hcal^{n-1}(Z_{\lambda})$ of the nodal (zero) set of
$\phi_{\lambda}$, but it is yet more complicated due to the
instability of the critical point set as the metric varies.  For a
generic metric, all eigenfunctions are Morse functions and the
critical point set is discrete. One may ask to count the number of
critical points asymptotically as $\lambda \to \infty$. But there
exist metrics (such as the flat metric on the torus, or the round
metric on the sphere) for which the eigenfunctions have critical
manifolds rather than points.  To get around this obstruction, we
change the problem from counting critical points to counting
critical values
$$CV(\phi_{\lambda}) = \{\phi_{\lambda}(x): \nabla \phi_{\lambda}(x) = 0\}. $$
Since a real analytic function on a compact real
analytic manifold has only finitely many critical values, 
eigenfunctions of real analytic Riemannian manifolds $(M, g)$ have
only finitely many critical values and we can ask to count them.  Moreover for generic real
analytic metrics, all eigenfunctions are Morse functions and there
exists precisely one critical point for each critical value. Thus,
in the generic situation, counting critical values is equivalent
to counting critical points. To our knowledge, there are no results on this problem,
although it is possible to bound the $\hcal^{n-1}$-measure of $C(\phi_{\lambda})$ (see Theorem
 \cite{Ba}). However $\hcal^{n-1} (C(\phi_{\lambda})) = 0$ in the generic case
and in special cases where it is not zero the method is almost identical to bounds on the nodal set. 
Thus, such results bypass all of the difficulties in counting critical values. We will present one
new (unpublished)  result which generalizes eqref{ID} to critical points. But the resulting identity
is much more complicated than for zeros.

Singular points are critical points which occur on the nodal sets.
We  recall (see \cite{H, HHL,  HHON})  that the the singular
set
$$\Sigma(\phi_{\lambda})= \{x \in Z_{\phi_{\lambda}}: \nabla
\phi_{\lambda}(x) = 0\} $$ satisfies $\hcal^{n-2}(\Sigma
(\phi_{\lambda})) < \infty$. Thus, outside of a codimension one
subset, $Z_{\phi_{\lambda}}$ is a smooth manifold, and the
Riemannian surface measure  $d S = \iota_{\frac{\nabla
\phi_{\lambda}}{|\nabla \phi_{\lambda}|}} dV_g$ on
$Z_{\phi_{\lambda}}$ is well-defined. We refer to \cite{HHON,H,HHL,HS} for background.

\subsection{Inradius}

It is known that in dimension two, the minimal possible area of a
nodal domain of a Euclidean eigenfunction  is $\pi (
\frac{j_1}{\lambda})^2$. This follows from the two-dimensional
Faber-Krahn inequality,
$$\lambda_k(\Omega) \mbox{Area}(D) = \lambda_1(D) \mbox{Area}(D)  =  \geq \pi j_1^2$$
where $D$ is a nodal domain in $\Omega$. In higher dimensions, the
Faber-Krahn inequality shows that on any Riemannian manifold the
volume of any nodal domain is $\geq C \lambda^{-n}$ \cite{EK}.

Another size measure of a nodal domain is its inradius
$r_{\lambda}$, i.e. the radius of the largest ball contained
inside the nodal domain. As can be seen from computer graphics
(see e.g. \cite{HEJ}), there are a variety of `types' of nodal
components.
 In \cite{Man3}, Mangoubi proves that
\begin{equation} \frac{C_1}{\lambda} \geq r_{\lambda} \geq
\frac{C_2}{\lambda^{\half k(n)} (\log \lambda)^{2n-4}},
\end{equation} where $k(n) =  n^2 - 15n/8 + 1/4$; note that
eigenvalues in \cite{Man} are denoted $\lambda$ while here we
denote them by $\lambda^2$. In dimension $2$, it  is known (loc.cit.) that
\begin{equation} \frac{C_1}{\lambda} \geq r_{\lambda} \geq
\frac{C_2}{\lambda}. \end{equation}

\subsection{\label{DECOMP} Decompositions of $M$ with respect to $\phi_{\lambda}$}

There are two natural decompositions (partitions) of $M$ associated to an eigenfunction (or any smooth function). 
\medskip

\noindent{\bf (i)} Nodal domain decomposition.
\medskip

First is the  decomposition of $M$ into nodal domains of $\phi_{\lambda}$.  As in \cite{PS} we denote the
collection of nodal domains by $\acal(\phi_{\lambda})$ and denote a nodal domain by $A$. Thus,
$$M \backslash Z_{\phi_{\lambda}}= \bigcup_{A \in \acal(\phi_{\lambda})} A. $$ When $0$ is a regular
value of $\phi_{\lambda}$ the level sets are smooth hypersurfaces and one can ask how many components
of $Z_{\phi_{\lambda}}$ occur, how many components of the complement, the topological types of components
or the combinatorics of the set of domains.  When $0$ is a singular value, the nodal set is a singular
hypersurface and  can be connected but
one may ask similar questions taking multiplicities of the singular points into account. 

To be precise, let 
$$\mu(\phi_{\lambda})= \# \acal (\phi_{\lambda}), \;\;\;  \nu(\phi_{\lambda}) = \# \;\mbox{components of}\;
Z(\phi_{\lambda}). $$
The best-known problem is to estimate $\mu(\phi_{\lambda})$. According to the Courant nodal domain
theorem, $\mu(\phi_{\lambda_n}) \leq n$. In the case of spherical harmonics, where many orthonormal bases
are possible, it is better to estimate the number in terms of the eigenvalue, and the estimate has
the form $\mu(\phi_{\lambda}) \leq C(g) \lambda^{m}$ where $m = \dim M$ and $C(g) > 0$ is a constant
depending on $g$.  In dimension 2, Pleijel used the Faber-Krahn theorem to  improve the bound to
$$\limsup_{\lambda \to \infty} \frac{\mu(\phi_{\lambda})}{\lambda^2} \leq \frac{4}{j_0^2} < 0. 69 $$
where $j_0$ is the smallest zero of the $J_0$ Bessel function. 

A wide variety of behavior is exhibited by spherical harmonics of degree $N$. We review the definitions below. 
The even degree harmonics  are equivalent to  real projective plane curves of degree $ N$.  But each point of
$\R \PP^2$ corresponds to a pair of points of $S^2$ and at most one component of the nodal set is invariant
under the anti-podal map. For other components, the anti-podal map takes a component to a disjoint component.
Thus there are essentially twice the number of components in the nodal set as components of the associated
plane curve. 

 As discussed in \cite{Ley},  one has;

\begin{itemize}

\item Harnack's inequality: the number of components of any irreducible  real projective plane curve is bounded by 
$g + 1$ where $g$ is the genus of the curve.

\item If $p$ is a real projectove plane curve of degree $N$ then its genus is given by Noether's formula
$$g = \frac{(N-1)(N-2)}{2} - \sum_{\mbox{singular points x} }  \frac{\mbox{ord}_p(x) (\mbox{ord}_p(x) -1)}{2} $$
where $\mbox{ord}_p(x)$ is the order of vanishing of $\phi_{\lambda}$ at $x$.  Thus, the number of components
is $\leq  \frac{(N-1)(N-2)}{2}  + 1$ for a non-singular irreducible plane curve of degree $N$.

\end{itemize}

Curves which achieve the maximum are called $M$-curves.  Also famous are Harnack curves, which  are $M$ curves for which there
exist three distinct lines $\ell_j$  of $\R \PP^2$ and three distinct arcs $a_j$ of the curve on one component
so that $\# a_j \cap \ell_j = N$.  It follows from Pleijel's bound that nodal sets of spherical harmonics cannot
be maximal for large $N$, since half of the Pleijel bound is roughly $.35 N^2$ which is below the
threshold $ .5 N^2 + O(N)$ for maximal curves.

Associated to the collection of nodal domains is its  incidence  graph $\Gamma_{\lambda}$, which has 
 one vertex for each nodal domain, and one edge linking each pair of nodal domains with a common 
boundary component. Here we assume that $0$ is a regular value of $\phi_{\lambda}$ so that the nodal
set is a union of embedded submanifolds.  The Euler characteristic of the graph is the difference beween the number
of nodal domains and nodal components. In the non-singular case, one  can convert the nodal decomposition into
a cell decomposition by attaching a one cell between two adjacent components, and then
one has  $\mu(\phi_{\lambda}) = \nu(\phi_{\lambda}) + 1$ (see Lemma 8 of \cite{Ley}). 

The possible topological types of arrangements of nodal components of spherical harmonics is studied
in \cite{EJN}.  They prove that for any $m \leq N$  with $N-m$ even and for every set of $m$ disjoint closed
curves whose union is invariant with respect to the antipodal map, there exists an eigenfunction whose nodal
set has the topological type of the union of curves. Note that these spherical harmonics have relatively few nodal
domains compared to the Pleijel bound. It is proved in \cite{NS} that random spherical harmonics have
$a N^2$ nodal components for some (undetermined) $a > 0$.
\medskip

\noindent{\bf Morse-Smale decomposition}

For generic metrics, all eigenfunctions are Morse functions \cite{U}.  Suppose that $f: M \to \R$ is a Morse function. For each critical point $p$
let $W^sp^s$ (the  stable or descending cell through p)  denote the union of the gradient flow lines which have $p$ as their initial point, i.e. their $\alpha$-limit
point. Then $W_p$ is a cell of dimension $\lambda_p = $ number of negative eigenvalues of $H_p f$. By the Morse-Smale
decomposition we mean the decomposition
$$M = \bigcup_{p: df(p) = 0} W_p^s$$  It is  not a good cell
decomposition in general.  If we change $f$ to $-f$ we
 get the decomposition into ascending (unstable) cells
$M = \bigcup_{p: df(p) = 0} W_p^u. $
If the intersections $W_p^s \cap W_qu$ are always transversal then $\nabla f$ is said to be  transversal. 
In this case $\dim (W_p^s \cap W_q^u) = \lambda_p - \lambda_1 + 1$ and the number of gradient curves joining
two critical points whose Morse index differs by 1 is finite. 

We are mainly interested  in the stable  cells of maximum dimension, i.e. basins of attraction of the 
gradient flow to each local minimum. We then have the partition
\begin{equation} \label{MSD} M = \bigcup_{p \; \mbox{a local min} } W_p^s. \end{equation}  This decomposition is somtimes used in condensed matter physics
(see e.g. \cite{Wei}) 
and in computational shape analysis \cite{Reu}.  In dimension two, the surface is
partitioned into `polygons' defined by the basins of attraction of
the local minima of $\phi $. The boundaries of these polygons are
gradient lines of $\phi$ which emanate from saddle points. The
vertices occur at local maxima.

\begin{center} 
\includegraphics[scale=0.6]{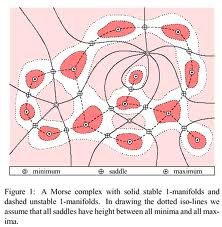}  

 \end{center}

An eigenfunction is a Neumann eigenfunction in each basin since the boundary is formed by integral curves
of $\nabla \phi_{\lambda}$.  Possibly it is `often'  the first non-constant Neumann eigenfunction (analogously to
$\phi_{\lambda}$ being the lowest Dirichlet eigenfunction in each nodal domain), but this does not seem
obvious. Hence it is not clear how to relate the global eigenvalue $\lambda^2$ to the Neumann eigenvalues
of the basins, which would be useful in understanding the areas or diameters of these
domains.  Note that $$\int_{W_p^s} \phi_j d V = \int_{\partial W_p^s} \nabla \phi_{\lambda} \cdot \nu d S = 0,$$ where $\nu$ is the unit normal to $\partial W_p^s$,
since $\nabla \phi_{\lambda} $ is tangent to the boundary. In particular,
the intersection $Z_{\phi_{\lambda}} \cap W_p^s$ is non-empty and is a connected hypersurface which separates
$W_p^s$ into two components on which $\phi_{\lambda}$ has a fixed sign. 
To our knowledge, there do not exist rigorous  results bounding  the number of local minima from above or below, i.e.
there is no analogue of the Courant upper bound for the number of local minima basins.  It is possible to obtain statstical
results on the asymptotic expected number of local minima, say for random spherical
harmonics of degree $N$.  The methods of \cite{DSZ}  adapt to this problem if one replaces holomorphic
Szeg\"o kernels by spectral projections  (see also \cite{Nic}.)  Thus, in a statistical sense it is much simpler to count the number
of ``Neumann domains" or Morse-Smale basins than to count nodal domains as in \cite{NS}.

\section{Examples}

Before proceeding to rigorous results, we go over a number of explicitly solvable
examples. Almost by definition, they are highly non-generic and in
fact represent the eigenfunctions of quantum integrable systems. Aside from being
explicitly solvable, the eigenfunctions of this section are extremals for a number of problems.

\subsection{\label{Flat tori section} Flat tori}

The basic real valued   eigenfunctions are  $\phi_{k}(x) = \sin \langle
k, x \rangle$ or $\cos \langle k, x \rangle$  ($k \in \Z^n$) on the flat torus ${\bf T} =
\R^n/\Z^n$. The zero set consists of the hyperplanes $\langle k, x
\rangle = 0 $ mod $2 \pi$ or in other words $\langle x, \frac{k}{|k|} \rangle \in \frac{1}{2 \pi |k| } \Z$. 
Thus the normalized  delta function $\frac{1}{|k|} d S|_{Z_{\phi_k}}$  tends to uniform distribution along
rays in the lattice $\Z^n$.  The lattice arises as the joint spectrum of the commuting operators
$D_j = \frac{ \partial}{i \partial x_j}$ and is a feature of quantum integrable systems.

\begin{center} 
\includegraphics[scale=0.3]{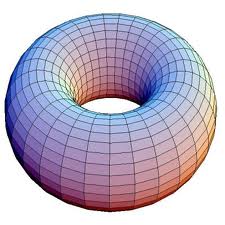}  

 \end{center}

The critical point equation for $\cos \langle k, x \rangle$ is $k \sin \langle k, x \rangle = 0$ and
is thus the same as the nodal equation. In particular, the critical point sets are hypersurfaces in this
case. There is just one critical value $= 1$. 

Instead of the square torus we could consider $\R^n/L$ where $L \subset \R^n$ is a lattice of full rank. 
Then the joint spectrum becomes the dual lattice $L^*$ and the eigenfunctions are $\cos \langle  k, x \rangle,
\sin \langle k, x \rangle$ with $k \in L^*$. 

The real eigenspace $\hcal_{\lambda} = \R-\mbox{span}\{\sin \langle k, x \rangle, \cos \langle k, x \rangle:
|k| =  \lambda\}$ is of multiplicity 2 for generic $L$ but has unbounded multiplicity in the case of $L = \Z^n$
and other rational lattices. In that case, one may take linear combinations of the basic eigenfunctions and
study their nodal and critcal point sets. For background,  some recent results and further references we refer  to \cite{BZ}.

\medskip

\subsection{\label{S2} Spherical harmonics on $S^2$}
\medskip

The spectral decomposition for the Laplacian is the orthogonal sum
of the spaces of spherical harmonics of degree $N$,
\begin{equation} L^2(S^2) = \bigoplus_{N=0}^{\infty} V_N,\;\;\;
\Delta |_{V_N} = \lambda_N Id.
\end{equation}
The eigenvalues   are   given by $\lambda_N^{S^2} = N (N + 1)$ and
the multiplicities are given by $m_N = 2 N + 1 $. A standard basis
is given by the (complex valued) spherical harmonics $Y^N_m$ which
transform by $e^{i m \theta}$ under rotations preserving the
poles.

The $Y^N_m$ are complex valued, so we study the nodal sets of their real and imaginary parts. 
They are separable, i.e. factor as $C_{N, m} P^N_m(r) \sin (m \theta)$ (resp. $\cos (m \theta)$
where $P^N_m$ is an associated Legendre function. 
Thus the nodal sets of these special eigenfunctions form a checkerboard pattern that can be
explicitly determined from the known behavior of zeros of associated Legendre functions.  See the
first image in the illustration below.

\begin{center} 
\includegraphics[scale=0.6]{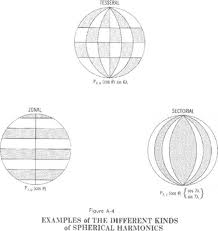} 
 \end{center}

Among the basic spherical harmonics, there are two special ones: the zonal spherical harmonics (i.e. the
rotationally invariant harmonics) and the highest weight spherical harmonics.  Their nodal sets and intensity
plots are graphed in the bottom two images, respectively. 

Since the  zonal spherical harmonics $Y^N_0 $ on $S^2$ are real-valued and rotationally invariant, their
zero sets consist of a union of circles, i.e. orbits of the $S^1$ rotation action around the third axis. 
 It is well known that $Y^N_0(r) =
\sqrt{\frac{(2 N + 1)}{2 \pi}} P_N(\cos r)$, where $P_N$ is the
$N$th Legendre function and the normalizing constant is chosen so
that $||Y^N_0||_{L^2(S^2)} = 1$, i.e.  $4 \pi \int_0^{\pi/2}
|P_N(\cos r)|^2 dv(r) = 1, $ where $dv(r) = \sin r dr$ is the
polar part of the area form. Thus the circles occur at values of $r$ so that $P_N(\cos r) =0$. All zeros
of $P_N(x) $ are real and it has $N$ zeros in $[-1,1]$. It is classical that the zeros $r_1, \dots, r_N$
of $P_N(\cos r)$ in $(0, \pi)$  become uniformly distributed with respect to $dr$ \cite{Sz}.  It is 
also known that $P_N$ has $N-1$ distinct critical points \cite{C,Sz2} and so the critical points of $Y^N_0$
is a union of $N-1$ lattitude circles.


We now consider  real or imaginary parts of  highest
weight spherical harmonics $Y^N_N$. Up to a scalar multiple, 
$Y_N(x_1, x_2, x_2) = (x_1 + i x_2)^N$ as a harmonic polynomial on $\R^3$. It is an example of a  Gaussian beams along
a closed geodesic $\gamma$ (such as exist on equators of convex
surfaces of revolution). See  \cite{R} for  background on Gaussian beams on Riemannian
manifolds.  

The real and imaginary parts are of the form $P_N^N(\cos r) \cos N \theta, P_N^N(\cos r) \sin N \theta$ where
$P_N^N(x)$
is a constant multiple of $(1 - x^2)^{N/2}$ so $P_N^N(\cos r) = (\sin r)^N. $
The factors $\sin N \theta, \cos N \theta$ have $N$ zeros on $(0, 2 \pi)$. The Legendre funtions satisfy the recursion relation $P_{\ell + 1}^{\ell + 1}
= -(2 \ell + 1) \sqrt{1 - x^2} P_{\ell}^{\ell}(x)$ with $P_0^0 = 1$ and therefore have no real zeros away
from the poles. Thus, the nodal set consists of
$N$ circles of longitude with equally spaced intersections with the equator.

The critical points are solutions of  the pair of equations $\frac{d}{dr} P_N^N(r) \cos N \theta = 0, P_N^N \sin N \theta  = 0$. 
Since $P_N^N$ has no zeros away from the poles, the second equation forces the zeros to occur at zeros of $\sin N \theta$. 
But then $\cos N \theta \not = 0$ so the zeros must occur at the zeros of  $\frac{d}{dr} P_N^N(r) $. The  critical points only
occur  when
$\sin r = 0$ or $\cos r = 0$ on $(0, \pi)$. 
 There  are critical
points  at the poles  where $Y_N^N$ vanishes to order $N$ and there is a local maximum  at the value $r = \frac{\pi}{2}$ of the equator. Thus, $\Re Y_N^N$ has $ N$ isolated critical points on the equator and multiple critical points at the poles.

We note that $|\Re Y^N_N|^2$ is a Gaussian bump with peak along the equator in the radial direction. Its radial
Gaussian decay implies that it extremely small outside a $N^{\half}$ tube around the equator. The complement
of this tube is known in physics as the classically forbidden region. We see that the nodal set stretches a long
distance into the classically forbidden region. This creates problems for nodal estimates since  exponentially
small values  (in terms of the eigenvalue)  are  hard to distinguish from zeros. On the other hand, it has only
two (highly multiple) critical points away from the equator. 

\subsection{Random spherical harmonics and chaotic eigenfunctions}

The examples above exhibit quite disparate behavior but all are eigenfunctions of quantum integrable systems. 
We do not review the general results in this case but plan to treat this case in an article in preparation \cite{Z9}. 

We now contrast the nodal set behavior with that of random spherical harmonics (left)
and a chaotic billiard domain  (the graphics  are due to E. J. Heller).

\begin{center} 
\includegraphics[scale=0.15]{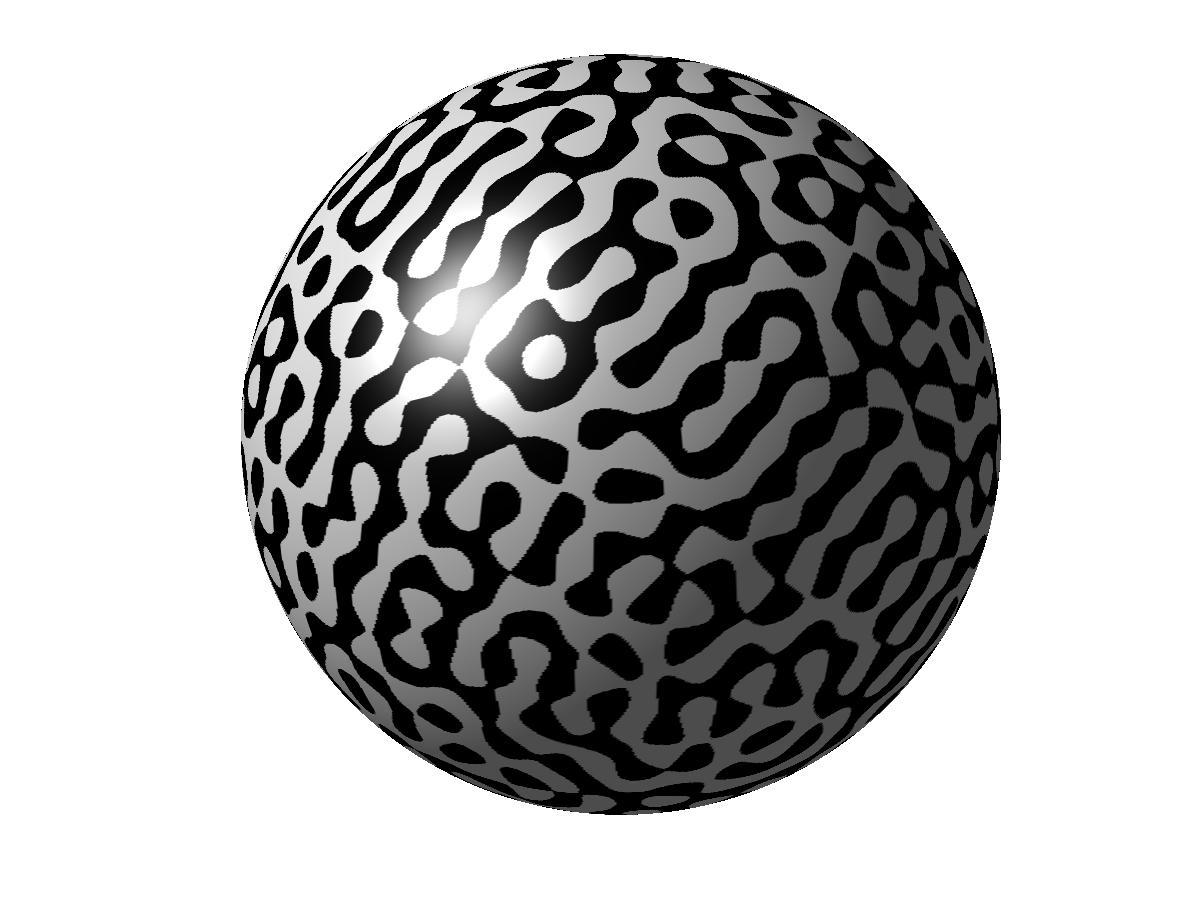}  
\includegraphics[scale=0.25]{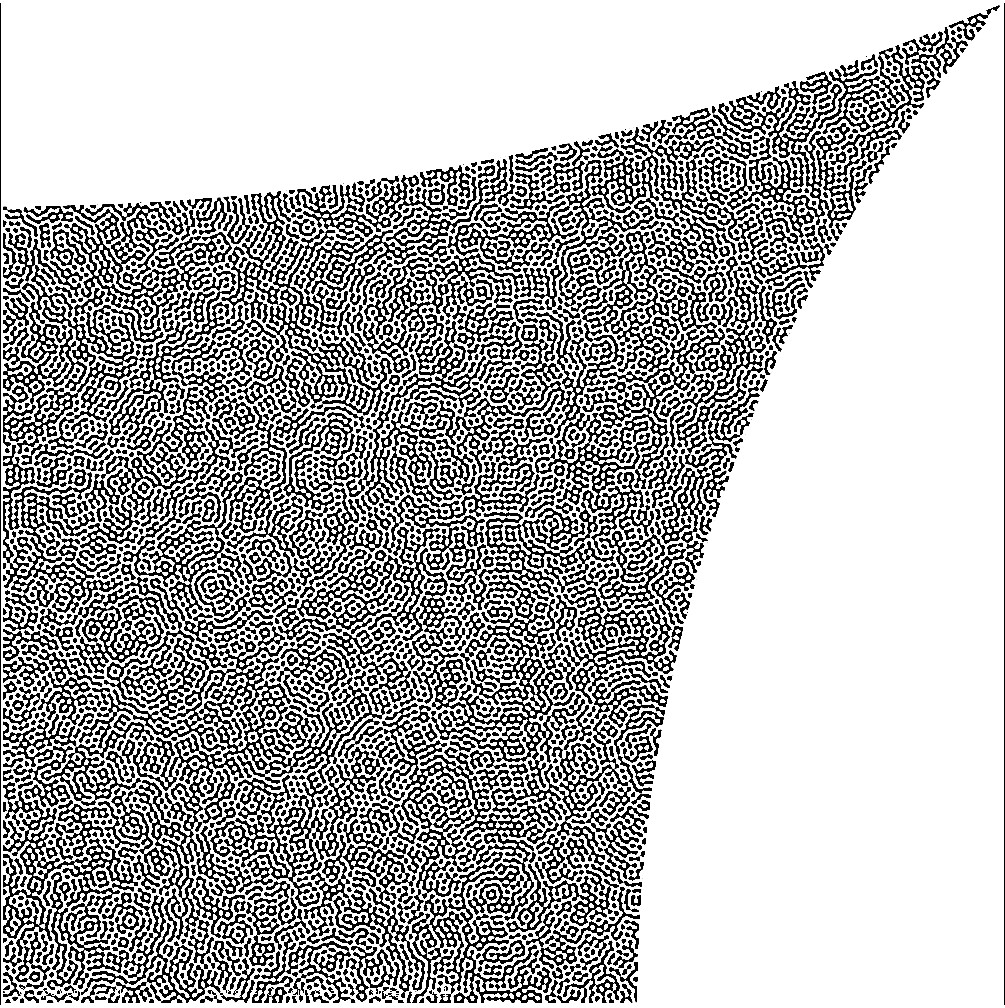}
 \end{center}

\section{\label{LB} Lower bounds on hypersurface areas of nodal sets and level sets in the $C^{\infty}$ case}

In this section we review some recent lower bounds on $\hcal^{n-1}(Z_{\phi_{\lambda}})$ from
\cite{CM,SoZ,HS,HW}.   Before we begin, we recall the co-area formula:  Let $f: M \to \R$ be Lipschitz. Then
for any continuous function $u$ on $M$,
$$\int_M u(x)  d V = \int_{\R}  (\int_{f^{-1}(y)} u \frac{dV}{df} )\; d y. $$
Equivalently, 
$$\int_M u(x)  ||\nabla f|| d V = \int_{\R}  (\int_{f^{-1}(y)} u  d \hcal^{n-1} )d y. $$ We refer to
$\frac{dV}{df}$ as the ``Leray form" on the level set $\{f = y\}$.  Unlike the Riemannian
surface measure $dS = d \hcal^{n-1}$ it depends on the choice of defining function $f$. 
The surface measures are related by  $d \hcal^{n-1} = |\nabla f| \frac{dV}{df}.$
For background, see  Theorem 1.1 of \cite{HL}.

\begin{theo}\label{theorem}  
As above assume that $||\phi_{\lambda}|| = 1$. Then
\begin{equation}\label{main}
\lambda \left(\, \int_M |\phi_\lambda|\, dV_g\, \right)^2 \le C  \hcal^{n-1} (Z_{\phi_{\lambda}}), \quad \lambda \ge1,
\end{equation}
for some uniform constant $C$.  Consequently,
\begin{equation}\label{5}
\lambda^{\frac{3-n}2}\lesssim   \hcal^{n-1} (Z_{\phi_{\lambda}}), \quad \lambda\ge1.
\end{equation}
\end{theo}

Inequality \eqref{5} follows from \eqref{main} and the lower bounds in \cite{SoZ}
\begin{equation}\label{3}
\lambda^{\frac{1-n}4} \lesssim \int_M |\phi_\lambda| \, dV_g.
\end{equation}
The lower bound \eqref{5} was first proved  by Colding and Minicozzi~\cite{CM}.  A slightly weaker
result was proved in \cite{SoZ} by a different method inspired by the article of R. T. Dong \cite{Dong} which
was successively improved in \cite{HW,HS}  to the same bound as \cite{CM}.

The $L^1$-lower bounds in \eqref{3}  were proved in \cite{SoZ} using 
H\"older's inequality and the 
$L^p$ eigenfunction estimates of Sogge \cite{Sog}
for the range where $2<p\le \frac{2(n+1)}{n-1}$.  As will be seen below, the estimate is sharp for
Gaussian beams such as highest weight spherical harmonics. The random wave model would predict that
$||\phi_{\lambda}||_{L^1} \geq C(M, g) > 0$ and that would imply  Yau's conjectured lower bound in
the `chaotic case'. This provides ample motivation to study $L^1$ norms of eigenfunctions on manifolds
with ergodic geodesic flow. 

The proof of Theorem~\ref{theorem} is based on an identity  from \cite{SoZ} (inspired by an identity
in \cite{Dong}): 
\begin{equation}\label{1}
\int_M |\phi_{\lambda}| \, (\Delta_g+\lambda^2)f \, dV_g = 2 \int_{Z_{\phi_{\lambda}} }|\nabla_g \phi_\lambda| \, f
\, dS,
\end{equation}
We recall that $dS$ is the Riemannian surface measure on $Z_{\phi_{\lambda}}$. In \cite{SoZ}, we substituted
$f = 1$ and used \eqref{3} to obtain a power lower bound. 
This was improved in \cite{HW,HS} by  putting  $f\equiv1$ but applying  Schwarz's inequality to get
\begin{equation}\label{2}
\lambda^2 \int_M |\phi_\lambda| \, dV_g \le 2 (\hcal^{n-1}(Z_{\phi_{\lambda}}))^{1/2} \, \left(\, 
\int_{Z_{\phi_{\lambda}}}|\nabla_g \phi_\lambda|^2 \, dS\, \right)^{1/2}.
\end{equation}
The final important step is to show that
\begin{equation}\label{4}
\int_{Z_{\phi_{\lambda}}}|\nabla_g \phi_\lambda|^2 \, dS_ \lesssim \lambda^{3}
\end{equation}
by choosing 
\begin{equation}\label{6}
f=\big(\, 1+\lambda^2 \phi_\lambda^2 + |\nabla_g \phi_\lambda|^2_g \, \bigr)^{\frac12}.
\end{equation}
From \eqref{1}  it follows that
\begin{equation*}
2\int_{Z_{\phi_{\lambda}}}|\nabla_g \phi_\lambda|^2_g dS
\le \int_M |\phi_\lambda| \, (\Delta_g+\lambda^2)\big(\,1+ \lambda \phi_\lambda^2 + |\nabla_g \phi_\lambda|^2 \, \bigr)^{\frac12} \, dV_g.
\end{equation*}
By the  $L^2$-Sobolev bounds
$ \|  \phi_\lambda\|_{H^s(M)}=O(\lambda^s)$ it follows that the ``second term" of the right side satisfies
$$\lambda^2 \int_M |\phi_\lambda| \, \big(\,1+ \lambda^2 \phi_\lambda^2 + |\nabla_g \phi_\lambda|^2_g \, \bigr)^{\frac12}\, dV_g = O(\lambda^{3}),
$$
and thus to prove \eqref{4}, it suffices to show that the ``first terms" satisfies
\begin{equation}\label{8}
\int_M |\phi_\lambda| \, \Delta_g\big(\, 1+\lambda^2 \phi_\lambda^2 + |\nabla_g \phi_\lambda|^2_g \, \bigr)^{\frac12}
\, dV_g = O(\lambda^3).
\end{equation}
We refer to \cite{HS} for further details on this bound. A simpler approach was suggested by 
 W. Minicozzi, who  pointed out that  \eqref{4} also follows from the identity
\begin{equation}\label{dividentity}
2\int_{Z_\lambda}|\nabla_g e_\lambda|^2 \, dS_g = 
-\int_M \sgn (e_\lambda) \, \text{div}_g\big(\, |\nabla_g e_\lambda| \, \nabla_g e_\lambda\, \bigr) \, dV_g.
\end{equation}
This approach is used in \cite{Ar} to generalize the nodal bounds to Dirichlet and Neumann
eigenfunctions of  bounded domains.  In the next setction we explain how to obtain more general identities.

There are several ways to prove the identity in \eqref{1}. One way to see it is that
 $d\mu_{\lambda}: = ( \Delta+ \lambda^2) |\phi_{\lambda}| dV = 0$
away from $\{\phi_{\lambda} = 0\}$. Hence   this
distribution is a positive measure supported on $Z_{\phi_{\lambda}}.$ To determine
the coefficient of the surface measure $dS$ we  calculate the limit as $\delta \to 0$ of the integral
$$\int_M f (\Delta + \lambda^2) |\phi_{\lambda}| dV =
\int_{|\phi_{\lambda}| \leq \delta} f (\Delta + \lambda^2)
|\phi_{\lambda}|  dV. $$ Here $f \in C^2(M)$ and  with no loss of generality we may assume that $\delta$ is a regular value of
$\phi_{\lambda}$ (by Sard's theorem).   By  the Gauss-Green theorem,
$$ \int_{|\phi_{\lambda}| \leq \delta} f (\Delta + \lambda^2)
|\phi_{\lambda}|  dV - \int_{|\phi_{\lambda}| \leq \delta}
|\phi_{\lambda}|   (\Delta + \lambda^2) f  dV =
\int_{|\phi_{\lambda}| = \delta} (f \partial_{\nu}
|\phi_{\lambda}| - |\phi_{\lambda}|
\partial_{\nu} f) dS. $$
Here, $\nu$ is the outer unit normal and $\partial_{\nu}$ is the
associated directional derivative. For $\delta
> 0$, we have
\begin{equation} \label{NORMAL} \nu = \frac{\nabla \phi_{\lambda}}{|\nabla \phi_{\lambda}|}
\;\; \mbox{on}\;\; \{\phi_{\lambda} = \delta\}, \;\;\; \nu = -
\frac{\nabla \phi_{\lambda}}{|\nabla \phi_{\lambda}|} \;\;
\mbox{on}\;\; \{\phi_{\lambda} = - \delta\}. \end{equation}

 Letting $\delta \to 0$ (through the sequence of regular values)  we get
$$ \int_M f (\Delta + \lambda^2) |\phi_{\lambda}| dV = \lim_{\delta \to 0} \int_{|\phi_{\lambda}| \leq \delta} f (\Delta + \lambda^2)
|\phi_{\lambda}|  dV = \lim_{\delta \to 0}
 \int_{|\phi_{\lambda}|
= \delta}   f
\partial_{\nu} |\phi_{\lambda}|dS. $$
Since $|\phi_{\lambda}| = \pm \phi_{\lambda}$ on $\{\phi_{\lambda}
= \pm \delta\}$ and by (\ref{NORMAL}), we see that
$$ \begin{array}{lll} \int_M f (\Delta + \lambda^2) |\phi_{\lambda}| dV & = & \lim_{\delta \to 0}
 \int_{|\phi_{\lambda}|
= \delta}   f \frac{\nabla |\phi_{\lambda}| }{|\nabla
|\phi_{\lambda}| | } \cdot \nabla|\phi_{\lambda}|dS \\ && \\ & = &
\lim_{\delta \to 0} \sum_{\pm}
 \int_{\phi_{\lambda}
= \pm  \delta}   f |\nabla\phi_{\lambda}| dS \\ && \\& = & 2
\int_{Z_{\phi_{\lambda}}}   f |\nabla \phi_{\lambda}| dS.
\end{array}
$$
The Gauss-Green formula and limit are justified by the fact that the singular set $\Sigma_{\phi_{\lambda}}$
has codimension two. We refer to \cite{SoZ} for further details.

The $L^1$ lower bound of  \eqref{3} follows from  eigenfunction estimates in
\cite{Sog}, which say that
$$\|\phi_\lambda\|_{L^p}\le C \lambda^{\frac{(n-1)(p-2)}{4p}},
\quad 2<p\le \tfrac{2(n+1)}{n-1}.
$$
If we pick such a $2<p<\tfrac{2(n+1)}{n-1}$, then by H\"older's
inequality, we have
$$1=\|\phi_\lambda\|_{L^2}^{1/\theta}\le \|\phi_\lambda\|_{L^1} \, \|\phi_\lambda\|_{L^p}^{\frac1\theta-1}\le  \|\phi_\lambda\|_{L^1}\bigl(\, C\lambda^{\frac{(n-1)(p-2)}{4p}}\, \bigr)^{\frac1\theta-1},
\quad
\theta=\tfrac{p}{p-1}(\tfrac12-\tfrac1p)=\tfrac{(p-2)}{2(p-1)},$$
which  implies $\|\phi_\lambda\|_{L^1}\ge
c\lambda^{-\frac{n-1}4}$, since $(1-\tfrac1\theta)
\tfrac{(n-1)(p-2)}{4p}=\tfrac{n-1}4$.

\begin{rem} One can also integrate the identity \eqref{1} over a basin of attraction of a local minimum
(or maximum)  \eqref{MSD}, since
the boundary term vanishes. Thus we get an identity between the $L^1$ norm of $\phi_{\lambda}$ on
each basin and the $|\nabla \phi_{\lambda}| dS$-measure of the nodal line inside the basin.   \end{rem}

\subsection{\label{MG} More general  identities}

For any function $\chi$, we have 
$$\Delta \chi(\phi) = \chi''(\phi) |\nabla \phi|^2 - \lambda^2 \chi'(\phi) \phi. $$ We then
take $\chi$ to be the meromorphic family of homogeneous distribution $x_+^s$. 
We recall that for  $\Re a > -1$, 
$$x_+^a := \left\{ \begin{array}{ll} x^a, & x \geq 0 \\ & \\
0, & x < 0. \end{array} \right. $$
The family extends to  $a \in \C$ as a meromorphic family of distributions with simple poles
at $a = -1, -2, \dots, -k, \dots$  using the equation 
$\frac{d}{dx} x_+s = s x_+^{s-1}$ to extend it one unit strip at a time.  One can convert $x_+^s$ to the  holomorphic family 
$$\chi_+^{\alpha} = \frac{x_+^{\alpha}}{\Gamma(\alpha + 1)},\;\;\;\;\mbox{with}\;\; \chi_+^{-k} = \delta_0^{(k-1)}. $$

The identity we used above belongs to the family, 
\begin{equation} \label{GEN}  (\Delta + s \lambda^2) \phi_+^s = s(s - 1)  |\nabla \phi|^2 \phi_+^{s-2}. \end{equation}
Here $\phi_+^s = \phi^* x_+^s$ has poles at $s = -1, -2, \cdots$. The calculation in \eqref{1}  used $|\phi|$
but is equivalent to using \eqref{GEN} when $s = 1$. Then $\phi_+^{s-2}$ has a pole when $s = 1$ with
residue $\delta_0(\phi) = \frac{d S}{|\nabla \phi|} d S |_{Z_{\phi_{\lambda}}}$; it is  cancelled by
the factor $s - 1$ and  we obtain \eqref{1}. This calculation is formal because the pullback formulae are only
valid when $d \phi \not= 0$ when $\phi = 0$, but as above they can be justified because the singular set has
codimension 2. The right side also has a pole at $s = 0$ and we get $\Delta \phi_+^0 = - |\nabla \phi|^2 \delta'(\phi)$,
which is equivalent to the divergence identity above.  There are further poles at $s = -1, -2, \dots$ but they 
now occur on both sides of the formulae. It is possible that they have further uses.  

The question arises of how such identities are related to the Bernstein-Kashiwara theorem that for any real
analytic function $f$ one may meromorphically extend $f_+^s$ to $\C$ by constructing a family $P_s(D)$
of differential operators with analytic coefficients and a meromorphic function $b(s)$ so
that $P_s(D) f^{s+1} = b(s) f^{s}. $ In the case $f = \phi_{\lambda}$,  the operator $|\nabla \phi|^{-2} (\Delta + s \lambda^2)$
accomplishes something like this, but it does not have analytic coefficients due to poles at the critical points of
$\phi$. One wonders what $P_s(D), b(s)$ might be when $f = \phi_{\lambda}$.

 \subsection{Other level sets}

These results generalize easily to any level set
$\ncal_{\phi_{\lambda}}^c : = \{\phi_{\lambda} = c\}$.  Let $\sgn
(x) = \frac{x}{|x|}$.

\begin{prop} \label{BOUNDSc} For any $C^{\infty}$ Riemannian
manifold, and any $f \in C(M)$ we have,

\begin{equation} \label{DONGTYPEc}  \int_M f (\Delta + \lambda^2)\; |\phi_{\lambda} - c| \;dV
+ \lambda^2 c \int f \mbox{\sgn} (\phi_{\lambda} - c) dV = 2\;
\int_{\ncal^c_{\phi_{\lambda}}}   f |\nabla \phi_{\lambda}| dS.
\end{equation}

\end{prop}

This identity has similar implications for
$\hcal^{n-1}(\ncal^c_{\phi_{\lambda}})$ and for the
equidistribution of level sets.  Note that if $c
> \sup |\phi_{\lambda}(x)|$ then indeed both sides are zero.

\begin{cor}\label{cintro} For
$c \in {\mathbb R}$
$$\lambda^2\int_{\phi_\lambda\ge c}\phi_\lambda dV
= \int_{\ncal^c_{\phi_{\lambda}}}   |\nabla \phi_{\lambda}| dS
\leq \lambda^2 Vol(M)^{1/2}. $$ Consequently, if $c>0$
$$\hcal^{n-1}(\ncal^c_{\phi_{\lambda}})
+\hcal^{n-1}(\ncal^{-c}_{\phi_{\lambda}}) \geq \;C_g \; \lambda^{2
- \frac{n + 1}{2}} \int_{|\phi_{\lambda}| \geq c} |\phi_{\lambda}|
dV.
$$

\end{cor}

The Corollary  follows by integrating $\Delta$ by parts,
 and by using the identity,
 \begin{equation} \label{c} \begin{array}{lll} \int_M |\phi_{\lambda} - c| + c \;\sgn(\phi_{\lambda} - c)
 \;dV & = & \int_{\phi_{\lambda} > c} \phi_{\lambda} dV -
 \int_{\phi_{\lambda} < c} \phi_{\lambda} dV \\ && \\ &=& 2\int_{\phi_{\lambda} > c} \phi_{\lambda} dV ,
 \end{array} \end{equation}
 since $0=\int_M \phi_\lambda dV=\int_{\phi_\lambda>c}\phi_\lambda dV
 +\int_{\phi_\lambda<c}\phi_\lambda dV$.

\subsection{\label{EXAMPLES} Examples}

The lower bound of Theorem \ref{theorem} is far from the lower bound conjectured by Yau, which
by Theorem \ref{DF} is correct at least in the real analytic case.  In this
section we go over the model examples to understand why the methds are not always getting sharp results. 

\subsection{\label{Flat tori section2} Flat tori}

We have,  $|\nabla \sin \langle k, x
\rangle|^2 = \cos^2 \langle k, x \rangle |k|^2$. Since $\cos
\langle k, x \rangle = 1$ when $\sin \langle k, x \rangle  = 0$
the integral is simply $|k|$ times the surface volume of the nodal
set, which is known to be of size $|k|$.  Also, we
have $\int_{{\bf T}} |\sin \langle k, x \rangle| dx \geq C$. Thus,
our method gives  the sharp lower bound
$\hcal^{n-1}(Z_{\phi_{\lambda}}) \geq C \lambda^{1}$ in this
example.

So the upper bound  is achieved in this example. Also, we
have $\int_{{\bf T}} |\sin \langle k, x \rangle| dx \geq C$. Thus,
our method gives  the sharp lower bound
$\hcal^{n-1}(Z_{\phi_{\lambda}}) \geq C \lambda^{1}$ in this
example.
Since $\cos
\langle k, x \rangle = 1$ when $\sin \langle k, x \rangle  = 0$
the integral is simply $|k|$ times the surface volume of the nodal
set, which is known to be of size $|k|$. 
\medskip

\subsection{Spherical harmonics on $S^2$}
\medskip

The  $L^1$ of $Y^N_0$  norm can be derived from
the asymptotics of Legendre polynomials
$$P_N(\cos \theta) = \sqrt{2} (\pi N \sin \theta)^{-\half} \cos
\left( (N + \half) \theta - \frac{\pi}{4} \right) + O(N^{-3/2}) $$
where the remainder is uniform on any interval $\epsilon < \theta
< \pi - \epsilon$. We have
$$||Y^N_0||_{L^1} = 4 \pi  \sqrt{\frac{(2 N +
1)}{2 \pi}} \int_0^{\pi/2} |P_N(\cos r)| dv(r) \sim C_0 > 0,$$
i.e. the $L^1$ norm is asymptotically a positive constant. Hence
$\int_{Z_{Y^N_0}}  |\nabla Y^N_0| ds \simeq C_0 N^2 $. In this
example $|\nabla Y_0^N|_{L^{\infty}} = N^{\frac{3}{2}}$ saturates
the sup norm bound. So the estimate of \eqref{3}
produces the lower bound $\hcal^{n-1}(Z_{\phi_{\lambda}}) \geq
\lambda^{\half}$. The accurate lower bound is $\lambda$, as one
sees from the rotational invariance and by the fact that $P_N$ has
$N$ zeros. The defect in the argument is that the bound  $|\nabla
Y_0^N|_{L^{\infty}} = N^{\frac{3}{2}}$ is only obtained on the
nodal components near the poles, where each component has  length
$\simeq \frac{1}{N}$.
\bigskip

\noindent{\bf Gaussian beams} \medskip

Gaussian beams are Gaussian shaped lumps which are concentrated on
$\lambda^{-\half}$ tubes $\tcal_{\lambda^{- \half}}(\gamma)$
around closed geodesics and have height $\lambda^{\frac{n-1}{4}}$.
We note that their $L^1$ norms decrease like
$\lambda^{-\frac{(n-1)}{4}}$, i.e. they saturate the $L^p$ bounds of \cite{Sog} for
small $p$.   In such cases we have $\int_{Z_{\phi_{\lambda}}}
|\nabla \phi_{\lambda}| dS \simeq \lambda^2
||\phi_{\lambda}||_{L^1} \simeq \lambda^{2 - \frac{n-1}{4}}. $ It
is likely that Gaussian beams are minimizers of the $L^1$ norm
among $L^2$-normalized eigenfunctions of Riemannian manifolds.
Also, the gradient bound $||\nabla \phi_{\lambda}||_{L^{\infty}} =
O(\lambda^{\frac{n + 1}{2}})$ is far off for Gaussian beams, the
correct upper bound being $\lambda^{1 + \frac{n-1}{4}}$.  If we
use these estimates on $||\phi_{\lambda}||_{L^1}$ and $||\nabla
\phi_{\lambda}||_{L^{\infty}}$,  our method gives
$\hcal^{n-1}(Z_{\phi_{\lambda}}) \geq C \lambda^{1 -
\frac{n-1}{2} }$,  while $\lambda$ is the correct lower bound for
Gaussian beams in the case of surfaces of revolution (or any real
analytic case). The defect is again that the gradient estimate is
achieved only very close to the closed geodesic of the Gaussian
beam. Outside of the tube  $\tcal_{\lambda^{- \half}}(\gamma)$ of
radius $\lambda^{- \half}$  around the geodesic, the Gaussian beam
and all of its derivatives decay like $e^{- \lambda d^2}$ where
$d$ is the distance to the geodesic. Hence
$\int_{Z_{\phi_{\lambda}}} |\nabla \phi_{\lambda}| dS \simeq
\int_{Z_{\phi_{\lambda}} \cap \tcal_{\lambda^{-
\half}}(\gamma)} |\nabla \phi_{\lambda}| dS. $ Applying  the
gradient bound for Gaussian beams  to the latter integral
 gives $\hcal^{n-1}(Z_{\phi_{\lambda}} \cap
\tcal_{\lambda^{- \half}}(\gamma)) \geq C \lambda^{1 -
\frac{n-1}{2}}$, which is sharp since the intersection
$Z_{\phi_{\lambda}} \cap \tcal_{\lambda^{- \half}}(\gamma)$
cuts across $\gamma$ in $\simeq \lambda$ equally spaced points (as
one sees from the Gaussian beam approximation).

\subsection{Non-scarring of nodal sets on $(M, g)$ with ergodic geodesic flow}

In this section, we  prove a rather simple (unpublished) result on nodal sets when the geodesic
flow of $(M, g)$ is ergodic. Since there exist many expositions of quantum ergodic eigenfunctions, we 
only briefly recall the main facts and definitions and refer to \cite{Z5,Z6} for further background. 

Quantum ergodicity concerns the semi-classical (large $\lambda$) asymptotics of eigenfunctions in
the case where the geodesic flow $G^t$ of $(M, g)$ is ergodic. We recall that the geodesic flow is
the Hamiltonian flow of the Hamiltonian $H(x, \xi) = |\xi|_g^2$ (the length squared) and that ergodicity
means that the only $G^t$-invariant subsets of the unit cosphere bundle $S^* M$ have either full Liouville
measure or zero Liouville measure (Liouville measure is the natural measure on the level set $H = 1$ induced
by the symplectic volume measure of $T^* M$).

We will say that a sequence $\{\phi_{j_k}\}$ of $L^2$-normalized
eigenfunctions  is {\it quantum ergodic} if
\begin{equation} \label{QEDEF} \langle A \phi_{j_k}, \phi_{j_k} \rangle \to
\frac{1}{\mu(S^*M)} \int_{S^*M} \sigma_A d\mu,\;\;\; \forall A \in
\Psi^0(M). \end{equation} Here, $\Psi^s(M)$ denotes the space of
pseudodifferential operators of order $s$, and $d \mu$ denotes
Liouville measure on the unit cosphere bundle $S^*M$ of $(M, g)$.
More generally, we denote by $d \mu_{r}$ the (surface) Liouville
measure on $\partial B^*_{r} M$, defined by
\begin{equation} \label{LIOUVILLE} d \mu_r = \frac{\omega^m}{d |\xi|_g} \;\; \mbox{on}\;\; \partial B^*_r
M.
\end{equation}
We also denote by $\alpha$ the canonical action $1$-form of
$T^*M$.

The main result is   that there exists a subsequence $\{\phi_{j_k}\}$ of
eigenfunctions whose indices $j_k$ have counting density one for
which $\rho_{j_k}(A): = \langle A \phi_{j_k}, \phi_{j_k}\rangle \to \omega(A)$ (where as above
$\omega(A) = \frac{1}{\mu(S^*M)} \int_{S^*M} \sigma_A d\mu $ is the normalized Liouville average of $\sigma_A$).  Such a
 sequence of  eigenfunctions is called a sequence of  `ergodic
eigenfunctions'. The key quantities to study are the quantum variances
\begin{equation} \label{diag} V_A(\lambda) : =
\frac{1}{N(\lambda)} \sum_{j:  \lambda_j \leq \lambda} |\langle A
\phi_j, \phi_j \rangle - \omega(A)|^2.
\end{equation}
The following result is the culmination of the results in \cite{Sh.1,Z1,CV,ZZw,GL}.

\begin{theo} \label{QE}    Let $(M,g)$ be a compact
Riemannian manifold (possibly with boundary), and let
$\{\lambda_j, \phi_j\}$ be the spectral data of its Laplacian
$\Delta.$ Then the geodesic flow
 $G^t$ is ergodic  on $(S^*M,d\mu)$ if and only if, for every
$A \in \Psi^o(M)$,  we have:
\medskip

\begin{enumerate}

 \item $\lim_{\lambda \rightarrow \infty} V_A(\lambda) =0.$
\medskip

 \item $(\forall \epsilon)(\exists \delta)
\limsup_{\lambda \rightarrow \infty} \frac{1} {N(\lambda)}
\sum_{{j \not= k: \lambda_j, \lambda_k \leq \lambda}\atop {
|\lambda_j - \lambda_k| < \delta}} |( A \phi_j, \phi_k )|^2 <
\epsilon $
\end{enumerate}

\end{theo}

Since all the terms in (1)
are positive, no cancellation is possible, hence  (1)  is
equivalent to the existence of a subset ${\mathcal S} \subset \N$
of density one such that ${\mathcal Q}_{{\mathcal S}} := \{ d
\Phi_k : k \in {\mathcal S}\}$ has only $\omega$ as a weak* limit
point.

We now consider nodal sets of quantum ergodic eigenfunctions. The following result says that
if we equip nodal sets with the measure $\frac{1}{\lambda_j^2}|\nabla \phi_{\lambda_j}| dS$,
then nodal sets cannot `scar', i.e. concentrate singularly as $\lambda_j \to \infty$.

\begin{mainprop} \label{L1REGa} Suppose that $\{\phi_{\lambda_j}\}$ is a quantum ergodic
sequence. Then any weak limit of $\{\frac{1}{\lambda_j^2}|\nabla \phi_{\lambda_j}| dS\}$ must
be absolutely continuous with respect to $dV$. \end{mainprop}

 We recall that,   for any $f \in C^2(M)$,
\begin{equation} \label{DONGTYPE}  \int_M \left((\Delta + \lambda^2) f \right)  |\phi_{\lambda}| dV =
\int_{Z_{\phi_{\lambda}}}   f |\nabla \phi_{\lambda}| dS.
\end{equation}

The identity for general $f \in C^2(M)$  can  be used to investigate the equidistribution of
nodal sets equipped with the surface measure $|\nabla
\phi_{\lambda}| dS$. We denote the normalized measure by
$\lambda^{-2} |\nabla \phi_{\lambda_j}| dS
|_{Z_{\phi_{\lambda}}}$.

\begin{mainlem} \label{W*}The weak * limits of the sequence  $\{\lambda^{-2} |\nabla
\phi_{\lambda_j}| dS |_{Z_{\phi_{\lambda}}}\}$ of bounded
positive measures are the same as the weak * limits of
$\{|\phi_{\lambda_j|}\}$ (against $f \in C(M)$.)
\end{mainlem}

 We let $f \in C^2(M)$ and multiply the identity (\ref{DONGTYPE}) by $\lambda^{-2}$. We then
integrate by parts to put $\Delta$ on $f$. This shows that for $f
\in C^2(M)$, we have $$\int_M f |\phi_{\lambda}| dV = \lambda^{-2}
\int _{Z_{\phi_{\lambda}}} f |\nabla_{\lambda}| dS +
O(\lambda^{-2}). $$ Letting $f = 1$, we see that the family of
measures $\{\lambda^{-2} |\nabla \phi_{\lambda_j}|^2 \delta
(\phi_{\lambda_j})\}$ is bounded. By uniform approximation of $f
\in C(M)$ by elements of $C^2(M)$, we see that the weak* limit
formula extends to $C(M)$.

\begin{mainlem} \label{L1REG} Suppose that $\{\phi_{\lambda_j}\}$ is a quantum ergodic
sequence. Then any weak limit of $\{|\phi_{\lambda_j}| dS\}$ must
be absolutely continuous with respect to $dV$. \end{mainlem}

We recall that  a sequence of measures $\mu_n $  converges weak
* to $\mu$ if  $\int_M f d\mu_n \to \int f d\mu $ for all
continuous $f$.  A basic fact about weak * convergence of measures
is that $\int f d\mu_n \to \int f d\mu$ for all $f \in C(M)$
implies that $\mu_n(E) \to \mu(E)$ for all sets $E$ with
$\mu(\partial E) = 0$ (Portmanteau theorem).

We also recall that a sequence of eigenfunctions is called quantum
ergodic (in the base) if
 \begin{equation} \label{QEa} \int f
|\phi_{\lambda_j}|^2 dV \to \frac{1}{Vol(M)} \int_M f dV.
\end{equation}
In other words, $\phi_{\lambda}^2 \to 1$ in the weak * topology,
i.e. the vague topology on measures.
We now prove Lemma \ref{L1REG}.

\begin{proof} Suppose that $|\phi_{\lambda_{j_k}}|dV \to d\mu$
and assume that  $d\mu = c dV + d\nu$ where $d\nu$ is singular
with respect to $dV$. Let $\Sigma =\mbox{supp}\;\; \nu$,  and let
$\sigma = \mu(\Sigma) = \nu(\Sigma)$. Let $\tcal_{\epsilon}$ be
the $\epsilon$-tube around $\Sigma$. Then $$\lim_{k \to \infty}
\int_{\tcal_{\epsilon}} |\phi_{\lambda_{j_k}}|dV = c
Vol(T_{\epsilon}) + \nu(\Sigma) = \sigma + O(\epsilon). $$ But for
any set $\Omega \subset M$, $\int_{\Omega} |\phi_{\lambda_j}| dV
\leq \sqrt{Vol(\Omega)} \sqrt{\int_{\Omega} |\phi_{\lambda_j}|^2 d
V}. $ Hence if $Vol(\partial \Omega) = 0$,  $\limsup_{j \to
\infty} \int_{\Omega} |\phi_{\lambda_j}| dV \leq Vol(\Omega)$.
Letting $\Omega = \tcal_{\epsilon}(\Sigma)$ we get $\sigma +
O(\epsilon) \leq Vol(T_{\epsilon}(\Sigma)) = O(\epsilon)$ since
$\lim_{k \to \infty} \int_{\tcal_{\epsilon}}
|\phi_{\lambda_{j_k}}|^2 dV = Vol(\tcal_{\epsilon}) =
O(\epsilon)$. Letting $\epsilon \to 0$ gives a contradiction.

\end{proof}

Of course, it is possible that the only weak* limit is zero.

\subsection{\label{LINFTYQE} Weak* limits for $L^{\infty}$ quantum ergodic
sequences}

To our knowledge, the question whether the limit (\ref{QE}) holds
 $f \in L^{\infty}$ when $(M,g)$ has ergodic geodesic flow  has not been studied. It is
equivalent to strengthening the Portmanteau statement to all
measurable sets $E$, and is equivalent to the statement that
$\{\phi_{\lambda_j}^2\} \to 1$ weakly in $L^1$.  We call such
sequences $L^{\infty}$ quantum ergodic on the base.  The term `on
the base' refers to the fact that we only demand   quantum
ergodicity for the projections of the `microlocal lifts' to the
base $M$. For instance, the exponential eigenfunctions of flat
tori are $L^{\infty}$  quantum ergodic in this sense.

\begin{lem} \label{EG} Suppose that $\{\phi_j\}$ is an $L^{\infty}$- quantum ergodic
sequence. Then there exists $\epsilon > 0$ so that
$||\phi_j||_{L^1} \geq \epsilon > 0$ for all $j$. \end{lem}

\begin{proof}  We argue by contradiction.  If the conclusion were false,  there would exist a subsequence   $\phi_{j_k} \to 0$ strongly  in
$L^1$, but with  $\phi_{j_k}^2 dV \to dV$ weakly in $L^1$.
The first  assumption implies the existence of a  subsequence (which we continue to denote by $\phi_{j_k}$)
satisfying $\phi_{j_k} \to 0$ a.e. $dV$. But $L^1$ has the weak Banach-Saks property:  any weakly convergent sequence
in $L^1$ has a subsequence whose 
arithmetic means converge
strongly (Szlenk's weak Banach-Saks theorem for $L^1$).  We choose such a subsequence for 
$\phi_{j_k}$ and continue to denote it as $\phi_{j_k}$.  This subsequence 
has the properties that 

\begin{enumerate}

\item  $\phi_{j_k} \to 0$ a.e.


 \item  $\psi_N := \frac{1}{N} \sum_{k\leq N} \phi_{j_k}^2 \to 1$
strongly in $L^1$. 

\end{enumerate} 

But 
 $\psi_N(x) \to 0$ on the same set  where $\phi_{j_k}(x) \to 0$, hence by (1) $\psi_N \to 0$ a.s.
This contradicts (2) and completes the proof.

\end{proof}

Combining with the above, we have 

\begin{cor}   Suppose that $\{\phi_{\lambda_j}\}$ is an $L^{\infty}$
quantum ergodic sequence on the base.  Then the conjectured Yau lower bound holds:
$  \hcal^{n-1} (Z_{\phi_{\lambda}}) \geq C _g \lambda$ for some $C_g > 0$. \end{cor}

 We also see that the limits in Proposition \ref{L1REGa} are non-zero:

\begin{cor} Suppose that $\{\phi_{\lambda_j}\}$ is an $L^{\infty}$
quantum ergodic sequence on the base. Then there exists $C >0$ so that any weak limit of the sequence
$\frac{1}{\lambda^2} |\nabla \phi_{\lambda_j}| dS |_{Z_{\phi_{\lambda_j}}}$ has mass $ \geq
C >0$.\end{cor}

Of course, such an abstract functional analysis argument only serves a purpose if we can prove
that  eigenfunctions of $\Delta$ 
are $L^{\infty}$ quantum ergodic on
the base in interesting cases. It is natural to conjecture that this condition holds on negatively curved
manifolds, since 
 the expected $L^1$ norm of a random wave is bounded
below by a positive constant.  The main problem is that $L^{\infty}(M)$ is a non-separable Banach space.
The standard quantum ergodicity arguments show that   (when quantum ergodicity is valid),  for any Borel
set $E$ there exists a subsequence $\scal_E$  of density one so that \begin{equation} \label{QEE} \lim_{k \to \infty, j_k \in \scal_E}
\int_E \phi_{j_k}^2 d V =  Vol(E).  \end{equation}
However, the non-separability of $L^{\infty}(M)$ means that one cannot use the diagonalization argument of \cite{Z1,CV} to show that there exists a density
one subsequence independent of $E$ so that \eqref{QEE} holds.
If $L^{\infty}$ quantum ergodicity fails, then zero-density subsequences of eigenfunctions would `scar' along Cantor
sets C of positive measure. That is, the mass $\int_C \phi_{j_k}^2 d V$ may tend to a larger value  than $Vol(C)$. 

\bigskip

\noindent{\bf Equidistributed sums of Gaussian beams and quantum
ergodicity}
\medskip

We briefly consider the question whether it is possible to have a
quantum ergodic sequence of eigenfunctions for which
$||\phi_j||_{L^1} \to 0. $ 

First, we observe that 
there do exist sequences of quantum ergodic functions (not eigenfunctions) with this property: 
 $\sum_{j = 1}^{M(n)} \sqrt{\frac{n}{M(n)}} \chi_{[x_j(n), x_j(n) + \frac{1}{n}]} \to 0$ in $L^1([0, 1],
dx)$ as long as $M(n) = o(n). $  But its square is the probability
measure $\frac{1}{M(n)} \sum_{j = 1}^{M(n)} n \chi_{[x_j(n),
x_j(n) + \frac{1}{n}]} $ and if the $\{x_j(n)\}$ are uniformly
distribution in $[0, 1]$ (w.r.t. $dx$), this tends weakly to $dx$.

It is tempting to construct  sequences of eigenfunctions with the same
property: a Gaussian beam $Y^N_{\gamma}$ on the standard $S^2$
associated to a closed geodesic $\gamma$ (i.e. a rotate of
$Y^N_N$) is of height $\lambda^{\frac{1}{2}}$ in a tube of radius
$\sqrt{\lambda}$ around $\gamma$. If we let $M(N) = o(N^{\half})$
and choose $M(N)$ closed geodesics which are
$\frac{1}{\sqrt{M(N)}}$--separated, and become equidistributed in
the space of closed geodesics, then $\phi_N =
\frac{1}{\sqrt{M(N)}}\sum_{j = 1}^{M(N)} Y^N_{\gamma_j}$ is an
eigenfunction whose $L^1$-norm tends to zero like $\sqrt{M(N)}
N^{- \frac{1}{4}}$ but whose $L^2$ norm is asymptotic to $1$ and
whose modulus square tends weak* to $1$. More precisely,
$\frac{1}{M} \sum_{j = 1}^{M(N)} |Y^N_{\gamma_j}|^2 \to 1$ weakly.
To prove that  $|\phi_N|^2 \to 1$ requires proving that
$\frac{1}{M(N)} \sum_{j \not= k} Y^N_{\gamma_j}
\overline{Y^N_{\gamma_k}} \to 0$. The sum is over $\sim M(N)^2$
terms which  are exponentially outside the tube intersections
$T_{\lambda^{-\half}}( \gamma_j)\cap
T_{\lambda^{-\half}}(\gamma_k)$. In the sum we may fix $j = j_0$
and multiply by $M(N)$. So we need then to show that $\sum_{k
\not= j_0} |\langle Y^N_{\gamma_{j_0}}, Y^N_{\gamma_k} \rangle|
\to 0$.    The geodesics are well-separated if the distance in the
space of geodesics between them is $\geq \frac{1}{\sqrt{M(N)}}$,
which means that the angle between $\gamma_j$ and $\gamma_k$ is at
least this amount. When the angle is $\geq \epsilon$ then the
inner product $|\langle Y^N_{\gamma_j}, Y^N_{\gamma_k} \rangle |
\leq \frac{1}{\epsilon} N^{-1}$ since the  area of
$T_{\lambda^{-\half}}( \gamma_j)\cap
T_{\lambda^{-\half}}(\gamma_k)$ is  bounded by this amount. For
any $\epsilon$ the sum over geodesics separated by $\epsilon$ is
$O(\frac{1}{\epsilon} M(N) N^{-1})$. The remaining number of terms
is $O(\epsilon^2 M(N))$. So if $\epsilon = o(\sqrt{M(N)})$ both
terms tend to zero.

\subsection{Intersections of nodal sets of orthogonal eigenfunctions}

A related question is whether nodal sets of orthogonal eigenfunctions of the same eigenvalue must intersect.
Of course, this question only arises when the eigenvalue has multiplicity $ > 1$.   A result of
this kind was obtained by V. Gichev under a topological condition on $M$.

\begin{theo}  \label{Gi} \cite{Gi}  Suppose that $H^1(M) = 0$ and that $\phi_{\lambda, 1}, \phi_{\lambda, 2}$ are
orthogonal eigenfunctions with the same eigenvalue $\lambda^2$. Then $Z_{\phi_{\lambda,1}}
\cap Z_{\phi_{\lambda,2}} \not= \emptyset$. \end{theo} 

We briefly sketch the proof: Let  $\acal_1$ resp. $\acal_2$ be the family of nodal domains of $\phi_{\lambda, 1}$
resp. $\phi_{\lambda, 2}$.  Each union $\bigcup_{W \in \acal_j} W$ covers $M$ up to the nodal set of $\phi_{\lambda, j}$.
If the nodal sets do not intersect then the nodal set of $\phi_{\lambda,2}$ is contained in  $\bigcup_{W \in \acal_1} W$,
for instance; similarly if the indices are reversed. Hence the nodal sets have empty intersection if and only if
$\bigcup_{W \in \acal_1} W \cup \bigcup_{W \in \acal_2} W$ covers $M$. Under this condition, Gichev constructs
a closed 1-form which is not exact by showing that the  incidence graph of the cover obtained from the union of the nodal
domains of $\phi_{\lambda, 1}$ and $\phi_{\lambda, 2}$ contains a cycle. He then considers a nodal domain $U$
of $\phi_{\lambda, 1}$ and a nodal domain $V$ of $\phi_{\lambda,2}$ which intersect. Let $Q = \partial U \cap V$.
Since $Q \cap \partial V \not= \emptyset$ there exists a smooth function $f$ on $M$ such that $f \equiv 1$
in a neighborhood of $Q$ and $f = 0$ near $\partial U \backslash Q$. Let $\eta $ be the one form which equals
$df$ on $U$ and $0$ on the complement of $U$. Clearly $\eta$ is closed and it is verified in \cite{Gi} that $\eta$ is 
not exact. 

Givech also proves that for $S^2$, if $0$ is a regular value of $\phi_{\lambda,1}$ then
$\# Z_{\phi_{\lambda,1}}
\cap Z_{\phi_{\lambda,2}} \geq 2$   for every orthogonal eigenfunction
$\phi_{\lambda, 2}$ with the same eigenvalue. The proof is simply to use Green's formula
for a nodal domain for $\phi_{\lambda, 1}$ and note that the integral of $\phi_{\lambda,2} \frac{\partial}{\partial \nu} \phi_{\lambda,1}$ equals zero on its boundary.

A related observation is the curious identity of \cite{SoZ}, which holds for any $(M, g)$:  for any pair of eigenfunctions,
$$(\lambda_j^2 - \lambda_k^2) \int_M \phi_{\lambda_k} |\phi_{\lambda_j}| dV
= 2\int_{Z_{\phi_{\lambda_j}}} \phi_{\lambda_k} |\nabla
\phi_{\lambda_j}| dS. $$
Hence for a pair of orthogonal eigenfunctions of the same eigenvalue, 
$$\int_{Z_{\phi_{\lambda_j}}} \phi_{\lambda_k} |\nabla
\phi_{\lambda_j}| dS = 0. $$

\section{Norms and nodal sets}

Studies of nodal sets often involve dual studies of $L^p$ norms of eigenfunctions. In this section, we review
a number of relatively recent results on $L^p$ norms, both in the global manifold $M$ and for restrictions of
eigenfunctions to submanifolds. 

\subsection{Polterovich-Sodin on norms and nodal sets}

Let $\acal(\phi_{\lambda})$ denote the collection of nodal domains of $\phi_{\lambda}$. For $A \in \acal(\phi_{\lambda})$
let $m_A = \max_A |\phi_{\lambda}|. $  In \cite{PS} the following is proved (see Corollary 1.7):

\begin{theo} \cite{PS} Let $(M, g)$ be a $C^{\infty}$ Riemannian surface.  For every  $\phi_{\lambda}$ with $\|\phi_{\lambda}\|=1$,
   $$ \sum_{A\in {\mathcal A}}m^6_A\leq k_g\lambda^{3}.$$  Hence,  for each $a>0$, the number of nodal domains
    $A$ of $\phi_{\lambda}$ where the maximal bound  $m_A\geq a\lambda^{1/2}$ is achieved in order of 
magnitude  does not exceed $k_g a^{-6}$. In particular, for fixed $a$, it remains bounded as $\lambda\rightarrow\infty$.
\end{theo}

The proof uses a certain  Bananch indicatrix,  the Sogge $L^6 $ bounds, and estimates on the inradius of
nodal domains. 
For a continuous function $u\in C(\Bbb{R})$,  the generalized Banach
   indicatrix is defined by $$ B(u,f)=\int_{-\infty}^{+\infty}u(c)\beta(c,f)dc, $$
where for a  regular value $c\in \Bbb{R}$ of $f$,  $\beta(c,f)$ is
  the number of connected components of $f^{-1}(c)$. 
In \cite{PS}, the  integral $B(u,f)$) is bounded from above
through the $L^2$-norms of the function $f$ and $\Delta
   f$. I.e.. in  Theorem 1.3. For any $f\in  {\mathcal F}_\lambda$ and any continuous function $u$ on $\Bbb{R}$,
   $$ B(u,f)\leq k_g\|u\circ f\|(\|f\|+\|\Delta f\|). $$

The proof is roughly as follows: Let $p_i$ be a point of $A_i$ where th maximum is achieved. 
By the  inradius bound \cite{Man3}, there exists $\mu > 0$ so that the disc $D(p_j,  \frac{\mu}{\lambda}) \subset A_i$. 
One can then express $\phi_{\lambda}$ in  $D(p_j,  \frac{\mu}{\lambda}) $ by the sum of  a Green's integral and Poisson
integral with respect to the Euclidean Dirichlet Green's function of a slightly smaller disc. In particular one
may express $\phi_{\lambda}(p_j)$ by such an integral. Apply H\"older's inequality one gets 
$$m_j^6 \leq k_g \lambda^2 \int_{D(p_j, r)} \phi_{\lambda}^6 d V, \;\;\; (r = \mu \lambda^{-\half}). $$
Since the discs are disjoint one can sum in $j$ and apply the Sogge $L^6$ bound to include the proof. 
Thus, the only fact one used about nodal domains was lower bound on the inradius.

 This result bears a curious comparison to the results of \cite{STZ} giving new constraints on $(M, g)$
which are of maximal eigenfunction growth, i.e. possess eigenfunctions such that $m_A \geq C \lambda^{\half}$
for some  sequence of eigenfunctions $\phi_{\lambda_j}$ with $\lambda_j \to \infty$. The result (building on
older results of Sogge and the author)  states that such a sequence can exist only if $(M, g)$ possesses a `pole' 
$p$ for which the set of geodesic loops $\lcal_p$ based at $p$ has positive measure in $S^*_p M$ (with respect to the natural
spherical volume measure) and such that the first return map has a recurrence property. In fact, the only known
surfaces  where the bounds are achieved are surfaces of revolution, and in this case  the first return map is
the identity. It is quite plausible that if $(M, g)$ has maximal eigenfunction growth, then the first return map 
must be the identity map on a set of positive measure in $\lcal_p$.

   Combined with the
Polterovich-Sodin result above, we see that such `poles' $p$, when they exist, can only occur in a uniformly bounded number
of nodal domains of a surface.  It would be interesting to know if there can exist only a finite number of such points at all
if one additionally assumes that the set of smoothly closed geodesics has measure zero.  For instance,, in that
case,  there
might be a unique pole in each of the finite number of possible nodal domains. This finitude problem
would be useful in strengthening the condition on $(M, g)$ of maximal eigenfunction growth.

\subsection{\label{NORMS} Norms of restrictions}

A problem of current interest is to consider $L^p$ norms of restrictions of eigenfunctions to hypersurfaces
or higher codimension submanifolds.  For expository purposes we only consider geodesics on surfaces here. 
Following earlier work of A. Reznikov, Burq, G\'erard and Tzvetkov \cite{BGT}  proved

\begin{theo}  \label{bgt} \cite{BGT} Suppose that $(M, g)$ is a compact surface, then there exists $\lambda_0(\epsilon),  C  > 0$ so that, for  any geodesic segment $\gamma $
of length $L_{\gamma}$ 
 and any eigenfunction $\phi_{\lambda}$ with $\lambda \geq \lambda_0$ we have
\begin{equation} \label{littleoha}  \frac{1}{L_{\gamma} }\int_{\gamma} |\phi_{\lambda}|^2 ds \leq C
  \lambda^{\frac{1}{2}}  ||\phi_{\lambda}||^2 \end{equation}
\end{theo}
Their estimate is sharp for the round
sphere $S^2$ because of the highest weight spherical harmonics  They 
also showed that for all geodesic segments $\gamma$ of unit length,
$$\left(\,   \frac{1}{L_{\gamma} } \int_\gamma |\phi_\la|^4\, ds \, \right)^{1/4}\le C\la^{\frac14}\|e_\la\|_{L^2(M)},
$$

The estimate is only known to be achieved when the geodesic is elliptic, and quite likely it can be improved
if the geodesic is hyperbolic. A result in this direction is:

\begin{theo}  \label{STR} \cite{SoZ2} Suppose that $(M, g)$ is a compact surface of non-positive curvature. 
Then for all $\epsilon$,  there exists $\lambda_0(\epsilon),  C  > 0$ so that, for  any geodesic segment $\gamma $
of length $L_{\gamma}$ 
 and any eigenfunction $\phi_{\lambda}$ with $\lambda \geq \lambda_0(\epsilon),$ we have
\begin{equation} \label{littleoh}  \frac{1}{L_{\gamma} }\int_{\gamma} |\phi_{\lambda}|^2 ds \leq C
\epsilon  \lambda^{\frac{1}{2}}  ||\phi_{\lambda}||^2 \end{equation}
\end{theo}

A related result on $L^4$ norms is,

\begin{theo}\label{theorem1.1} \cite{SoZ3}
Let $(M,g)$ be a surface and assume that the set
\begin{equation}\label{2.0}
{\mathcal P}=
\{(x,\xi)\in S^*M: g^t(x,\xi)
=(x,\xi), \, \, \text{some }\; t>0\}
\end{equation}  of periodic points has Liouville measure zero in $S^*M$. 
Then
there is a subsequence of eigenvalues $\lambda_{j_k}$ of density one 
so that
\begin{equation}\label{2.2}
\|e_{\lambda_{j_k}}\|_{L^4(M)}=o(\lambda^{1/8}_{j_k}).
\end{equation}

\end{theo}

The  results are based in part on a relatively new Kakeya-Nikodym maximal function estimate of Bourgain
\cite{Bour}, as improved by Sogge \cite{Sog2}. 
We believe that it can be improved the following
phase space Kakeya-Nikodym theorem. Let $T_{\delta} (\gamma)$ be the tube of radius $\delta$
around a geodesic arc in $M$, and let $ \chi_{\delta, \gamma}$ be a smooth cutoff to a  phase space tube of its
lift to $S^*M$. Then  for all $\epsilon$, there
exists $\delta(\epsilon)$ such that
$$ \limsup_{\lambda \to \infty}
\frac1{N(\lambda)}\sum_{\lambda_j\le \lambda} \sup_{\gamma\in
\Pi}\int_{T_{\delta(\epsilon), (\gamma)}} | \phi_\lambda|^2\, ds <
\epsilon.$$
We expect the sup
occurs when $\gamma$ is the orbit of $(x, \xi)$. But then it is easy to estimate the right side and one
should be able to get a quantitative improvement of Theorem \ref{theorem1.1}.

\subsection{\label{QERsection}Quantum ergodic restriction (QER) theorems}

In this section we briefly review a recent series of results \cite{TZ2,TZ3,DZ,CTZ} on quantum ergodic restriction
theorems. They are used in section \S \ref{NODALGEOS} to determine the limit distribution of intersections of
nodal lines and geodesics on real analytic surfaces (in the complex domain).

Let $H \subset M$ be a hypersurface and consider the Cauchy data $(\phi_j |_H, \lambda_j^{-1} \partial_{\nu} \phi_j|_H)$
of eigenfunctions along $H$; here $\partial_{\nu}$ is the normal derivative.  We refer to $\phi_j |_H$ as the 
Dirichlet data and to $\lambda_j^{-1} \partial_{\nu} \phi_j|_H$ as the Neumann data. A QER (quantum
ergodic restriction)  theorem
seeks to find limits of matrix elements of this data along $H$ with respect to pseudo-differential
operators $Op_H(a)$ on $H$. The main idea is that $S^*_H M$, the set of unit covectors with footpoints
on $H$, is a cross-section to the geodesic flow and the first return map of the geodesic flow for $S^*_H M$
is ergodic. The Cauchy data should be the quantum analogue of such a cross section and therefore
should be quantum ergodic on $H$. 

For  applications to nodal sets and other problems, it is important to know if the Dirichlet data
alone satisfies a QER theorem. The answer is obviously `no' in general. For instance if $(M, g)$
has an isometric involution and with a hypersurface $H$ of fixed points, then any eigenfunction which
is odd with respect to the involution vanishes on $H$. But in \cite{TZ2,TZ3} a sufficient condition
is given for quantum ergodic restriction, which rules out this and more general situations.  The symmetry
condition is that geodesics emanating from the `left side' of $H$ have a different return map from geodesics
on the `right side' when the initial conditions are reflections of each other through $TH$.   To take the simplest example of the circle, the restriction of $\sin k x$ to a point
is never quantum ergodic but the full Cauchy data $(\cos k x, \sin k x)$ of course satisfies $\cos^2 kx + \sin^2 kx = 1$.
In \cite{CTZ} it is proved that Cauchy data always satisfies QER for any hypersurface. This has implications for 
(at least complex) zeros of even or odd eigenfunctions along an axis of symmetry, e.g. for the case of Maass 
forms for the modular domain $SL(2, \Z)/\HH^2$  (see \S \ref{NODALGEOS}).

To state the QER theorem, we introduce some notation.   We put
\begin{equation} T^*_H M = \{(q, \xi) \in T_q^* M, \;\; q\in H\} , \;\;\;T^* H=  \{(q, \eta) \in T_q^* H, \;\; q\in H\}. \end{equation}  We further denote by $ \pi_H : T^*_H M \to T^*
H $  the restriction map,
\begin{equation} \label{RESCV}  \pi_H(x, \xi)  = \xi |_{TH}.
\end{equation}

For any orientable (embedded) hypersurface $H \subset M$, there
exists two unit normal co-vector fields $\nu_{\pm}$ to $H$ which
span half ray bundles $N_{\pm} = \R_+ \nu_{\pm} \subset N^* H$.
Infinitesimally, they define two `sides' of $H$, indeed they are
the two components of  $T^*_H M \backslash T^* H$. We  use
Fermi normal coordinates $(s, y_n)$ along $H$ with $s \in H$ and
with $x = \exp_x y_n \nu$ and  let $\sigma, \eta_n$ denote the dual
symplectic coordinates.
For  $(s, \sigma) \in B^* H$ (the co-ball bundle),   there exist two unit
covectors $\xi_{\pm}(s, \sigma) \in S^*_s M$ such that
$|\xi_{\pm}(s, \sigma)| = 1$ and $\xi|_{T_s H} = \sigma$. In the
above orthogonal decomposition, they are given by
\begin{equation} \label{xipm} \xi_{\pm}(s, \sigma) = \sigma \pm \sqrt{1 - |\sigma|^2}
\nu_+(s). \end{equation} We define the reflection involution
through $T^* H$ by
\begin{equation}\label{rHDEF}  r_H: T_H^* M \to T_H^* M, \;\;\;\;
r_H(s, \mu\; \xi_{\pm}(s, \sigma)) =(s, \mu\; \xi_{\mp}(s,
\sigma)), \,\,\,  \mu \in \R_{+}.
\end{equation} Its  fixed point
set is $T^* H$.

We denote by $G^t$ the homogeneous  geodesic flow of $(M, g)$,
i.e.  Hamiltonian flow on $T^*M - 0$ generated by $|\xi|_g$. 
We define the {\it first return time}
$T(s, \xi)$ on $S^*_H M$ by,
\begin{equation} \label{FRTIME} T(s, \xi) = \inf\{t > 0: G^t (s,
\xi) \in S^*_H M, \;\;\ (s, \xi) \in S^*_H M)\}. \end{equation} By
definition $T(s, \xi) = + \infty$ if the trajectory through $(s,
\xi)$ fails to return to $H$. 
Inductively, we
define the jth return time $T^{(j)}(s, \xi)$ to $S^*_H M$ and the
jth return map $\Phi^j$ when the return times are finite.

We define the first return map on the same domain by
\begin{equation} \label{FIRSTRETURN} \Phi: S^*_H M \to S^*_H M, \;\;\;\; \Phi(s, \xi) = G^{T(s, \xi)} (s, \xi) \end{equation}
When $G^t$ is ergodic, $\Phi$ is defined almost everywhere and is
also ergodic with respect to Liouville measure  $\mu_{L, H}$  on $S^*_H M$. 

\begin{defin} \label{ANC}  We say that  $H$ has a positive
measure of microlocal reflection symmetry if 
$$  \mu_{L, H} \left( \bigcup_{j \not= 0}^{\infty}  \{(s, \xi) \in S^*_H M : r_H G^{T^{(j)}(s, \xi)} (s, \xi)  =
 G^{T^{(j)}(s, \xi)} r_H (s, \xi)  \}\right) > 0.  $$ 
Otherwise we say that $H$ is asymmetric with respect to the geodesic flow. 

\end{defin}

The QER theorem we state below   holds for both poly-homogeneous (Kohn-Nirenberg) pseudo-differential
operators as in \cite{HoI-IV}  and also for semi-classical
 pseudo-differential operators on $H$ \cite{Zw}  with essentially the same proof.   To avoid confusion between pseudodifferential operators  on the ambient manifold $M$  and those on $H$,  we denote
 the latter by $Op_H(a)$ where $a \in S^{0}_{cl}(T^*H).$ 
 By Kohn-Nirenberg  pseudo-differential operators
we mean operators  with classical poly-homogeneous symbols   $a(s,\sigma) \in C^{\infty}(T^*H),$ 
 $$ a(s,\sigma) \sim \sum_{k=0}^{\infty} a_{-k}(s,\sigma), \,\,(a_{-k} \; \mbox{ positive homogeneous of order} 
\; -k) $$
 as $|\sigma| \rightarrow \infty$  on $T^* H$ as in \cite{HoI-IV}.  By semi-classical pseudo-differential operators
we mean $h$-quantizations of   semi-classical  symbols  $a \in S^{0,0}(T^*H \times (0,h_0])$
of the form
 $$ a_{h} (s,\sigma) \sim \sum_{k=0}^{\infty} h^k \;a_{-k}(s,\sigma), \,\,(a_{-k} \; \in  S_{1,0}^{0}(T^* H)) $$ as in
\cite{Zw,HZ,TZ}.

We further introduce the zeroth order homogeneous function
\begin{equation}\label{gammaDEF} \gamma(s, y_n, \sigma, \eta_n) =  \frac{|\eta_n|}{\sqrt{|\sigma|^2
+ |\eta_n|^2}} = (1 - \frac{|\sigma|^2}{r^2})^{\half} ,\;\;\;(r^2 = |\sigma|^2 + |\eta_n|^2)  \end{equation}    
on $T^*_H M$ and also denote by
\begin{equation} \label{gammaBH} \gamma_{B^*H} = (1 - |\sigma|^2)^{\half} \end{equation}
its restriction to $S^*_H M = \{r = 1\}$.

For homogeneous pseudo-differential operators, the QER theorem is as follows:

   \begin{theo} \label{phgtheorem} \cite{TZ,TZ2,DZ} Let $(M, g)$ be a compact manifold with ergodic geodesic flow, and let  $H \subset
      M$ be a hypersurface.  Let $\phi_{\lambda_j}; j=1,2,...$ denote the
       $L^{2}$-normalized eigenfunctions of $\Delta_g$.
       If $H$ has a zero measure of microlocal symmetry, then
 there exists a  density-one subset $S$ of ${\mathbb N}$ such that
  for $\lambda_0 >0$ and  $a(s,\sigma) \in S^{0}_{cl}(T^*H)$
$$ \lim_{\lambda_j \rightarrow \infty; j \in S} \langle  Op_H(a) 
 \gamma_H \phi_{\lambda_j}, \gamma_H \phi_{\lambda_j}\rangle_{L^{2}(H)} = \omega(a), $$
 where 
 $$  \omega(a) = \frac{2}{ vol(S^*M) } \int_{B^{*}H}  a_0( s, \sigma )  \,  \gamma^{-1}_{B^*H}(s,\sigma)  \, ds d\sigma. $$

\end{theo}

 Alternatively, one can write  $ \omega(a) = \frac{1}{vol (S^*M)} \int_{S^*_H M} a_0(s, \pi_H(\xi)) d\mu_{L, H} (\xi).$  Note that  $a_0(s, \sigma)$ is bounded but is not
defined for $\sigma = 0$, hence $ a_0(s, \pi_H(\xi))$  is not
defined for $ \xi \in N^* H$ if  $a_0(s,\sigma)$  is homogeneous of order zero on $T^* H$.
The analogous result for semi-classical pseudo-differential operators is:

  \begin{theo} \cite{TZ,TZ2,DZ} \label{sctheorem} Let $(M, g)$ be a compact manifold with ergodic geodesic flow, and let  $H \subset
      M$ be a hypersurface.  
       If $H$ has a zero measure of microlocal symmetry, then
 there exists a  density-one subset $S$ of ${\mathbb N}$ such that
  for $a \in S^{0,0}(T^*H \times [0,h_0)),$
$$ \lim_{h_j \rightarrow 0^+; j \in S} \langle Op_{h_j}(a)
\gamma_H \phi_{h_j}, \gamma_H \phi_{h_j}\rangle_{L^{2}(H)} = \omega(a), $$
 where 
 $$  \omega(a) = \frac{2}{ vol(S^*M) } \int_{B^{*}H}  a_0( s, \sigma )  \,  \gamma^{-1}_{B^*H}(s,\sigma)  \, ds d\sigma.$$
\end{theo}

Examples of asymmetric curves on surfaces in the case where   $(M, g)$ is  a finite area hyperbolic surface are the following:
\begin{itemize}

\item   $H$ is a geodesic
circle; 

\item $H$ is  a closed horocycle  of radius $r < inj(M, g)$, the
injectivity radius.

\item $H$ is a generic closed geodesic or an arc of a generic non-closed geodesic.

\end{itemize}

\section{Critical points}

In this section, we briefly discuss some analogues of \eqref{ID} and \eqref{1} for critical points on 
surfaces.  To be sure, it is not hard to generate many identities; the main problem is to derive information from them.

 We denote the gradient of a function $\phi$  by $\nabla \phi$ and its Hessian by $ \nabla^2 \phi : = \nabla d \phi$,
where $\nabla$ is the Riemannian connection.  We also denote the area form by $d A $ and
 the scalar curvature 
by $K$. 
   The results
are based on unpublished work in progress  of the author. It is often said that measuring critical point
sets and values is much more difficult than measuring nodal sets, and in a sense the identities reflect this
difficulty, and we immediately see one difficulty in that the identities become signed:

\begin{prop} \label{MAINPROP} Suppose that $(M, g)$ is a Riemannian surface, and
that $\phi$ is a Morse eigenfunction with 
 $(\Delta + \lambda^2) \phi = 0$.  Let $V \in C^2(M)$. Then
\begin{equation} \label{IDENTITY2} \begin{array}{lll}     2\pi \sum_{p: d \phi(p) = 0}  \mbox{sign} (\det \nabla^2 \phi(p)) \; V(p)   & =  & 
  2 \lambda^2 \int_M \frac{ \phi }{|\nabla \phi|}  \frac{\nabla V \cdot \nabla \phi}{|\nabla \phi|} d A + 2  \int_M K V dA  \\ && \\ & &  - 
 \int_M (\Delta V)  \log |\nabla \phi|^2 d A. 
 
 \end{array}  \end{equation} 
\end{prop}
Here, $ \mbox{sign} (\det \nabla^2 \phi(p)) = 1$ if $p$ is a local maximum or minimum and $= -1$ if $p$ is a saddle point.  When $V \equiv 1$,  the identity reduces to the Gauss-Bonnet theorem $\int K d A = 2 \pi \chi(M)$ and the Hopf
index formula  $\chi(M) = \sum_{x: \nabla \phi(x) = 0}    \mbox{sign} (\det \nabla^2 \phi(p)) . $ As this indicates,
the main problem with applying the identity to counting critical points is that the left side is an alternating sum over critical points rather than a positive sum. In \cite{Dong} a related
identity using $|\nabla \phi|^2 + \lambda^2 \phi^2$ produced a sum of constant sign over the singular
points of $\phi$, but singular points are  always saddle points of index $-1$ and hence  of constant sign.
 Note that under the Morse assumption, $\log |\nabla \phi|, |\nabla \phi|^{-1} \in L^1(M, dA)$,
so that the right side is a well defined measure integrated against $V$.

We now make some interesting choices of $V$. As mentioned above, (weighted) counting of critical values should
be simpler than weighted counting of critical points. Hence we put $V = f(\phi)$ for smooth $f$. This choice does give cancellation of the `bad factor'
$|\nabla \phi|^{-1}$ and (using that
$\Delta f(\phi) = f''(\phi) |\nabla \phi|^2 - f'(\phi) \lambda^2 \phi$) we get

\begin{maincor}\label{V=fphi} With the assumptions of Proposition \ref{MAINPROP},  if $f \in C^2(\R)$, then
\begin{equation} \label{IDENTITYfphi} \begin{array}{lll}     2\pi \sum_{p: d \phi(p) = 0} \mbox{sign} (\det \nabla^2 \phi(p)) f(\phi (p)) \;   & =  & 
      2 \lambda^2 \int_M  \phi f'(\phi)  d A + 2  \int_M K f(\phi) dA  \\ && \\ & &  - 
 \int_M (f''(\phi) |\nabla \phi|^2 - f'(\phi) \lambda^2 \phi))  \log |\nabla \phi|^2 d A. 
 \end{array}  \end{equation}
\end{maincor}
Of course, this still has the defect that the left side is an oscillating sum, and 
the factor $f(\phi)$ in the sum damps out the critical points in regions of exponential decay. 
To illustrate,  if $f(x) = x$ we get 
\begin{equation} \label{IDENTITYphi} \begin{array}{lll}     2\pi \sum_{p: d \phi(p) = 0} \mbox{sign} (\det \nabla^2 \phi(p)) \phi(p) \;   & =  & 
    2  \int_M K \phi dA  + \lambda^2
 \int_M \phi  \log |\nabla \phi|^2 d A.
 \end{array}  \end{equation}

To highlight the sign issue, we break  up the sum into the sub-sum over maxima/minima and the sub-sum over saddle points, denoting
 the set of local maxima (resp. minima) by $\max$ (resp. $\min$) and the set of saddle points by $Sad$. Of course
we have $\#(\max \cup \min) - \# Sad = \chi(M)$.  Then \eqref{IDENTITYfphi} is equivalent to

\begin{equation} \label{IDENTITYfphib} \begin{array}{lll}     2\pi \sum_{p \in \max \cup \min) } f(\phi (p)) \;   & =  & 
     2\pi \sum_{p \in Sad) } f(\phi (p)) \;  +
      2 \lambda^2 \int_M  \phi f'(\phi)  d A + 2  \int_M K f(\phi) dA  \\ && \\ & &  - 
 \int_M (f''(\phi) |\nabla \phi|^2 - f'(\phi) \lambda^2 \phi))  \log |\nabla \phi|^2 d A. 
 \end{array}  \end{equation}

 We write $\log r = \log_+ r - \log_- r$ where
$\log_+ r= \max\{\log r, 0\}$. We note that on any compact Riemannian manifold, $\log_+ |\nabla \phi|^2 = O(\log \lambda)$ uniformly in $x$
as $\lambda \to \infty$ while $\log_- |\nabla \phi|^2$ can be quite complicated to estimate.  When $f = x^2$ we get,

\begin{equation} \label{IDENTITYphi2} \begin{array}{lll}     2\pi \sum_{p \in \max, \min}  \phi^2 (p) \;   & =  & 
2 \pi \sum_{p \in Sad} \phi(p)^2 \\ && \\ && +  
 4 \lambda^2 +      
2  \int_M (\lambda^2\phi^2 - |\nabla \phi|^2)  \log |\nabla \phi|^2 d A +  2  \int_M K  \phi^2 dA.
 \end{array}  \end{equation}
 Assuming $\phi$ is a Morse eigenfunction, this implies

\begin{equation} \label{IDENTITYphi2cor} \begin{array}{lll}      \sum_{p \in \max, \min}  \phi^2 (p) \;   & \leq  & 
 \sum_{p \in Sad} \phi(p)^2 +   O(\lambda^2 \log \lambda).
 \end{array}  \end{equation}

To get rid of the signs in the sum, we could   choose  $V =   W \det \nabla^2 \phi$, where the determinant is defined by the metric.  Since
$(\mbox{sign} \det \nabla^2 \phi) \det \nabla^2 \phi = |\det \nabla^2 \phi|$ we obtain

\begin{equation} \label{IDENTITYdet} \begin{array}{lll}     2\pi \sum_{p: d \phi(p) = 0} |\det \nabla^2 \phi(p)| W(p) ) \;   & =  & 
  2 \lambda^2 \int_M \frac{ \phi }{|\nabla \phi|}  \frac{\nabla  (W \det \nabla^2 \phi) \cdot \nabla \phi}{|\nabla \phi|} d A + 2  \int_M K  W  \det \nabla^2 \phi dA  \\ && \\ & &  - 
 \int_M (\Delta W  \det \nabla^2 \phi)  \log |\nabla \phi|^2 d A. 
 \end{array}  \end{equation}
But the first term appears to be  difficult to estimate.

\section{Analytic continuation of eigenfunctions for real analytic $(M, g)$}

We now take up the theme mentioned in the introduction of analytically continuing eigenfunctions on
real analytic $(M, g)$ to the complex domain.  In the next sections we apply the analytic continuation to the study of nodal
 of eigenfunctions in the real analytic case.  For background we refer to \cite{LS1,LS2,GS1,GS2,GLS,Z8}.

A real analytic manifold $M$ always possesses a unique
complexification $M_{\C}$ generalizing the complexification of
$\R^m$ as $\C^m$. The complexification is  an open  complex
manifold in  which $M$ embeds $\iota: M \to M_{\C}$  as a totally
real submanifold (Bruhat-Whitney). As examples, we have:

\begin{itemize}

\item  $M = \R^m/\Z^m$ is $M_{\C} = \C^m/\Z^m$.

\item   The unit sphere $S^n$ defined by  $x_1^2 + \cdots +
x_{n+1}^2 = 1$ in $\R^{n+1}$ is complexified as the complex
quadric $S^2_{\C} = \{(z_1, \dots, z_n) \in \C^{n + 1}: z_1^2 +
\cdots + z_{n+1}^2 = 1\}. $

\item  The hyperboloid model of hyperbolic space is the
hypersurface in $\R^{n+1}$ defined by
$$\Hh^n = \{ x_1^2 + \cdots
x_n^2 - x_{n+1}^2 = -1, \;\; x_n > 0\}. $$ Then,
$$H^n_{\C} = \{(z_1, \dots, z_{n+1}) \in \C^{n+1}:  z_1^2 + \cdots
z_n^2 - z_{n+1}^2 = -1\}. $$

\item Any real algebraic subvariety of $\R^m$ has a similar
complexification.

\item Any  Lie group $G$ (or symmetric space) admits a
complexification $G_{\C}$.
\end{itemize}

The Riemannian metric determines a special kind of distance
function on $M_{\C}$  known as a Grauert tube function. It is the  plurisubharmonic function $\sqrt{\rho} =
\sqrt{\rho}_g$ on $M_{\C}$ defined as the unique solution of the
Monge-Amp\`ere equation
$$(\ddbar \sqrt{\rho})^m = \delta_{M_{\R}, dV_g}, \;\; \iota^*
(i \ddbar \rho) = g. $$ Here, $\delta_{M_{\R}, dV_g}$ is the
delta-function on the real $M$ with respect to the volume form
$dV_g$, i.e. $f \to \int_M f dV_g$. In fact,  it is observed in \cite{GS1,GLS} that the Grauert
tube function is obtained from the distance function by setting $\sqrt{\rho}(\zeta) =
i \sqrt{ r^2(\zeta, \bar{\zeta})}$ where $r^2(x, y)$ is the
squared distance function in a neighborhood of the diagonal in $M
\times M$.

One defines the Grauert tubes $M_{\tau} = \{\zeta \in M_{\C}:
\sqrt{\rho}(\zeta) \leq \tau\}$. There exists a maximal $\tau_0$
for which $\sqrt{\rho}$ is well defined, known as the Grauert tube
radius. For $\tau \leq \tau_0$, $M_{\tau}$ is a strictly
pseudo-convex domain in $M_{\C}$.

The complexified exponential map $(x, \xi) \to exp_x i \xi$
defines a diffeomorphism from $B_{\tau}^* M$ to $M_{\tau}$ and
pulls back $\sqrt{\rho}$ to $|\xi|_g$. The one-complex dimensional
null foliation of $\ddbar \sqrt{\rho}$, known as the
`Monge-Amp\`ere' or Riemann foliation,  are the complex curves $t
+ i \tau \to \tau \dot{\gamma}(t)$, where $\gamma$ is a geodesic,
where $\tau > 0$ and where  $\tau \dot{\gamma}(t)$ denotes
multiplication of the tangent vector to $\gamma$ by $\tau$. We
refer to \cite{LS1,GLS,Z8} for further discussion.

\subsection{Poisson operator and analytic Continuation of eigenfunctions}

The half-wave group of $(M, g)$ is the unitary group
$U(t)  = e^{i t \sqrt{\Delta}}$ generated by the square root of the positive Laplacian. Its Schwartz kernel
is a distribution on $\R \times M \times M$ with  the eigenfunction
expansion
\begin{equation} \label{Ut} U(t, x, y) = \sum_{j = 0}^{\infty} e^{i
t  \lambda_j} \phi_{ j}(x) \phi_j(y).  \end{equation}

By the Poisson operator we mean the analytic continuation of$U(t)$ to positive imaginary time,
\begin{equation} \label{POISSON} e^{- \tau \sqrt{\Delta}}  = U(i \tau) .\end{equation}
The eigenfunction expansion then converges absolutely to a real analytic function on $\R_+ \times M \times M$.

Let $A(\tau)$ denote the operator of analytic continuation of a
function on $M$ to the Grauert tube $M_{\tau}$. Since
\begin{equation} U_{\C} (i \tau) \phi_{\lambda} = e^{- \tau \lambda}
\phi_{\lambda}^{\C}, \end{equation} it  is simple to
see that \begin{equation}  \label{ATAU} A(\tau) = U_{\C}(i \tau) e^{\tau \sqrt{\Delta}} \end{equation} where
 $U_{\C}(i
\tau, \zeta, y)$ is the analytic continuation of the Poisson kernel in $x$ to
$M_{\tau}$. In terms of the eigenfunction expansion, one has
\begin{equation} \label{UI} U(i \tau, \zeta, y) = \sum_{j = 0}^{\infty} e^{-
\tau  \lambda_j} \phi_{ j}^{\C} (\zeta) \phi_j(y),\;\;\; (\zeta,
y) \in M_{\epsilon}  \times M.  \end{equation}  This is a very useful observation because $ U_{\C}(i \tau) e^{\tau \sqrt{\Delta}} $
is a Fourier integral operator with complex phase and can be related to the geodesic flow.  
The analytic continuability of
the Poisson operator to $M_{\tau}$  implies that  every
eigenfunction analytically continues to the same Grauert tube.

\subsection{Analytic continuation of the Poisson wave group}  The analytic continuation of the Possion-wave
kernel to $M_{\tau}$ in the $x$ variable is discussed in detail in \cite{Z8} and ultimately derives from the
analysis by Hadamard of his parametrix construction.  We only briefly discuss  it here and refer to
\cite{Z8} for further details.  In the case 
of Euclidean $\R^n$ and its wave kernel  $U(t, x, y) =
\int_{\R^n} e^{i t |\xi|} e^{i \langle \xi, x - y \rangle} d\xi$
which  analytically continues  to $t + i \tau, \zeta = x + i p \in
\C_+ \times \C^n$ as the integral
$$U_{\C} (t + i \tau , x + i p , y) = \int_{\R^n} e^{i (t + i \tau)  |\xi|} e^{i \langle \xi, x + i p - y
\rangle} d\xi. $$ The integral clearly converges absolutely for
$|p| < \tau.$

Exact formulae of this kind exist for $S^m$ and $\H^m$. For a
general real analytic Riemannian manifold, there exists an
oscillatry integral expression for the wave kernel of the form,
\begin{equation} \label{PARAONE} U (t, x, y) = \int_{T^*_y M} e^{i
t |\xi|_{g_y} } e^{i \langle \xi, \exp_y^{-1} (x) \rangle} A(t, x,
y, \xi) d\xi
\end{equation} where $A(t, x, y, \xi)$ is a polyhomogeneous amplitude of
order $0$.  The
holomorphic extension of (\ref{PARAONE}) to the Grauert tube
$|\zeta| < \tau$ in $x$ at time $t = i \tau$ then has the form
\begin{equation} \label{CXPARAONE} U_{\C} (i \tau,
\zeta, y) = \int_{T^*_y} e^{- \tau  |\xi|_{g_y} } e^{i \langle
\xi, \exp_y^{-1} (\zeta) \rangle} A(t, \zeta, y, \xi) d\xi
\;\;\;(\zeta = x + i p).
\end{equation}

\subsection{Analytic continuation of eigenfunctions}

Thus, a
function   $f \in C^{\infty}(M)$ has a holomorphic extension to
the closed  tube $\sqrt{\rho}(\zeta) \leq \tau$ if and only if $f
\in Dom(e^{\tau \sqrt{\Delta}}), $ where $e^{\tau \sqrt{\Delta}}$
is the backwards `heat operator' generated by $\sqrt{\Delta}$
(rather than $\Delta$).
 That is, $f = \sum_{n = 0}^{\infty}
a_n \phi_{\lambda_n}$ admits an analytic continuation to the open
Grauert tube $M_{\tau}$ if and only if $f$ is in the domain of
$e^{\tau \sqrt{\Delta}}$, i.e. if $\sum_n |a_n|^2  e^{2 \tau
\lambda_n} < \infty $. Indeed, the analytic continuation is
$U_{\C}(i \tau) e^{\tau \sqrt{\Delta}} f$. The subtlety is in the
nature of the restriction to the boundary of the maximal  Grauert
tube.

This result generalizes one of the classical Paley-Wiener theorems
to real analytic Riemannian manifolds \cite{Bou,GS2}. In the
simplest case of $M = S^1$, $f \sim \sum_{n \in \Z} a_n e^{in
\theta} \in C^{\omega}(S^1)$ is the restriction of a holomorphic
function $F \sim \sum_{n \in \Z} a_n z^n $ on the annulus
$S^1_{\tau} = \{|\log |z| | < \tau \}$ and with $F \in
L^2(\partial S^1_{\tau})$ if and only if $\sum_n |\hat{f}(n)|^2 \;
e^{2 |n| \tau} < \infty$. The case of $\R^m$ is more complicated
since it is non-compact. We are mainly concerned with compact
manifolds and so the complications are not very relevant here. But
we recall that one of the classical Paley-Wiener theorems states
that a
 real analytic function $f$ on $\R^n$ is the restriction of a holomorphic
 function on the closed  tube $|\Im \zeta| \leq \tau$ which satisfies
$\int_{\R^m} |F(x + i \xi)|^2 dx \leq C$ for $\xi \leq  \tau$ if
and only if $\hat{f} e^{\tau |\Im \zeta|} \in L^2(\R^n)$.

Let us consider examples of holomorphic continuations of
eigenfunctions:

\begin{itemize}

\item On  the flat torus $\R^m/\Z^m,$   the real eigenfunctions
are $\cos \langle k, x \rangle, \sin \langle k, x \rangle$ with $k
\in 2 \pi \Z^m.$ The complexified torus is $\C^m/\Z^m$ and the
complexified eigenfunctions are $\cos \langle k, \zeta \rangle,
\sin \langle k, \zeta \rangle$ with $\zeta  = x + i \xi.$

\item On the unit sphere $S^m$, eigenfunctions are restrictions of
homogeneous harmonic functions on $\R^{m + 1}$. The latter extend
holomorphically to holomorphic harmonic polynomials on $\C^{m +
1}$ and restrict to holomorphic function on $S^m_{\C}$.

\item On $\H^m$, one may use the hyperbolic plane waves $e^{ (i
\lambda + 1) \langle z, b \rangle}$, where  $\langle z, b \rangle$
is the (signed) hyperbolic distance of the horocycle passing
through $z$ and $b$ to $0$. They  may be holomorphically extended
to the maximal tube of radius $\pi/4$.

\item
 On  compact hyperbolic quotients ${\bf
H}^m/\Gamma$, eigenfunctions can be then represented by Helgason's
generalized Poisson integral formula \cite{H},
$$\phi_{\lambda}(z) = \int_B e^{(i \lambda + 1)\langle z, b
\rangle } dT_{\lambda}(b). $$ Here, $z \in D$ (the unit disc), $B
=
\partial D$, and $dT_{\lambda}
\in \dcal'(B)$ is the boundary value of $\phi_{\lambda}$, taken in
a weak sense along circles centered at the origin $0$.
 To  analytically continue $\phi_{\lambda}$ it suffices  to
analytically continue $\langle z, b\rangle. $ Writing the latter
as  $\langle \zeta, b \rangle,  $ we have:
\begin{equation} \label{HEL} \phi_{\lambda}^{\C} (\zeta) = \int_B e^{(i \lambda +
1)\langle \zeta, b \rangle } dT_{\lambda}(b). \end{equation}

\end{itemize}

\subsection{Complexified spectral projections}

The next step is  to  holomorphically extend the spectral projectors
$d\Pi_{[0, \lambda]}(x,y) = \sum_j \delta(\lambda - \lambda_j)
\phi_j(x) \phi_j(y) $ of $\sqrt{\Delta}$.
 The  complexified
diagonal  spectral projections
 measure is defined by
 \begin{equation} d_{\lambda} \Pi_{[0, \lambda]}^{\C}(\zeta, \bar{\zeta}) = \sum_j \delta(\lambda -
 \lambda_j) |\phi_j^{\C}(\zeta)|^2. \end{equation}
 Henceforth, we generally omit the superscript and write the
 kernel as $\Pi_{[0, \lambda]}^{\C}(\zeta, \bar{\zeta})$.
 This kernel is not a tempered distribution due to the exponential
 growth of $|\phi_j^{\C}(\zeta)|^2$. Since many asymptotic techniques
 assume spectral functions are of polynomial growth,  we simultaneously
 consider the damped spectral projections measure
  \begin{equation} \label{SPPROJDAMPED} d_{\lambda} P_{[0, \lambda]
  }^{\tau}(\zeta, \bar{\zeta}) = \sum_j \delta(\lambda -
 \lambda_j) e^{- 2 \tau \lambda_j} |\phi_j^{\C}(\zeta)|^2, \end{equation}
 which   is a temperate distribution as long as $\sqrt{\rho}(\zeta)
 \leq \tau. $ When we set $\tau = \sqrt{\rho}(\zeta)$ we omit the
 $\tau$ and put
  \begin{equation} \label{SPPROJDAMPEDz} d_{\lambda} P_{[0, \lambda]
  }(\zeta, \bar{\zeta}) = \sum_j \delta(\lambda -
 \lambda_j) e^{- 2 \sqrt{\rho}(\zeta) \lambda_j} |\phi_j^{\C}(\zeta)|^2. \end{equation}

The integral of the spectral measure over an interval $I$  gives
$$\Pi_{I}(x,y) = \sum_{j: \lambda_j \in I} \phi_j(x) \phi_j(y).$$
Its complexification gives the kernel (\ref{CXSPECPROJ}) along the
diagonal,
\begin{equation}\label{CXSP} \Pi_{I}(\zeta, \bar{\zeta}) =
 \sum_{j: \lambda_j \in I}
|\phi_j^{\C}(\zeta)|^2,  \end{equation}  and the integral of
(\ref{SPPROJDAMPED}) gives its temperate version
\begin{equation}\label{CXDSP}  P^{\tau}_{I}(\zeta, \bar{\zeta}) =
 \sum_{j: \lambda_j \in I}  e^{- 2 \tau \lambda_j}
|\phi_j^{\C}(\zeta)|^2,
\end{equation}
or in the crucial case of $\tau = \sqrt{\rho}(\zeta)$,
\begin{equation}\label{CXDSPa}  P_{I}(\zeta, \bar{\zeta}) =
 \sum_{j: \lambda_j \in I}  e^{- 2 \sqrt{\rho}(\zeta)\lambda_j}
|\phi_j^{\C}(\zeta)|^2,
\end{equation}

\subsection{Poisson operator as a complex Fourier integral
operator}

The damped spectral projection measure  $d_{\lambda} \; P_{[0,
\lambda]}^{\tau}(\zeta, \bar{\zeta})$ (\ref{SPPROJDAMPED}) is dual
under the real Fourier transform in the $t$ variable to the
restriction
\begin{equation} \label{CXWVGP} U (t + 2 i \tau, \zeta, \bar{\zeta}) = \sum_j
e^{(- 2 \tau + i t) \lambda_j} |\phi_j^{\C}(\zeta)|^2
\end{equation}  to the
anti-diagonal of the mixed Poisson-wave group. The
adjoint of the Poisson kernel $U(i \tau, x, y)$ also admits an
anti-holomorphic extension in the $y$ variable. The sum
(\ref{CXWVGP}) are the diagonal values of the complexified wave
kernel
\begin{equation} \label{EFORM}
\begin{array}{lll} U (t + 2 i \tau, \zeta, \bar{\zeta}') &  = &
\int_M  U (t + i \tau, \zeta, y)  E(i \tau, y, \bar{\zeta}'
) dV_g(x)\\
&& \\
&&  = \sum_j  e^{(- 2 \tau + i t) \lambda_j} \phi_j^{\C}(\zeta)
\overline{\phi_j^{\C}(\zeta')}.
\end{array}
\end{equation} We obtain
(\ref{EFORM}) by  orthogonality of the real eigenfunctions on $M$.

Since $U(t + 2 i \tau, \zeta, y)$ takes its values in the CR
holomorphic functions on $\partial M _{\tau}$, we consider the
Sobolev spaces $\ocal^{s + \frac{n-1}{4}}(\partial M _{\tau})$ of
CR holomorphic functions on the boundaries of the strictly
pseudo-convex domains $M_{\epsilon}$, i.e.
$${\mathcal O}^{s + \frac{m-1}{4}}(\partial M_{\tau}) =
W^{s + \frac{m-1}{4}}(\partial M_{\tau}) \cap \ocal (\partial
M_{\tau}), $$ where $W_s$  is the $s$th Sobolev space and where $
\ocal (\partial M_{\epsilon})$ is the space of boundary values of
holomorphic functions. The inner product on $\ocal^0 (\partial M
_{\tau} )$ is with respect to the Liouville measure
\begin{equation} \label{LIOUVILLEa} d\mu_{\tau} = (i \ddbar
\sqrt{\rho})^{m-1} \wedge d^c \sqrt{\rho}. \end{equation}

We then regard  $U(t + i \tau, \zeta, y)$ as the kernel of an
operator from $L^2(M) \to \ocal^0(\partial M_{\tau})$. It equals
its composition $ \Pi _{\tau} \circ U (t + i \tau)$ with the
\szego projector
$$ \Pi_{\tau} : L^2(\partial M_{\tau}) \to \ocal^0(
\partial M_{\tau})$$  for the tube $
M_{\tau}$, i.e.   the orthogonal projection onto boundary values
of holomorphic functions in the tube.

 This is a useful expression
for  the complexified wave kernel, because  $\tilde{\Pi}_{\tau}$
is a complex Fourier integral operator with a small wave front
relation. More precisely,  the real points of its  canonical
relation form  the graph $\Delta_{\Sigma}$ of the identity map on
the symplectic one
 $\Sigma_{\tau}
\subset T^*
\partial M_{\tau}$ spanned by the real one-form $d^c \rho$,
i.e. \begin{equation} \label{SIGMATAU} \Sigma_{\tau} = \{(\zeta; r
d^c \rho (\zeta)) ,\;\;\; \zeta \in \partial M_{\tau},\; r > 0 \}
 \subset T^* (\partial M_{\tau}).\;\;  \end{equation}
 We note that for each $\tau,$ there exists a symplectic equivalence $ \Sigma_{\tau} \simeq
 T^*M$ by the map $(\zeta, r d^c \rho(\zeta) )  \to
 (E_{\C}^{-1}(\zeta), r \alpha)$, where $\alpha = \xi \cdot dx$ is
 the action form (cf. \cite{GS2}).

The following result was first stated by  Boutet de Monvel (for
more details, see also \cite{GS2,Z8}).

\begin{theo}\label{BOUFIO}  \cite{Bou, GS2} \label{BDM} $\Pi_{\epsilon} \circ U (i \epsilon): L^2(M)
\to \ocal(\partial M_{\epsilon})$ is a  complex Fourier integral
operator of order $- \frac{m-1}{4}$  associated to the canonical
relation
$$\Gamma = \{(y, \eta, \iota_{\epsilon} (y, \eta) \} \subset T^*M \times \Sigma_{\epsilon}.$$
Moreover, for any $s$,
$$\Pi_{\epsilon} \circ U (i \epsilon): W^s(M) \to {\mathcal O}^{s +
\frac{m-1}{4}}(\partial  M_{\epsilon})$$ is a continuous
isomorphism.
\end{theo}

In \cite{Z8} we give the following sharpening of the sup norm estimates of \cite{Bou,GLS}:

\begin{prop} \label{PW} Suppose  $(M, g)$ is real analytic.  Then
$$ \sup_{\zeta \in M_{\tau}} |\phi^{\C}_{\lambda}(\zeta)| \leq C
  \lambda^{\frac{m+1}{2}} e^{\tau \lambda}, \;\;\;\; \sup_{\zeta \in M_{\tau}} |\frac{\partial \phi^{\C}_{\lambda}(\zeta)}{\partial \zeta_j}| \leq C
  \lambda^{\frac{m+3}{2}} e^{\tau \lambda}
$$
\end{prop}
The proof follows  easily from the fact that the complexified
Poisson kernel is a complex Fourier integral operator of finite
order. The estimates can be improved further.

\subsection{Maximal plurisubharmonic functions and growth of $\phi_{\lambda}^{\C}$}

In \cite{Z8}, we discussed   analogues in the setting of Gruaert tubes for
the basic notions of pluripotential theory on domains in $\C^m$.
Of relevance here is that  the Grauert tube function $\sqrt{\rho}$ is the
analogue of the pluri-complex Green's function. 
We recall that the
maximal PSH function (or pluri-complex Green's function) relative
to a subset $E \subset \Omega$ is defined by
$$V_{E}(\zeta) = \sup\{u(z): u \in PSH(\Omega), u|_{E}  \leq 0, u
|_{\partial \Omega} \leq 1\}. $$

On a real analytic Riemannian manifold, the natural analogue of
$\pcal^N$ is the space
$$\hcal^{\lambda} = \{p =  \sum_{j: \lambda_j \leq \lambda} a_j
\phi_{\lambda_j}, \;\; a_1, \dots, a_{N(\lambda)} \in \R  \} $$
spanned by eigenfunctions with frequencies $\leq \lambda$. Rather
than using the sup norm, it is convenient  to work with $L^2$
based norms than sup norms, and so we define
$$ \hcal^{\lambda}_M = \{p =  \sum_{j: \lambda_j \leq \lambda} a_j
\phi_{\lambda_j}, \;\;||p||_{L^2(M)} =  \sum_{j = 1}^{N(\lambda)}
|a_j|^2 = 1 \}. $$ We define the   $\lambda$-Siciak extremal
function by
 $$ \Phi_M^{\lambda} (z) = \sup \{|\psi(z)|^{1/\lambda} \colon
\psi \in \hcal_{\lambda};   \|\psi \|_M \le 1 \},  $$ and the
extremal function by
$$\Phi_M(z) = \sup_{\lambda} \Phi_M^{\lambda}(z). $$

The extremal PSH function is defined by
$$V_g(\zeta; \tau) =  \sup\{u(z): u \in PSH(M_{\tau}), u|_{M}  \leq 0, u
|_{\partial M_{\tau}} \leq \tau\}. $$ In \cite{Z8} we proved that  $V_g =
\sqrt{\rho}$ and that
\begin{equation} \label{SICIAK} \Phi_M = V_g. \end{equation}
The proof is based on the properties of \eqref{CXSP}. 
By using a Bernstein-Walsh
inequality
$$\frac{1}{N(\lambda)} \leq \frac{\Pi_{[0, \lambda]}(\zeta,
\bar{\zeta})}{\Phi_M^{\lambda}(\zeta)^2} \leq C N(\lambda)\;
e^{\epsilon N(\lambda)}, $$ it is not hard to show that
\begin{equation} \Phi_M(z) = \lim_{\lambda \to \infty} \frac{1}{\lambda} \log  \Pi_{[0,
\lambda}(\zeta, \bar{\zeta}). \end{equation}
 To evaluate the logarithm, one can show that the kernel is essentially $e^{\lambda \sqrt{\rho}}$ times
the temperate projection defined by the Poisson operator,
\begin{equation}\label{CXDSPb}  P_{[0, \lambda]}(\zeta, \bar{\zeta}) =
 \sum_{j: \lambda_j \in [0, \lambda]}  e^{- 2 \sqrt{\rho}(\zeta)\lambda_j}
|\phi_j^{\C}(\zeta)|^2.
\end{equation}
The equality (\ref{SICIAK}) follows from the fact that
 $\lim_{\lambda \to \infty} \frac{1}{\lambda} \log   P_{[0, \lambda]}(\zeta, \bar{\zeta}) = 0$.

We now return to nodal sets, where we will see the same extremal functions arise.

\section{\label{PLANEDOMAIN} Counting nodal lines which touch the boundary in analytic
plane domains}

It is often possible to obtain more refined results on nodal sets
by studying their intersections with some fixed (and often
special) hypersurface. This has been most successful in dimension
two.  In this section, we review the results of \cite{TZ} giving upper bounds
on the number of intersections of the nodal set with the boundary of an analytic (or more
generally piecewise analytic) 
plane domain.  One may expect that the results of
this section can also be generalized to higher dimensions by measuring codimension two nodal hypersurface volumes within
the boundary. 

Thus we would like to  count the number of 
nodal lines (i.e. components of the nodal set)  which touch the boundary. Here we assume that $0$
is a regular value so that components of the nodal set are either loops in the interior (closed nodal
loops)  or curves
which touch the boundary in two points (open nodal lines).  It is known that  for generic piecewise analytic
plane domains, zero is a regular value of all the eigenfunctions
$\phi_{\lambda_j}$, i.e. $\nabla \phi_{\lambda_j} \not= 0$ on
$Z_{\phi_{\lambda_j}}$ \cite{U}; we then call the nodal set
regular.
Since the boundary lies in the nodal set for Dirichlet boundary
conditions, we remove it from the nodal set before counting
components. Henceforth, the number of components of the nodal set
in the Dirichlet case means the number of components of
$Z_{\phi_{\lambda_j}} \backslash \partial \Omega.$

In the following, and henceforth, $C_{\Omega} > 0$ denotes a
positive constant depending only on the domain $\Omega$.

\begin{theo} \label{COR} Let $\Omega$ be a piecewise analytic domain
and  let  $n_{\partial \Omega}(\lambda_j)$ be the number of
components of the nodal set of the $j$th Neumann or Dirichlet
eigenfunction which intersect $\partial \Omega$.  Then there
exists  $C_{\Omega}$   such that $n_{\partial \Omega}(\lambda_j)
\leq C_{\Omega} \lambda_j.$
\end{theo}

By a piecewise
analytic domain $\Omega^2 \subset \R^2$,  we mean a
compact domain with piecewise analytic boundary, i.e. $\partial
\Omega$ is a union of a finite number of piecewise analytic curves
which intersect only at their common endpoints.
Such domains are often studied as archtypes of domains with ergodic billiards
and quantum chaotic eigenfunctions, in particular the 
 Bunimovich stadium or Sinai billiard. Their nodal sets have been the subject of a
number of numerical studies (e.g. \cite{BGS,FGS}).  

In general, there does not exist a non-trivial lower bound for the number of components touching the boundary. 
E.g. in a disc, the zero sets of the eigenfunctions are unions of circles concentric with the origin and spokes emanating
from the center. Only the spokes intersect the boundary and their number reflects the angular momentum rather
than the eigenvalue of the eigenfunction. But
 we  conjecture that for  piecewise analytic domains with
ergodic billiards, the
  the number  of complex zeros of $\phi_{\lambda_j}^{\C}|_{\partial \Omega_{\C}}$
    is bounded below by  $ C_{\Omega} \lambda_j$.  We discuss work in progress on this conjecture in
\S \ref{NODALGEOS}.

In comparison to the order $O(\lambda_j)$ of the number of
boundary nodal points,  the total number of connected components
of $Z_{\phi_{\lambda_j}}$ has the upper bound $O(\lambda_j^2)$
by the Courant nodal domain theorem. It is not known in general whether
the Courant upper bound is achieved, but we expect that it is often achieved
in order of magnitude. In   \cite{NS}  it is proved that the average number of
 nodal components of a random spherical harmonic is of
 order of magnitude $\lambda_j^2$.  Thus, the number of components touching the
boundary is one order of magnitude below the total number of components.

\subsection{Boundary critical points}

The article \cite{TZ} also contains  a similar estimate on the number  of critical
points of $\phi_{\lambda_j}$ which occur on the boundary. We
denote the boundary critical set by
 $$\ccal_{\phi_{\lambda_j}}= \{q \in \partial
\Omega:  (d \phi_{\lambda_j}) (q) = 0\}. $$ In the case of Neumann
eigenfunctions,  $q \in \ccal_{ \phi_{\lambda_j}} \iff d
(\phi_{\lambda_j} |_{\partial \Omega}(q)) = 0 $ since the normal
derivative automatically equals zero on the boundary, while in the
Dirichlet case $q \in \ccal_{\phi_{\lambda_j}} \iff
\partial_{\nu} \phi_{\lambda_j}(q)= 0$ since the boundary is a level
set.

We observe that radial eigenfunctions on the disc are constant on the
boundary; thus, boundary critical point sets need not be isolated.
We therefore impose a non-degeneracy condition on the
tangential  derivative $\partial_t (\phi_{\lambda_j} |_{\partial
\Omega})$ to ensure that its zeros are isolated and can be counted.
We  say that the Neumann
problem for a bounded domain has the  {\it asymptotic Schiffer
property} if there exists $C
> 0$ such that, for all Neumann eigenfunctions $\phi_{\lambda_j}$ with sufficiently large $\lambda_j$,
   \begin{equation} \label{NEUND}
\frac{\|\partial_t \phi_{\lambda_j}\|_{L^{2} (\partial
\Omega)}}{\|\phi_{\lambda_j}\|_{L^{2}(\partial \Omega)}}
      \geq e^{- C \lambda_j}. \end{equation}
     Here, $\partial_t$ is the unit tangential derivative, and the $L^2$ norms refer to the restrictions of the
     eigenfunction to  $\partial \Omega$.

  \begin{theo}\label{CRITSb}  Let $\Omega \subset \R^2$ be piecewise real analytic.
  Suppose that $\phi_{\lambda_j}|_{\partial \Omega}$ satisfies the asymptotic Schiffer  condition (\ref{NEUND})  in the Neumann case.
   Then the
   number of $n_{\mbox{crit}}(\lambda_j) = \# \ccal_{\phi_{\lambda_j} }$ of
   critical points of a Neumann or Dirichlet eigenfunction $\phi_{\lambda_j}$
which lie  on $\partial \Omega$ satisfies  $\nj \leq C_{\Omega}
\lambda_j$ for some $C_{\Omega} > 0$
\end{theo}

   In the case of Dirichlet eigenfunctions, endpoints of open
nodal lines are always boundary critical points, since they must
be singular points of $\phi_{\lambda_j}$.  Hence, an upper bound
for $\nj$ also gives an upper bound for the number of open nodal
lines.

\begin{cor}\label{BNPD}  Suppose that $\Omega \subset \R^2$ is a piecewise real analytic  plane domain.
Let  $n_{\partial \Omega}(\lambda_j)$ be the number of open nodal
lines of the $j$th Dirichlet eigenfunction, i.e. connected
components of $\{\phi_{\lambda_j} = 0\} \subset \Omega^o$ whose
closure  intersects $\partial \Omega$. Then there exists
$C_{\Omega}> 0$  such that  $n_{\partial \Omega}(\lambda_j) \leq
C_{\Omega}  \lambda_j.$
\end{cor}

There does not exist a  non-trivial lower bound
on the number of interior critical points \cite{JN}. 

\subsection{Proof by analytic continuation}

For the Neumann problem, the boundary nodal points are the same as
the zeros of the boundary values $\phi_{\lambda_j} |_{\partial
\Omega}$ of the eigenfunctions. The number of boundary nodal
points  is thus twice the number of open nodal lines. Hence in the
Neumann case, Theorem \ref{COR} follows from:

\begin{theo}\label{BNP}  Suppose that $\Omega \subset \R^2$ is a piecewise real analytic  plane domain.
 Then the number $n(\lambda_j) = \# Z_{\phi_{\lambda_j}} \cap \partial
 \Omega$ of zeros of the boundary values  $\phi_{\lambda_j} |_{\partial
\Omega}$ of the $j$th Neumann  eigenfunction satisfies
$n(\lambda_j) \leq C_{\Omega}
 \lambda_j$, for some $C_{\Omega} > 0$.
\end{theo}
This is a more precise version of   Theorem \ref{COR}  since it does not assume that $0$ is
a regular value. 
In keeping with the theme of this survey, we prove Theorem \ref{BNP}  by analytically
continuing the boundary values of the eigenfunctions and counting  {\it complex zeros and critical points} of analytic
continuations of Cauchy data of eigenfunctions. When $\partial
\Omega \in C^{\omega}$,  the eigenfunctions can be holomorphically
continued to an open tube domain  in $\C^2$ projecting over an
open neighborhood $W$ in $\R^2$ of $\Omega$ which is independent
of the eigenvalue.  We denote by $\Omega_{\C} \subset \C^2$ the
points $\zeta = x + i \xi \in \C^2 $ with $x \in \Omega$.  Then
$\phi_{\lambda_j}(x)$ extends to a holomorphic function
$\phi_{\lambda_j}^{\C}(\zeta)$ where $x \in W$ and where $ |\xi|
\leq \epsilon_0$ for some $\epsilon_0 > 0$.  

Assuming $\partial \Omega$ real analytic, we   define the
(interior) complex nodal set by
$$Z_{\phi_{\lambda_j}}^{\C} = \{\zeta \in  \Omega_{\C}:
\phi_{\lambda_j}^{\C}(\zeta) = 0 \},  $$ and the (interior)
complex critical point set by
$$\ccal_{\phi_{\lambda_j} }^{\C} = \{\zeta \in  \Omega_{\C}: d \phi_{\lambda_j}^{\C} (\zeta) =
0\}.$$

\begin{theo} \label{mainthm} Suppose that $\Omega \subset \R^2$ is a piecewise real analytic  plane
domain, and denote by $(\partial \Omega)_{\C}$ the union of the
complexifications of its real analytic boundary components.

\begin{enumerate}

\item  Let
  $n(\lambda_j, \partial \Omega_{\C} ) = \# Z_{\phi_{\lambda_j}}^{\partial \Omega_{\C}}$ be the number
of complex zeros on the complex boundary. 
 Then there exists a constant $C_{\Omega} > 0$ independent of the radius of
  $(\partial \Omega)_{\C}$ such that
$n(\lambda_j, \partial \Omega_{\C})
 \leq C_{\Omega} \lambda_j. $

 \item Suppose that the Neumann eigenfunctions satisfy (\ref{NEUND}) and
  let $n_{\mbox{crit}}(\lambda_j, \partial \Omega_{\C}) = \# \ccal_{\phi_{\lambda_j} }^{\partial
 \Omega_{\C}}$.
 Then there exists $C_{\Omega} > 0$ independent of the radius of
  $(\partial \Omega)_{\C}$ such that  $n_{\mbox{crit}}(\lambda_j, \partial \Omega_{\C})
 \leq C_{\Omega} \lambda_j. $

\end{enumerate}

\end{theo}

The  theorems on real nodal lines and critical points  follow from
the fact that real zeros and critical points are also complex
zeros and critical points, hence
\begin{equation} n(\lambda_j) \leq n(\lambda_j, \partial \Omega_{\C} );  \;\;\;\; \nj \leq
n_{\mbox{crit}}(\lambda_j, \partial \Omega_{\C}). \end{equation}
All of the results are sharp, and are already obtained for certain
sequences of eigenfunctions  on a disc (see \S \ref{EXAMPLES}). If
the condition (\ref{NEUND}) is not satisfied, the boundary value
of $\phi_{\lambda_j}$ must equal a constant $C_j$ modulo an error
of the form $o(e^{- C \lambda_j})$.  We conjecture that this
forces the boundary values to be constant. 

The
method of proof of Theorem \ref{mainthm} generalizes from
$\partial \Omega$ to a rather large class of  real analytic curves
$C \subset \Omega$, even when $\partial \Omega$ is not real
analytic. Let us call a real analytic curve $C$ a {\it good} curve
if there exists a constant $a > 0$ so that  for all $\lambda_j$
sufficiently large,
 \begin{equation}
\label{GOOD} \frac{\|\phi_{\lambda_j}\|_{L^{2} (\partial
\Omega)}}{\|\phi_{\lambda_j}\|_{L^{2}(C)}}
     \leq e^{a \lambda_j} .\end{equation}
     Here, the $L^2$ norms refer to the restrictions of the
     eigenfunction to $C$ and to $\partial \Omega$.
The following result deals with the case where $C \subset
\partial \Omega$ is an {\em interior} real-analytic
   curve. The  real curve $C$ may then
be holomorphically continued to a complex curve  $C_{\C} \subset
\C^2$ obtained by analytically continuing a real analytic
parametrization of $C$.

   \begin{theo}\label{CNPthm}
 Suppose that $\Omega \subset \R^2$ is a $C^{\infty}$ plane
domain, and let $C \subset \Omega$ be a good  interior real
analytic curve in the sense of (\ref{GOOD}). Let
  $n(\lambda_j, C) = \# Z_{\phi_{\lambda_j} } \cap C $ be the number of intersection points of
  the nodal set of the $j$-th Neumann (or Dirichlet) eigenfunction  with $C$.  Then there exists
  $A_{C, \Omega} > 0$ depending only on $C, \Omega$ such that  $n(\lambda_j, C)
  \leq A_{C, \Omega}
 \lambda_j$.

 \end{theo}

A recent paper of J. Jung shows that many natural curves in the hyperbolic plane are `good' \cite{JJ}.

\subsection{Application to Pleijel's conjecture}

We also note an interesting application due to I. Polterovich
\cite{Po} of Theorem \ref{COR} to an old conjecture of A. Pleijel
regarding Courant's nodal domain theorem, which says that the
number $n_k$ of nodal domains (components of  $\Omega \backslash
Z_{\phi_{\lambda_k}}$)
 of the $k$th eigenfunction  satisfies $n_k \leq k$.  Pleijel \cite{P} improved this result for Dirichlet eigefunctions
 of plane domains:
For any plane domain with Dirichlet boundary conditions,
$\limsup_{k \to \infty} \frac{n_k}{k} \leq \frac{4}{j_1^2} \simeq
0. 691...$, where $j_1$ is the first zero of the $J_0$ Bessel
function.  He conjectured that the same result should be true for
a free membrane, i.e. for Neumann boundary conditions. This was
recently proved in the real analytic case
 by I. Polterovich \cite{Po}. His argument is roughly the following: Pleijel's original argument applies to
 all nodal domains which do not touch the boundary, since the eigenfunction is a Dirichlet eigenfunction in such
  a nodal domain. The argument does not apply to nodal domains which touch the boundary, but by Theorem \ref{COR} the number
  of such domains is negligible for the Pleijel bound.

\section{\label{ERGODICNODAL} Equidistribution of complex nodal sets of real ergodic
eigenfunctions on analytic $(M, g)$ with ergodic geodesic flow}

We now consider global results when hypotheses are made on the
dynamics of the geodesic flow.
 Use of the global wave operator brings into
play the relation between the geodesic flow and the complexified
eigenfunctions, and this allows one to prove gobal results on
nodal hypersurfaces that reflect the dynamics of the geodesic
flow. In some cases, one can determine not just the volume, but
the limit distribution of complex nodal hypersurfaces. Since we have discussed 
this result elsewhere \cite{Z6} we only briefly review it here. 

The complex nodal hypersurface of an eigenfunction is   defined by
\begin{equation} Z_{\phi_{\lambda}^{\C}} = \{\zeta \in
B^*_{\epsilon_0} M: \phi_{\lambda}^{\C}(\zeta) = 0 \}.
\end{equation}
There exists  a natural current of integration over the nodal
hypersurface in any ball bundle $B^*_{\epsilon} M$ with $\epsilon
< \epsilon_0$ , given by
\begin{equation}\label{ZDEF}  \langle [Z_{\phi_{\lambda}^{\C}}] , \phi \rangle =  \frac{i}{2 \pi} \int_{B^*_{\epsilon} M} \ddbar \log
|\phi_{\lambda}^{\C}|^2 \wedge \phi =
\int_{Z_{\phi_{\lambda}^{\C}} } \phi,\;\;\; \phi \in \dcal^{ (m-1,
m-1)} (B^*_{\epsilon} M). \end{equation} In the second equality we
used the Poincar\'e-Lelong formula. The notation $\dcal^{ (m-1,
m-1)} (B^*_{\epsilon} M)$ stands for smooth test $(m-1,
m-1)$-forms with support in $B^*_{\epsilon} M.$

The nodal hypersurface $Z_{\phi_{\lambda}^{\C}}$ also carries a
natural volume form $|Z_{\phi_{\lambda}^{\C}}|$ as a complex
hypersurface in a K\"ahler manifold. By Wirtinger's formula, it
equals the restriction of $\frac{\omega_g^{m-1}}{(m - 1)!}$ to
$Z_{\phi_{\lambda}^{\C}}$. Hence, one can regard
$Z_{\phi_{\lambda}^{\C}}$ as defining  the  measure
\begin{equation} \langle |Z_{\phi_{\lambda}^{\C}}| , \phi \rangle
= \int_{Z_{\phi_{\lambda}^{\C}} } \phi \frac{\omega_g^{m-1}}{(m -
1)!},\;\;\; \phi \in C(B^*_{\epsilon} M).
\end{equation}
 We prefer to state results in terms of the
current $[Z_{\phi_{\lambda}^{\C}}]$ since it carries more
information.

\begin{theo}\label{ZERO}  Let $(M, g)$ be  real analytic, and let $\{\phi_{j_k}\}$ denote a quantum ergodic sequence
of eigenfunctions of its Laplacian $\Delta$.  Let
$(B^*_{\epsilon_0} M, J)$ be the maximal Grauert tube around $M$
with complex structure $J_g$ adapted to $g$. Let $\epsilon <
\epsilon_0$. Then:
$$\frac{1}{\lambda_{j_k}} [Z_{\phi_{j_k}^{\C}}] \to  \frac{i}{ \pi} \ddbar \sqrt{\rho}\;\;
\mbox{weakly in }\;\;  \dcal^{' (1,1)} (B^*_{\epsilon} M), $$ in
the sense that,   for any continuous test form $\psi \in \dcal^{
(m-1, m-1)}(B^*_{\epsilon} M)$, we have
$$\frac{1}{\lambda_{j_k}} \int_{Z_{\phi_{j_k}^{\C}}} \psi \to
 \frac{i}{ \pi} \int_{B^*_{\epsilon} M} \psi \wedge \ddbar
\sqrt{\rho}. $$ Equivalently, for any  $\phi \in C(B^*_{\epsilon}
M)$,
$$\frac{1}{\lambda_{j_k}} \int_{Z_{\phi_{j_k}^{\C}}} \phi \frac{\omega_g^{m-1}}{(m -
1)!}  \to
 \frac{i}{ \pi} \int_{B^*_{\epsilon} M} \phi  \ddbar
\sqrt{\rho}  \wedge \frac{\omega_g^{m-1}}{(m - 1)!} . $$
\end{theo}

\begin{cor}\label{ZEROCOR}  Let $(M, g)$ be a real analytic with ergodic  geodesic
flow.  Let $\{\phi_{j_k}\}$ denote a full density ergodic
sequence. Then for all $\epsilon < \epsilon_0$,
$$\frac{1}{\lambda_{j_k}} [Z_{\phi_{j_k}^{\C}}] \to  \frac{i}{ \pi} \ddbar \sqrt{\rho},\;\;
 \mbox{weakly in}\;\; \dcal^{' (1,1)} (B^*_{\epsilon} M). $$
\end{cor}

The proof consists of three ingredients:

\begin{enumerate}

\item By the Poincar\'e-Lelong formula, $[Z_{\phi_{\lambda}^{\C}}]
= i \ddbar \log |\phi_{\lambda}^{\C}|. $ This reduces the theorem
to determining the limit of $\frac{1}{\lambda} \log
|\phi_{\lambda}^{\C}|$.

\item $\frac{1}{\lambda} \log |\phi_{\lambda}^{\C}|$ is a sequence
of PSH functions which are uniformly bounded above by
$\sqrt{\rho}$. By a standard compactness theorem, the sequence is
pre-compact in $L^1$: every sequence from the family has an $L^1$
convergent subsequence.

\item $|\phi_{\lambda}^{\C}|^2$, when properly $L^2$ normalized on
each $\partial M_{\tau}$ is a quantum ergodic sequence on
$\partial M_{\tau}$. This property implies that the $L^2$ norm of
$|\phi_{\lambda}^{\C}|^2$
  on $\partial \Omega$ is asymtotically $\sqrt{\rho}$.

  \item Ergodicity and the calculation of the $L^2$ norm imply that  the only possible $L^1$ limit
of $\frac{1}{\lambda} \log |\phi_{\lambda}^{\C}|$. This concludes
the proof.

\end{enumerate}

We note that the first two steps are valid on any real analytic
$(M, g)$. The difference is that the $L^2$ norms of
$\phi_{\lambda}^{\C}$ may depend on the subsequence and can often
not equal $\sqrt{\rho}$. That is,  $\frac{1}{\lambda}
|\phi_{\lambda}^{\C}|$ behaves like the maximal PSH function in
the ergodic case, but not in general. For instance, on a flat
torus, the complex zero sets of ladders of eigenfunctions
concentrate on a real hypersurface in $M_{\C}$. This may be seen
from the complexified  real eigenfunctions $ \sin \langle k, x + i
\xi \rangle$, which vanish if and only if $\langle k, x \rangle
\in 2 \pi \Z$ and $\langle k, \xi \rangle = 0$. Here, $k \in \N^m$
is a lattice point. The exact limit distribution depends on which
ray or ladder of lattice points one takes in the limit.  The
result reflects the quantum integrability of the flat torus, and a
similar (but more complicated) description of the zeros exists in
all quantum integrable cases. The fact that $\frac{1}{\lambda}
\log |\phi_{\lambda}^{\C}|$ is pre-compact on a Grauert tube of
any real analytic Riemannian manifold  confirms the upper bound on
complex  nodal hypersurface volumes.

\section{\label{NODALGEOS} Intersections of nodal sets and gedoesics on real
analytic surfaces}

In \S \ref{PLANEDOMAIN}  we discussed upper bounds on the number of  intersection points of the nodal set with the bounary
of a real analytic plane domain and more general `good' analytic curves. In this section, we discuss work in progress
on intersections of nodal sets and geodesics on surfaces with ergodic geodesic flow.  Of course, the results are only
tentative but it seems worthwhile at this point in time to explain the role of ergodicity in obtaining lower bounds
and asymptotics.  We restrict to geodesic curves because they have rather special properties that makes the 
analysis somewhat different than for more general curves such as distance circles.  The dimensional restriction
is due to the fact that the results are partly based on the quantum ergodic restriction theorems of \cite{TZ2,TZ3},
which concern restrictions of eigenfunctions to hypersurfaces. Nodal sets and geodesics have complementary
dimensions and intersect in points, and therefore it makes sense to count the number of intersections. 

We fix $(x, \xi) \in S^* M$ and let 
\begin{equation} \label{GAMMAX} \gamma_{x, \xi}: \R \to M, \;\;\;\gamma_{x, \xi}(0) = x, \;\;
\gamma_{x, \xi}'(0) = \xi \in T_x M  \end{equation}  denote the corresponding parametrized geodesic. 
Our goal is to  determine the asymptotic  distribution of intersection points of $\gamma_{x, \xi}$ with the nodal 
set of a highly eigenfunction. As usual, we cannot cope with this problem in the real domain and therefore analytically
continue it to the complex domain. Thus, we consider the intersections 
$$\ncal^{\gamma_{x, \xi}^{\C}}_{\lambda_j} = Z_{\phi_{_j}^{\C}}  \cap \gamma_{x, \xi}^{\C} $$  of the complex nodal set  with the (image of the) complexification of a generic geodesic
 If  \begin{equation} \label{SEP} S_{\epsilon} = \{(t
+ i \tau \in \C: |\tau| \leq \epsilon\} \end{equation} then $\gamma_{x, \xi}$ admits an  analytic
continuation
\begin{equation} \label{gammaXCX} \gamma_{x, \xi}^{\C}: S_{\epsilon} \to M_{\epsilon}.  \end{equation}
In other words, we consider the zeros of the pullback, 
$$\{\gamma_{x, \xi}^* \phi_{\lambda}^{\C} = 0\} \subset S_{\epsilon}. $$

 We encode the discrete 
set by the measure
\begin{equation} \label{NCALCURRENT} [\ncal^{\gamma_{x, \xi}^{\C}}_{\lambda_j}] = \sum_{(t + i \tau):\; \phi_j^{\C}(\gamma_{x, \xi}^{\C}(t + i \tau)) = 0} \delta_{t + i \tau}.
\end{equation}

We would like to show that for generic geodesics, the complex zeros on the complexified geodesic condense
on the real points and become uniformly distributed with respect to arc-length. This does not always occur:
as in our discussion of QER theorems, if 
$\gamma_{x, \xi}$ is the fixed point set of an isometric involution,  then ``odd" eigenfunctions under the 
involution will vanish on the geodesic. The additional hypothesis is that QER holds for $\gamma_{x, \xi}$, i.e.
that Theorem \ref{sctheorem} is valid. 
 The following
conjecture appears to be  proved  (\cite{Z3}), but to be conservative, we state it here only as a conjecture:

\begin{conj}\label{MAINCOR} Let $(M^2, g)$ be a real analytic Riemannian surface  with ergodic
geodesic flow. Let $\gamma_{x, \xi}$ satisfy the QER hypothesis.   Then there
exists a subsequence of eigenvalues $\lambda_{j_k}$ of density one
such that for any $f \in C_c(S_{\epsilon})$,
$$\lim_{k \to \infty} \sum_{(t + i \tau):\; \phi_j^{\C}(\gamma_{x, \xi}^{\C}(t + i \tau)) = 0} f(t + i
\tau ) = \int_{\R} f(t) dt.  $$
\end{conj}

In other words,
$$\mbox{weak}^* \lim_{k \to \infty} \frac{i}{\pi \lambda_{j_k}} [\ncal^{\gamma_{x, \xi}^{\C}}_{\lambda_j}] =
 \delta_{\tau = 0}, $$ in the sense of  weak* convergence on
$C_c(S_{\epsilon})$. Thus, the complex nodal set intersects the
(parametrized) complexified geodesic in a discrete set  which is
asymptotically (as $\lambda \to \infty$) concentrated along the
real geodesic  with respect to its arclength.

This concentration- equidistribution result is a `restricted'
version of the result of \S \ref{ERGODICNODAL}. As noted there, the limit distribution of
complex nodal sets in the ergodic case is a singular current $dd^c \sqrt{\rho}$. The motivation
for restricting to geodesics is that restriction magnifies the singularity of this current.  In the
case of a geodesic, the singularity is magnified to a delta-function; for other curves there is additionally
a smooth background measure.

The assumption of ergodicity is crucial. For instance, in the 
case of a flat torus, say $\R^2/L$ where $L \subset \R^2$ is a generic
lattice, the real eigenfunctions are
$\cos \langle \lambda, x \rangle, \sin \langle \lambda,x \rangle$
where $\lambda \in L^*$, the dual lattice, with eigenvalue $-
|\lambda|^2$.  Consider  a geodesic
$\gamma_{x,\xi}(t) = x + t \xi$. Due to the flatness, the
restriction $\sin \langle \lambda, x_0 + t \xi_0 \rangle$  of the
eigenfunction to a geodesic is an eigenfunction of the Laplacian
$-\frac{d^2}{dt^2}$ of  submanifold metric along the geodesic with
eigenvalue $- \langle \lambda, \xi_0 \rangle^2$. The
complexification of the restricted eigenfunction is $\sin \langle
\lambda,  x_0 + (t + i \tau) \xi_0 \rangle |$ and its exponent of
its growth is $\tau |\langle \frac{\lambda}{|\lambda|} , \xi_0
\rangle|$, which can have a wide range of values as the eigenvalue
moves along different rays in $L^*$. The limit current is $i
\ddbar$ applied to the limit and thus also has many limits

The proof  involves several  new principles which played no role in the global
result of \S \ref{ERGODICNODAL}  and which are specific to geodesics. However, the first steps in the proof
are the same as in the global case. 
 By the Poincar\'e-Lelong formula, we may express the
current of summation over the intersection points in \eqref{NCALCURRENT} in the form,
\begin{equation}\label{PLL}  [\ncal^{\gamma_{x, \xi}^{\C}}_{\lambda_j}] = i \ddbar_{t + i
\tau} \log \left| \gamma_{x, \xi}^* \phi_{\lambda_j}^{\C} (t + i \tau)
\right|^2. \end{equation}  
Thus, the main point of the proof is to determine the asymptotics of  $\frac{1}{\lambda_j} \log \left| \gamma_{x, \xi}^* \phi_{\lambda_j}^{\C} (t + i \tau)
\right|^2$.  When we freeze $\tau$ we put
\begin{equation} \label{gammatau} \gamma_{x, \xi}^{\tau} (t) = \gamma^{\C}_{x, \xi}(t + i \tau). \end{equation}

\begin{prop} \label{MAINPROPa} (Growth saturation) If $\{\phi_{j_k}\}$ satisfies QER along any arcs of $\gamma_{x, \xi}$, then in $L^1_{loc} (S_{\tau}),  $ we have   $$\lim_{k \to \infty} \frac{1}{\lambda_{j_k}} \log \left| \gamma_{x, \xi}^{\tau *} \phi_{\lambda_{j_k}}^{\C} (t + i \tau)
\right|^2 = |\tau|. $$

\end{prop}

Proposition \ref{MAINPROPa} immediately implies Theorem \ref{MAINCOR} since we can apply $\ddbar$ to the
$L^1$ convergent sequence $\frac{1}{\lambda_{j_k}} \log \left| \gamma_{x, \xi}^* \phi_{\lambda_{j_k}}^{\C} (t + i \tau)
\right|^2 $ to obtain $\ddbar |\tau|$.

 The upper bound in Proposition \ref{MAINPROPa} follows immediately from the known global estimate 
$$\lim_{k \to \infty} \frac{1}{\lambda_j} \log |\phi_{j_k}(\gamma_{x, \xi}^{\C}(\zeta)| \leq |\tau|$$
on all of $\partial M_{\tau}$.
Hence the difficult point is to prove   that this growth rate is actually obtained
upon restriction to $\gamma_{x, \xi}^{\C}$.  This requires new kinds of arguments related to the QER
theorem.

\begin{itemize}

\item Complexifications of restrictions of eigenfunctions to geodesics have incommensurate Fourier modes,
i.e. higher modes are exponentially larger than lower modes.

\item The quantum ergodic restriction theorem in the real domain  shows that the Fourier
coefficients of the top allowed modes are `large' (i.e. as large as the lower modes).  Consequently, the $L^2$
norms of the complexified eigenfunctions along arcs of $\gamma_{x, \xi}^{\C}$ achieve the lower bound
of Proposition \ref{MAINPROPa}. 

\item Invariance of Wigner measures along the geodesic flow implies that the Wigner measures of restrictions
of complexified eigenfunctions to complexified geodesics should tend to constant multiples of Lebesgue measures
$dt$ for each $\tau > 0$. Hence the eigenfunctions everywhere on $\gamma_{x, \xi}^{\C}$ achieve the growth
rate of the $L^2$ norms.

\end{itemize}

These principles are most easily understood in the case of periodic geodesics. We let
$\gamma_{x, \xi}: S^1 \to M$ parametrize the geodesic with arc-length (where $S^1 = \R/ L \Z$ where
$L$ is the length of $\gamma_{x, \xi}$). 

 First, we use Theorem
\ref{sctheorem} to prove

\begin{lem} \label{L2NORMintro} Assume that $\{\phi_j\}$ satsifies QER along the
periodic geodesic $\gamma_{x, \xi}$. Let $||\gamma_{x, \xi}^{\tau*} \phi_j^{\C}||^2_{L^2(S^1)}$ be the $L^2$-norm
of the complexified restriction of $\phi_j$ along $\gamma_{x, \xi}^{\tau}$. Then,
$$\lim_{\lambda_j \to \infty} \frac{1}{\lambda_j} \log ||\gamma_{x, \xi}^{\tau*} \phi_j^{\C}||^2_{L^2(S^1)}
= |\tau| .$$
\end{lem}

To prove Lemma \ref{L2NORMintro}, we study the      orbital Fourier series of $\gamma_{x, \xi}^{\tau*} \phi_j$
and of its complexification. The orbital Fourier coefficients are 
$$\nu_{\lambda_j}^{x, \xi}(n) = \frac{1}{L_{\gamma}} \int_0^{L_{\gamma}} \phi_{\lambda_j}(\gamma_{x, \xi}(t)) e^{- \frac{2 \pi i n t}{L_{\gamma}}} dt, $$
and the orbital Fourier series is 
\begin{equation} \label{PER} \phi_{\lambda_j}(\gamma_{x, \xi}(t) )= \sum_{n \in \Z}  \nu_{\lambda_j}^{x, \xi}(n)  e^{\frac{2 \pi i n t}{L_{\gamma}}}. 
\end{equation}
Hence the analytic continuation of $\gamma_{x, \xi}^{\tau*} \phi_j$  is given by 
\begin{equation} \label{ACPER} \phi^{\C}_{\lambda_j}(\gamma_{x, \xi}(t + i \tau) )= \sum_{n \in \Z}  \nu_{\lambda_j}^{x, \xi}(n)  e^{\frac{2 \pi i n (t + i \tau)}{L_{\gamma}}}. \end{equation}
By the Paley-Wiener theorem for Fourier series, the  series converges absolutely and uniformly for $|\tau| \leq \epsilon_0$. 
By ``energy localization" only the modes with $|n| \leq \lambda_j$ contribute substantially to the $L^2$ norm. We
then observe that the Fourier modes decouple, since they have different exponential growth rates. We use the
QER hypothesis in the following way:

\begin{lem} \label{FCSAT}  Suppose that $\{\phi_{\lambda_j}\}$ is QER along the periodic geodesic $\gamma_{x, \xi}$.
Then for all $\epsilon > 0$, there exists $C_{\epsilon} > 0$ so that
$$\sum_{n: |n| \geq (1 - \epsilon) \lambda_j}  |\nu_{\lambda_j}^{x, \xi}(n)|^2 \geq  C_{\epsilon}. $$

\end{lem}

 Lemma \ref{FCSAT}  implies Lemma \ref{L2NORMintro} since it implies that for any $\epsilon > 0$,
$$\sum_{n: |n| \geq (1 - \epsilon) \lambda_j}  |\nu_{\lambda_j}^{x, \xi}(n)|^2 e^{-2 n \tau}  \geq  C_{\epsilon} e^{2\tau(1 - 
\epsilon) \lambda_j}. $$

To go from asymptotics of $L^2$ norms of restrictions to Proposition \ref{MAINPROPa} we then
use the third principle:

\begin{prop} \label{LL} (Lebesgue limits)  If
 $\gamma_{x, \xi}^* \phi_j \not=  0$ (identically),  then for all $\tau > 0$ the
 sequence 
$$U_j^{x, \xi, \tau} = \frac{\gamma_{x, \xi}^{\tau *} \phi_j^{\C}}{||{\gamma_{x, \xi}^{\tau *} \phi_j^{\C}}||_{L^2(S^1)} }$$
  is QUE  with limit measure given by normalized  Lebesgue measure on $S^1$. \end{prop}
The proof of Proposition \ref{MAINPROPa} is completed by combining Lemma \ref{L2NORMintro}  and
Proposition \ref{LL}.   Conjecture  \ref{MAINCOR} follows easily from Proposition \ref{MAINPROP}.

The proof for non-periodic geodesics is considerably more involved, since one cannot use Fourier analysis
in quite the same way.

\section{\label{RWONB} Nodal and critical sets of Riemannian random waves }

We mentioned above that Riemannian random waves provide a probabilistic model that is conjectured
to predict the behavior of eigenfunctions when the geodesic flow of $(M, g)$ is ergodic. In this section,
we define the model precisely as in \cite{Z4} (see also \cite{Nic} for a similar model) and survey some of the current results and conjectures.  We should emphasize that some 
of the rigorous results on zeros or critical points of Riemannian random waves, both in the real and complex domain,
are much simpler than  for individual eigenfuntions, and therefore do not provide much guidance on how to prove
results for an orthonormal basis of eigenfunctions. But the relative  simplicity of random waves and their  value as predictors  provide the motivation for  studying
random waves. And there are many hopelessly difficult problems on random waves as well, which we will
survey in this section. 

For expository simplicity we assume 
that  the geodesic flow $G^t$ of $(M, g)$  is of  one of the
following two types:

\begin{enumerate}

\item  {\it aperiodic:} The Liouville measure of the closed
 orbits of $G^t$, i.e. the set of vectors lying on closed geodesics,  is zero; or

\item  {\it periodic = Zoll:} $G^T = id$ for some  $T>0$;
henceforth $T$ denotes the minimal period.  The common Morse index
of the $T$-periodic geodesics will be denoted by $\beta$.

\end{enumerate}
In the real analytic case, $(M, g)$ is automatically one of these
two types, since a positive measure of closed geodesics implies
that all geodesics are closed.
The two-term Weyl laws counting eigenvalues of $\sqrt{\Delta}$ are
very different in these two cases.

\begin{enumerate}

\item  In the {\it aperiodic} case, Ivrii's two term Weyl law
states
$$N(\lambda ) = \#\{j:\lambda _j\leq \lambda \}=c_m \;
Vol(M, g) \; \lambda^m +o(\lambda ^{m-1})$$
 where $m=\dim M$ and where $c_m$ is a universal constant.

\item  In the {\it periodic} case,
 the spectrum of $\sqrt{\Delta}$ is a union of eigenvalue clusters $C_N$ of the form
$$C_N=\{(\frac{2\pi}{T})(N+\frac{\beta}{4}) +
 \mu_{Ni}, \; i=1\dots d_N\}$$
with $\mu_{Ni} = 0(N^{-1})$.   The number $d_N$ of eigenvalues in
$C_N$ is a polynomial of degree $m-1$.
\end{enumerate}

We refer to \cite{HoI-IV,Z4} for background and further discussion.

 To define Riemannian random waves, we partition the
spectrum of $\sqrt{\Delta_g}$ into certain  intervals $I_N$ of
width one  and denote by $\Pi_{I_N} $ the spectral projections for
$\sqrt{\Delta_g}$ corresponding to the interval $I_N$. The choice
of the intervals $I_N$ is rather arbitrary for aperiodic $(M, g)$
and as mentioned above we  assume $I_N = [N, N + 1]$. 
In the Zoll case, we center the intervals around the center points
$\frac{2\pi}{T} N + \frac{\beta}{4}$ of the $N$th cluster $C_N$.
 We call call such a choice of
intervals a cluster decomposition. We denote by $d_N$ the number
of eigenvalues
  in $I_N$ and  put $\hcal_N = \mbox{ran} \Pi_{I_N}$ (the range of
 $\Pi_{I_N}$).

We choose an orthonormal basis $\{\phi_{N j}\}_{j = 1}^{d_N}$ for
$\hcal_N$. For instance, on $S^2$ one can choose the real and
imaginary parts of the standard $Y^N_m$'s.  We endow the real vector space $\mathcal{H}_N$ with the
Gaussian probability measure
 $\gamma_N$ defined by
\begin{equation}\label{gaussian}\gamma_N(f)=
\left(\frac{d_N}{\pi }\right)^{d_{N}/2}e^ {-d_{N}|c|^2}dc\,,\qquad
f=\sum_{j=1}^{d_{\lambda}}c_j \phi_{N j}, \,\;\; d_{N} = \dim
\hcal_{N}.
\end{equation} Here,  $dc$
is $d_{N}$-dimensional real Lebesgue measure. The normalization is
chosen so that $\E_{\gamma_N}\; \langle f,f\rangle=1$, where
$\E_{\gamma_N}$ is the expected value with respect to $\gamma_N$.
Equivalently,  the $d_N$ real variables $\ c_j$
($j=1,\dots,d_{N}$) are independent identically distributed
(i.i.d.) random variables with mean 0 and variance
$\frac{1}{2d_{N}}$; i.e.,
$$\E_{\gamma_N} c_j = 0,\quad  \E_{\gamma_N} c_j c_k =
\frac{1}{2 d_N}\de_{jk}\,.$$ We note  that the Gaussian ensemble
is equivalent to picking $f_N \in \hcal_{N}$ at random from the
unit sphere in $\hcal_{N}$ with respect to the $L^2$ inner
product.

Depending on the choice of intervals, we obtain the following special ensembles:
\begin{itemize}

\item The asymptotically  {\it fixed frequency} ensemble
$\hcal_{I_{\lambda}}$,  where $I_{\lambda} = [\lambda, \lambda +
1]$ and where  $\hcal_{I_{\lambda}}$ is the vector space of linear
combinations
\begin{equation} \label{psi} f_{\lambda} = \sum_{j: \lambda_j
\in [\lambda, \lambda + 1]} c_j\;\; \phi_{\lambda_j},
\end{equation}  of eigenfunctions  with $\lambda_j$ (the
frequency) in an interval $[\lambda, \lambda + 1]$ of fixed width.
(Note that it is the square root of the eigenvalue of $\Delta$,
not the eigenvalue,  which is asymptotically fixed).

\item The {\it high frequency cut-off} ensembles $\hcal_{[0,
\lambda]}$ where the frequency is cut-off at $\lambda$:
\begin{equation} \label{psi2} f_{\lambda} = \sum_{j: \lambda_j
\leq \lambda } c_j\;\; \phi_{\lambda_j}.
\end{equation}

\item  The {\it cut-off Gaussian free field},
\begin{equation} \label{psi2a} f_{\lambda} = \sum_{j: \lambda_j
\leq \lambda } c_j\;\;\frac{ \phi_{\lambda_j}}{\lambda_j}.
\end{equation}

\end{itemize}

One  could use more general weights
$w(\lambda_j)$ on a Sobolev space of functions or distributions on
$M$. In the physics terminology, $w(\lambda_j)$ (or it square) is
referred to as the power spectrum.

The key reason why we can study the limit distribution of nodal sets in this 
ensemble is that the covariance kernel 

\begin{equation} \Pi_{I_N}(x, y) = \E_{\gamma_N} (f_N(x) f_N(y)) =
\sum_{j: \lambda_j \in I_N} \phi_{\lambda_j}(x)
\phi_{\lambda_j}(y),
\end{equation}
is  the spectral projections kernel for $\sqrt{\Delta}.$

\subsection{Equidistribution of nodal sets for almost all sequences of random waves}

The real zeros are straightforward to define. For each
$f_{\lambda} \in \hcal_{[0, \lambda]}$ or
 $ \hcal_{I_{\lambda}}$ we
 associated to
the zero set $Z_{f_{\lambda}} = \{x \in M: f_{\lambda} (x) = 0\}$
 the positive measure
\begin{equation} \label{LINSTAT} \langle |Z_{f_{\lambda}}|, \psi \rangle = \int_{Z_{f_{\lambda}}}
\psi d\hcal^{n - 1}, \end{equation} where $d\hcal^{m-1}$ is the
induced (Hausdorff) hypersurface measure.

The main result we review is the  limit law for random sequences of random real
Riemannian waves. By a random sequence,  we mean an element of the
product probability space
\begin{equation} \label{HINFTY} \hcal_{\infty} = \prod_{N = 1}^{\infty} \hcal_N, \;\;
\gamma_{\infty} = \prod_{N = 1}^N \gamma_N. \end{equation}

\begin{theo} \label{ZDSM} \cite{Z4}  Let $(M, g)$ be a compact Riemannian
manifold, and let  $\{f_{_N} \}$ be a random sequence in
(\ref{HINFTY}).  Then
$$\frac{1}{N}\sum_{n = 1}^N \frac{1}{\lambda_n}
|Z_{f_{n}}| \to dV_g \;\;\;\; \mbox{almost surely w.r.t.} \;
(\hcal_{\infty}, \gamma_{\infty}).
$$
\end{theo}

\subsection{Mean and variance}

 We first show  that the  normalized expected limit
distribution $\frac{1}{\lambda} \E |Z_{f_{\lambda}}|$ of zeros of
random Riemannian waves  tends to the volume form $dV_g$  as
$\lambda \to \infty$.  That is, we define the   `linear statistic', 
\begin{equation} \label{LINSTAT2} X_{\psi}^N (f_N) = \langle \psi, |Z_{f_N}| \rangle, \;\;\; \psi \in C(M) \end{equation} 
and then define
\begin{equation} \label{EZ} \langle \E_{\gamma_N} |Z_{f_N}|, \psi
\rangle = \E_{\gamma_N} X_{\psi}^N, \end{equation}

\begin{theo}\label{ZMEASURE} Let $(M, g)$ be a compact Riemannian
manifold,let $\hcal_{[0, \lambda]}$ be the cutoff ensemble  and
let $({\mathcal H}_{N}, \gamma_{N})$ be the ensemble of Riemannian
waves of asymptotically fixed frequency. Then in either ensemble:
\begin{enumerate}

\item For any $C^{\infty}$ $(M, g)$,  $\lim_{N \to \infty}
\frac{1}{N} {\bf E}_{\gamma_N} \langle |Z_{f_N}|, \psi \rangle =
\int_M \psi dV_g$.

\item  For a real analytic $(M, g)$, $Var(\frac{1}{N} X^N_{\psi} )
) \leq C. $

\end{enumerate}

\end{theo}

We restrict to real analytic metrics in (2) for the sake of
brevity. In that case, the  variance estimate follows easily from
Theorem \ref{DF}.

\subsection{\label{DEN} Density of real zeros}

The formula for the density of zeros of random elements of
$\hcal_{N}$ can be derived from the general Kac-Rice formula
\cite{BSZ1,BSZ2,Nic}:

\begin{equation}\label{d2} \E|Z_{f_N}| =K_1^N (z)dV_g \,,\quad
K_1^N(x)= \int D(0,\xi,x) ||\xi|| \; d\xi\,.
\end{equation}
Here, $D(q, \xi, x) dq d \xi$ is the joint probability distribution of the Gaussian
random variables $(\psi(x), \nabla \psi(x))$, i.e. the pushforward of the Gaussian
measure on $\hcal_{\lambda})$ under the map $\psi \to (\psi(x), \nabla \psi(x))$. 
Note that the factor $\det(\xi\xi^*) $ in \cite{BSZ1,BSZ2} equals $ ||\xi||^2$ in
the codimension one case. Indeed,  let  $df^*_x$ be the adjoint
map with respect to the inner product $g$  on  $T_x M$. Let $df_x
\circ df_x^*: \R \to \R$  be the composition. By $\det df_x \circ
df_x^*$ is meant the determinant with respect to the inner product
on $T_x M$; it clearly equals $|df|^2$ in the codimension one
case.

The formulae of \cite{BSZ1, BSZ2} (the `Kac-Rice' formulae) give
that
\begin{equation}\label{fg4}
D(0,\xi;z)=Z_{n}(z) D_{\La}(\xi;z),
\end{equation}
where
\begin{equation}\label{fg5}
D_{\La}(\xi;z)=\frac{1}{\pi^{m}\sqrt{\det\La}}\exp\left( -{\langle
\La^{-1}\xi,\xi\rangle}\right)
\end{equation}
is the Gaussian density with covariance matrix
\begin{equation}\label{fg6}
\La=C-B^*A^{-1}B =\left(C^{q}_{q'} - B_{q} A^{-1} B_{q'}\right),
\;\;(q = 1, \dots, m)
\end{equation}
and
\begin{equation}\label{fg7}
Z(x)=\frac{\sqrt{\det\Lambda}}{\pi\sqrt{\det\De}} =\frac{1}{\pi
\sqrt{  A}}\,.
\end{equation}

Here, 

\begin{eqnarray} \Delta = 
\Delta^{N}(x)&=&\left(
\begin{array}{cc}
A^{N} & B^{N}\\
B^{N *} & C^{N}
\end{array}\right)\,,\nonumber \\
\big( A^{N} \big) &=& \E\big ( X^2\big)=\frac{1}{d_N}\Pi_{I_N}(x,x)\,,\nonumber\\
\big( B^{N}\big)_q&=& \E\big( X\Xi_{q}\big)= \frac{1}{d_{N}}
\frac{\partial}{\partial y_q} \Pi_{I_N}(x,y)|_{x = y}\,,\nonumber\\
\big( C^{\lambda}\big)^{q}_{q'}&=& \E\big(  \Xi_{q} \Xi_{q'}\big)=
 \frac{1}{d_{N}}
\frac{\partial^2}{\partial x_q \partial y_{q'} }\Pi_{I_N}(x,y)|_{x = y}\,,\nonumber\\
&&  \quad q, q'=1,\dots,m\,.\nonumber
\end{eqnarray}

Making a simple change of variables in the integral (\ref{d2}), we
have

\begin{prop}\label{BSZ}  \cite{BSZ1} On a real Riemannian manifold of dimension $m$, the
density of zeros of a random Riemannian wave is

\begin{equation}\begin{array}{lll}\label{EVOL1}
K_1^N(x) & = &\frac{1 }{\pi^{m} (\sqrt{d_N^{-1} \;\;
\Pi_{I_N}(x,x)}} \int_{\R^m} || \Lambda^N (x)^{1/2} \xi||
\exp\left( - {\langle \xi,\xi\rangle}\right)d\xi,
  \end{array}
\end{equation}
where $\Lambda^N(x)$ is a symmetric form on $T_xM$. For the
asymptotically fixed freqency ensembles, it is  given by
$$\Lambda^N(x) = \frac{1}{d_N} \left(d_x \otimes d_y \Pi_{I_N}(x, y) |_{x = y} -
\frac{1}{\Pi_{I_N}(x, y)} d_x  \Pi_{I_N}(x, y) |_{x = y} \otimes
 d_y \Pi_{I_N}(x, y) |_{x = y} \right). $$
 In the cutoff ensemble the formula is the same except that
 $\Pi_{I_N}$ is replaced by $\Pi_{[0, N]}$.

\end{prop}

We then need the asymptotics of the matrix elements of $\Delta^N(x)$. They are simplest for the
round sphere, so we state them first in that case:

\begin{prop} \label{PINSM} Let $\Pi_N: L^2(S^m) \to \hcal_N$ be the orthogonal
projection. Then:

\begin{itemize}

\item (A) $\Pi_N(x,x) = \frac{1}{Vol(S^m)} d_N$;

\item (B) $d_x \Pi_N(x, y)|_{x = y} = d_y \Pi_N(x, y)|_{x = y} =
0$;

\item (C) $d_x \otimes d_y \Pi_N(x, y)|_{x = y} = \frac{1}{m
Vol(S^{m})} \lambda_N^2 d_N g_x.$

\end{itemize}

\end{prop}

We refer to \cite{Z4} for the calculation, which is quite simple because of the invariance under rotations. 
The expected density of random  nodal hypersurfaces is given as
follows

\begin{prop} \label{REALDENSITY} In the case of $S^m$,

\begin{equation}\begin{array}{lll}\label{EVOL}
K_1^N(x) & = & C_m  \lambda_N\sim C_m N ,  \end{array}
\end{equation} where $C_m = \frac{1}{\pi^{m}} \int_{\R^m} |\xi|
\exp\left( - {\langle\xi,\xi\rangle}\right) d\xi.$

\end{prop}

\begin{proof}

By Propositiosn \ref{BSZ} and \ref{PINSM},  we have
\begin{equation} K_1^N(x) = \frac{\sqrt{Vol(S^m)}}{\pi^{m}} \int_{\R^m} || \Lambda^N (x)^{1/2} \xi||
\exp\left( - {\langle \xi,\xi\rangle}\right)d\xi,
\end{equation}
where
$$\Lambda^N(x) = \frac{1}{d_N} \left(\frac{1}{m
Vol(S^{m})} \lambda_N^2 d_N g_x \right). $$

\end{proof}

\subsection{Random Riemannian waves: proof of Theorem \ref{ZMEASURE}}

We now generalize the result to any compact $C^{\infty}$
Riemannian manifold $(M, g)$ which is either aperiodic or Zoll. As
in the case of $S^m$, the key issue is the asymptotic behavior of
derivatives of the spectral projections
\begin{equation}\label{RESPECPROJ}  \Pi_{I_N}(x, y) = \sum_{j:
\lambda_j \in I_N} \phi_{\lambda_j}(x) \phi_{\lambda_j}(y).
\end{equation}

\begin{prop}\label{PING}   Assume $(M, g)$ is either aperiodic and $I_N = [N, N + 1]$ or Zoll and
 $I_N$ is a cluster decomposition. Let $\Pi_{I_N}: L^2(M) \to
\hcal_N$ be the orthogonal projection. Then:

\begin{itemize}

\item (A) $\Pi_{I_N}(x,x) = \frac{1}{Vol(M, g))} d_N (1 + o(1))$;

\item (B) $d_x \Pi_{I_N}(x, y)|_{x = y} = d_y \Pi_N(x, y)|_{x = y}
= o(N^{m })$;

\item (C) $d_x \otimes d_y \Pi_{I_N}(x, y)|_{x = y} =
\frac{1}{Vol(M, g))} \lambda_N^2  d_N g_x (1 + o(1)).$

\end{itemize}

In the aperiodic case,
\begin{enumerate}

\item $\Pi_{[0, \lambda]} (x,x) = C_m \lambda^m + o(\lambda^{m -
1}); $

\item $d_x \otimes d_y \Pi_{[0, \lambda]}(x, y) |_{x = y} = C_m
\lambda^{m +2} g_x + o(\lambda^{m + 1 }). $

\end{enumerate}
In the Zoll case, one adds the complete asymptotic expansions for
$\Pi_{I_N}$ over the $N$ clusters to obtain expansions for
$\Pi_N$.

\end{prop}

We then have:

\begin{prop} \label{REALDENSITYg} For the asymptotically fixed frequency
ensemble, and for any $C^{\infty}\;\; (M, g)$ which is either Zoll
or aperiodic (and with $I_N$ as in Proposition \ref{PING}) , we
have

\begin{equation}\begin{array}{lll}\label{EVOLa}
K_1^N(x)  & = & \frac{1}{\pi^{m} (\lambda_N)^{m/2}  } \int_{\R^m}
||\xi|| \exp\left( -\frac{1}{ \lambda_N}
{\langle\xi,\xi\rangle}\right) d\xi + o(1) \\& &  \\
& \sim &  C_m N ,  \end{array}
\end{equation} where $C_m = \frac{1}{\pi^{m}} \int_{\R^m} ||\xi||
\exp\left( - {\langle\xi,\xi\rangle}\right) d\xi.$ The same
formula holds for the cutoff ensemble.

\end{prop}

\begin{proof}

Both on a sphere $S^m$ or on a more general $(M, g)$ which is
either Zoll or aperiodic, we have by Propositions \ref{PINSM}
resp. \ref{PING} and the general formula for $\Delta^N$ in \S
\ref{DEN} that

\begin{eqnarray}
\; \Delta^N(z)&=& \frac{1}{Vol(M, g)}  \left(
\begin{array}{cc}
(1 + o(1))   & o(1)\\
o(1) &  N^2  \;  g_x(1 +   o(1))
\end{array}\right)\,,\nonumber \\
\end{eqnarray}
It follows that
\begin{equation}\label{fg6a}
\La^N=C^N-B^{N *}(A^N)^{-1}B^N =  \frac{1}{Vol(M, g)}  N^2 \; g_x
+ o(N). \end{equation}

Thus, we have
\begin{equation} \begin{array}{lll} K_1^N(x) & \sim  &
 \frac{\sqrt{Vol(M, g)}}{\pi^{m} } \int_{\R^m} || \Lambda^N (x)^{1/2} \xi||
\exp\left( - {\langle \xi,\xi\rangle}\right)d\xi \\ && \\
& = & \frac{N}{\pi^{m}} \int_{\R^m} || (I + o(1)) (x)^{1/2} \xi||
\exp\left( - {\langle \xi,\xi\rangle}\right)d\xi,
\end{array} \end{equation}
where $o(1)$ denotes a matrix whose norm is $o(1)$, as as $N \to
\infty$ we obtain the stated asymptotics.

\end{proof}

So far, we have only determined the expected values of the nodal
hypersurface measures. To complete the proof of Theorem
\ref{ZMEASURE}, we need to prove:

\begin{prop} \label{VAR} If $(M, g)$ is real analytic, then
 the variance of $\frac{1}{\lambda_N} X^N_{\psi}$ is bounded.
 \end{prop}

 \begin{proof} By Theorem \ref{DF},
 for  $f_N \in \hcal_{I_N}$,  $\frac{1}{\lambda_N}
Z_{f_N}$ has bounded mass. Hence, the random variable
$\frac{1}{\lambda_N} X^N_{\psi}$ is bounded, and therefore so is
its variance.

\end{proof}

\begin{rem}

The variance of $\frac{1}{\lambda_N} X^N_{\psi}$ is given by
\begin{equation} Var (\frac{1}{\lambda_N} X^N_{\psi}) =
\frac{1}{\lambda_N^2} \int_M \int_M \left(K^N_2(x, y) - K^N_1(x)
K^N_1(y) \right) \psi(x) \psi(y) dV_g(x) dV_g(y), \end{equation}
where $K_2^N(x, y) = \E_{\gamma_N} (Z_{f_N}(x) \otimes
Z_{f_N}(y))$ is the pair correlation function for zeros.  Hence,
boundedness would follow from
\begin{equation}
\frac{1}{\lambda_N^2} \int_M \int_M  K^N_2(x, y) \; dV_g(x)
dV_g(y) \leq C.
\end{equation}
There is a formula similar to that for the density in Proposition
\ref{BSZ} for $K^N_2(x,y)$ and it is likely that  it could be used
to prove boundedness of the variance for any $C^{\infty}$
Riemannian manifold.
\end{rem}

\subsection{Random sequences and  proof of Theorem  \ref{ZDSM}}

We recall that the set of   random sequences  of Riemannian waves
of increasing frequency is  the probability space $\hcal_{\infty}
= \prod_{N=1}^{\infty} \hcal_{I_N}$ with the measure
$\gamma_{\infty}  = \prod_{N=1}^{\infty} \gamma_N$. An element in
$\hcal_{\infty}$ will be denoted ${\bf f} = \{f_N\}$. We have,
$$|(\frac{1}{\lambda_N} Z_{f_N}, \psi)|\leq \frac{1}{\lambda_N}
\hcal^{n-1}(Z_{f_N}) \; \|\psi\|_{{ C}^0}. $$ By a density
argument it suffices to prove that the linear statistics
$\frac{1}{\lambda_N} (Z_{f_N}, \psi) - \frac{1}{Vol(M, g)} \int_M
\psi dV_g \to 0$ almost surely in $\hcal_{\infty}$ We know that

\noindent (i) $\lim_{N\rightarrow \infty} \frac{1}{N} \sum_{k\leq
N} \E(\frac{1}{\lambda_k} X^k_{\psi} )= \frac{1}{Vol(M, g)} \int_M
\psi dV_g;$

\noindent(ii)  $Var(\frac{1}{\lambda_N} X^N_{\psi})$
 is bounded on $\hcal_{\infty}$.

  Since $\frac{1}{\lambda_N} X^N_{\psi}$ for
$\{,N=1,2,\ldots\}$ is a sequence of independent random variables
in $\hcal_{\infty}$ with bounded variances, the Kolmogorov strong
law of large numbers  implies that
$$ \lim_{N\rightarrow \infty} \frac{1}{N} \sum_{k\leq
N} (\frac{1}{\lambda_k} X^k_{\psi} )= \frac{1}{Vol(M, g)} \int_M
\psi dV_g$$ almost surely.

\subsection{Complex zeros of random waves}

We now state a complex analogue of the equidistribution of real nodal sets and show that it agrees with the
the limit formula of  Theorem \ref{ZERO}.

We  complexify Riemannian  random waves as
$$f_N^{\C} = \sum_{j = 1}^{d_N} c_{N j} \phi_{N j}^{\C}. $$
We note that the coefficients $c_{N j}$ are real and that the
Gaussian measure on the coefficients remains the real Gaussian
measure $\gamma_N$.
The two point function is the analytic extensions to the totally real anti-diagonal in
$M_{\C} \times M_{\C}$ is therefore 
\begin{equation}\label{CXSPECPROJ}  \E (|f_N(\zeta)|^2) =  \Pi_{I_N}(\zeta, \bar{\zeta}) = \sum_{j:
\lambda_j \in I_k} |\phi_j^{\C}(\zeta)|^2.  \end{equation}

As in the proof of Theorem \ref{ZERO}, the  current of integration over the complex
zero set
$$Z_{f_N^{\C}} = \{\zeta \in M_{\C}: f_N^{\C}  = 0\}
$$ is the $(1,1)$ current defined by
$$\langle [Z_{f_N^{\C}}], \psi \rangle = \int_{Z_{f_N^{\C}}} \psi, \;\; \psi \in \dcal^{m-1, m-1}(M_{\C}), $$
for smooth test forms of bi-degree $(m-1, m-1)$. In terms of
scalar functions $\psi$ we may define $Z_{f_N^{\C}}$ as the
measure,
$$\langle [Z_{f_N^{\C}}], \psi \rangle = \int_{Z_{f_N^{\C}}}  \psi
\omega_g^{m-1}/(m-1)!, $$ where $\omega_g = i \ddbar \rho$ is the
\kahler metric adapted to $g.$

The proof of the next result is close to the proof of Theorem \ref{ZERO} and we therefore
refer to \cite{Z4} for the details:

\begin{theo}\label{ZERORAN} \cite{Z4}  Let $(M, g)$ be  a  real
analytic compact Riemannian manifold. Then for either of the Riemannian random wave 
ensembles 
$$\E_{\gamma_N} \left(\frac{1}{N} [Z_{f_N^{\C}}] \right) \to  \frac{i}{ \pi} \ddbar |\xi|_g,\;\;
 \mbox{weakly in}\;\; \dcal^{' (1,1)} (B^*_{\epsilon} M).  $$
\end{theo}
 As mentioned above, this  result shows that the  complex zeros of the random
waves have the same expected limit distribution found in \cite{Z3}
for real analytic compact Riemannian manifolds with ergodic
geodesic flow.

\section{\label{PERC} Percolation heuristics}

In this final section, we review some of the more speculative conjectures relating nodal sets of both
eigenfunctions and random waves  to  percolation theory.  The conjectures are often quoted and
it therefore seems worthwhile to try to state them precisely. 
The only rigorous result to date regarding
eigenfunctions  is the theorem 
of Nazarov-Sodin on the expected number of nodal domains for random spherical harmonics \cite{NS}  (see  \cite{Z5}  
for a brief over-view).

The percolation conjectures concern the statistics of sizes of nodal domains or nodal components. They are based on the idea that the nodal domains resemble percolation clusters. 
One might measure the `size' of a nodal component $A_{\lambda_j}$
by its hypersurface area $\hcal^{n-1}(A_{\lambda, j})$, and a
nodal domain $D_{\lambda, j}$  by its volume $\hcal^n(D_{\lambda,
j})$ .  Let us restrict to the case of surfaces. For the purposes of this article, we introduce the term
{\it length spectrum}  of the nodal set as the set
\begin{equation} \label{LSP} Lsp(\phi_{\lambda}) =
\{(\hcal^1(C_{\lambda, j}): Z_{\phi_{\lambda}} = \bigcup
C_{\lambda; j}\} \end{equation} of lengths of its components,
counted with multiplicity. It is encoded by the empirical measure
of surface areas\begin{equation} \label{LEMP} d\mu_L =
\frac{1}{\hcal^{1}(Z_{\phi_{\lambda}})} \sum_{C_{\lambda, j}}
\delta_{\hcal^{1}(C_{\lambda, j})} \in \pcal_1(\R),
\end{equation}
(where $\pcal(\Omega)$ is the set of probability measures on
$\Omega$), or equivalently by the length distribution function,
\begin{equation}\label{LDISTFUN}  \lcal_{\lambda}(t) = \sum_{j: \hcal^1(C_{\lambda, j})
 \leq t} \hcal^1(C_{\lambda, j}). \end{equation}  We also consider the {\it area spectrum},
\begin{equation} \label{ASP} Asp(\phi_{\lambda}) =
\{(\hcal^2(A_{\lambda, j}): M \backslash Z_{\phi_{\lambda}} = \bigcup
A_{\lambda; j}\},  \end{equation} encoded by its empirical measure
 It is encoded by the empirical measure
of surface areas\begin{equation} \label{AEMP} d\mu_A =
\frac{1}{\mbox{Area}(M)} \sum_{A_{\lambda, j}}
\delta_{\hcal^{2}(A_{\lambda, j})} \in \pcal_1(\R),
\end{equation} or by
the area distribution function,
\begin{equation}\label{ADISTFUN}  \acal_{\lambda}(t) = \sum_{j: \hcal^2(A_{\lambda, j})
 \leq t} \hcal^2(A_{\lambda, j}). \end{equation}
 Of course, there are some obvious constraints on such spectra;
 e.g. in the analytic case, there  could only exist $O(\lambda)$ components with
$\hcal^{1}$-length of order $1$, and only a bounded number of
order $\lambda$.

In computer graphics of eigenfunctions on plane domains or
surfaces,  one sees many `small' components $C_{\lambda, j}$ of
the nodal set  whose length appears to be   of  order
$\frac{1}{\lambda}$. But one also sees long snaky nodal lines. How
long are they? Do they persist as $\lambda \to \infty$? Roughly
speaking, one may ask what proportion of the components come in
sizes with different orders of magnitude. Of course, this depends
on how many components there are, so it could be simpler to work
with $\lcal(\phi_{\lambda}), \acal(\phi_{\lambda})$.
\begin{itemize}

\item How many components have $\hcal^{n-1}$-surface measure which
is $\geq C \lambda^{\gamma}$ for some given  $0 < \gamma \leq 1$.
It is possible that some individual nodal component has
$\hcal^{n-1}$-surface area commensurate with that of the entire
nodal set, as in the Lewy spherical harmonics with just two or
three nodal components \cite{Lew}.

\item How many components have $\hcal^{n-1}$-surface measure (i.e.
length in dimension two) which is bounded below by a constant $C >
0$ independent of $\lambda$? Such components are sometimes termed
``percolating nodal lines" since their hypersurface volume  is
commensurate with the size of the macroscopic object (i.e. $M$).

\item How many components have $\hcal^{n-1}$-surface measure of
the minimal  order $\frac{1}{\lambda}$?

\end{itemize}

The percolation conjectures relate the asymptotic
distribution of lengths of nodal components and areas of nodal
domains of eigenfunctions  as defined in \eqref{LEMP}-\eqref{AEMP}
to lengths of boundaries and areas of percolation clusters at
criticality. There are different types of conjectures for the {\it
fixed frequency } ensemble and the {\it high frequency cutoff}
ensemble (see \S \ref{RWONB} for the definitions). According to
the random wave hypothesis, the conjectures concerning the fixed
frequency ensemble  (e.g. random spherical harmonics of fixed
degree) should also  apply to nodal sets of eigenfunctions of
quantum chaotic systems.

Percolation theory is concerned with connectivity and transport in
a complex system. In particular, it studies connected clusters of
objects in a random graph. In bond percolation the edges of the
graph are independently open or closed with some probability $p$.
The open edges form a subgraph whose connected components form the
clusters. In site percolation the vertices are open or closed and
an open path is a path through open vertices. The open cluster
$C(v)$ of a vertex is the set of all open vertices which are
connected to $v$ by an open path.

There also exists an analogous continuum percolation theory for
level sets of random functions. We will assume the random
functions are Gaussian Riemannian random waves on a surface. The
main problem is to study the connectivity properties of  level
sets $\{f = t\}$. One imagines a random landscape of lakes and
islands depending on the variable height $t$ of the water,  the
islands being the super-level sets $\{f
> t\}$ of the random functions. For high water levels, the islands
are disconnected, but as the water level is lowered the islands
become more connected. At a critical level $t_c$ they `percolate',
i.e. it is possible to traverse the landscape while remaining on
the land. A review with many illustrations is given by Isichenko
\cite{Isi} (see Section E (c), pages 980-984).  As explained in
\cite{Isi} page 984, the contour lines of a random potential are
associated to hulls of percolation clusters.  Hence the area
spectrum \eqref{ASP} is similar to the set of sizes of connected
clusters in a percolation model.

In the physics literature, the random functions are usually
functions on $\R^2$ (or possibly higher dimensional $\R^n$) and
the Gaussian measure on the space of functions corresponds to a
Hilbert space inner product. The Hilbert space is usually taken to
be a Sobolev space, so that the inner product has the form $\int
w(\xi) |\hat{f}(\xi)|^2 d \xi$ (where $\hat{f}$ is the Fourier
transform of $f$) and  $w(\xi) = |\xi|^{2(1 + \zeta)}$. The case
$\zeta = 0$ is known as the Gaussian free field (or massless
scalar field) and is quite special in two dimensions since then
the inner product $\int_{\R^2} |\nabla f|^2 dx$ is conformally
invariant. There are rigorous results on level sets of
discretizations of the Gaussian free field and their continuum
limits in \cite{SS,Mi}, with authoritative comments on the
physics literature.

For purposes of this exposition, we assume the Riemannian random
waves fall are of the types discussed in \S \ref{RWONB}. In all
cases, we truncate the frequency above a spectral parameter
$\lambda$ and consider asymptotics as $\lambda \to \infty$. In
this high frequency limit, the random waves oscillate more rapidly
on the length scale $\frac{1}{\lambda}$. Since the conjectures and
results depend strongly on the chosen weight $w$, we break up the
discussion into two cases as in \S \ref{RWONB}: the high frequency
cutoff ensemble and the fixed frequency ensemble.
For each ensemble we let  $\E_{\lambda}$ denote the expectation
with respect to the Gaussian measure on the relevant space of
linear combinations. Then we may ask for the asymptotic behavior
of the expected distribution of lengths of nodal lines, resp. area
of nodal domains
\begin{equation}\label{EXPEMP}  \E_{\lambda} d\mu_{L}, \;\;\;\;\;\; \E_{\lambda}
d\mu_A, \end{equation} where $d\mu_L,$ resp. $d\mu_A$ are the
empirical measures of lengths \eqref{LEMP} of nodal lines, resp.
areas \eqref{AEMP} of nodal domains.

\subsection{High frequency cutoff ensembles}

The distribution of contour lengths of certain  Gaussian random
surfaces over $\R^2$ was studied at the physics level of rigor in
\cite{KH}. They define the Gaussian measure as $e^{- f_{\zeta}(h)}
dh$ where the `free energy' is defined by
$$f_{\zeta}(h) = \frac{K}{2} \int_{\R^2} \chi(\frac{|\xi|}{\lambda}) \; |\hat{h}(\xi)|^2
|\xi|^{2 (1 + \zeta)} d \xi,$$ where $\chi$ is a cutoff function
to $[0, 1]$ (they use the notation  $a$ for $ \frac{1}{\lambda}$
in our notation). When $\zeta = 0$, this is a truncated Gaussian
free field (truncated at frequencies $\leq \lambda$) and its
analogue on a surface $(M, g)$ is the Riemannian random wave model
with spectral interval $[0, \lambda]$ and weight $w(\lambda) =
\frac{1}{\lambda}$. The parameter $\zeta$ is referred to as the
'roughness exponent' in the physics literature. In the case of the
Gaussian free field $\zeta = 0$ the inner product is the Dirichlet
inner product $\int_{\R^2} |\nabla f|^2 dx$.

An important feature of the ensembles is scale-invariance. In the
special case $\zeta = 0$ (and dimension two), the Dirichlet inner
product $\int_M |\nabla f|^2_g dA_g$ is conformally invariant,
i.e. invariant under conformal changes $g \to e^u g$ of the
Riemannian metric. When $\zeta \not= 0$ this is not the case, but
it is  assumed in  \cite{KH}   that the fluctuations of the random
Gaussian surface with height function $h$  are invariant under the
rescaling $h(r) \to c^{- \zeta} h(c r)$ for any $c > 1$. The
authors of \cite{KH} then  make a number of conjectures concerning
the distribution of contour lengths, which we interpret as
conjectures concerning $\E d\mu_L$. First, they consider contours
(i.e. level sets) through a fixed point $x_0$ and measure its
length with the re-scaled arc-length measure $\lambda ds$, i.e.
with arclength $s$ in units of $\frac{1}{\lambda}$.  They define
the fractal dimension of a nodal line component as the dimension
$D$ so that $s \sim R^{D}$ where $R$ is the radius of the nodal
component (i.e. half the diameter). They define $P(s)$ as the
probability density that the contour through $x_0$ has length $s$.
The  principal claim is  that $P(s) \sim s^{- \tau - 1}$ satisfies
a power law for some exponent $\tau$ (\cite{KH} (4)). They also
defines the distribution of loop lengths $\tilde{P(s)} \sim
P(s)/s$ as the probability density that a random component has
length $s$. We interpret their  $\tilde{P(s)}$ as the  density of
$\lim_{\lambda \to \infty} \E d\mu_L$ with respect to $ds$ on
$\R$.  We thus interpret their conjecture as saying that a unique
weak* limit of this family of measures exists and has a density
relative to $ds$ with a power law decay as above.

The claims are based in part on scaling properties of the contour
ensemble. They also are based in part on
 the expectation that,  at `criticality', the key percolation `exponents'
 of power laws  are universal  and therefore should be the same for the
discrete and continuum percolation theories (see e.g.
\cite{IsiK}). In \cite{KH}, the authors suggest that when a
certain roughness exponent $\zeta $ vanishes (the critical
models), the continuum problem is related to the four-state Potts
model. The $q$-state Potts model is an Ising type spin  model on a
lattice where  each spin can take one of $q$ values. It is known
to be related to connectivity and percolation problems on a graph
 \cite{Bax,Wu}.

They compute $D, \tau$ by relating both to another exponent $x_1$
defined by a ``contour correlation function" $\gcal_1(r)$, which
measures the probability that points at $x,x + r$ lie on the same
contour loop. They claim that $\gcal_1(r) \sim |r|^{- 2 x_1}$.
They claim that $D (3 - \tau) = 2 - 2 x_1$ and $D (\tau - 1) = 2 -
\zeta$. As a result, $D = 2 - x_1 - \zeta/2, \tau - 1 = \frac{2 -
\zeta}{2 - x_1 - \zeta/2}$. From the mapping to the four-state
Potts model, they conclude that $x_1 = \half$.

There  exist rigorous results in \cite{SS,Mi}  relating
discretizations of the Gaussian free field (rather than high
frequency truncations) to the percolation models. They prove that
in various senses, the zero set of the discrete Gaussian free
field tends to an $SLE_4$ curve. It does not seem to be known at
present if zero sets of the high frequency truncation of the
Gaussian free field  also tends in the same sense to an $SLE_4$
curve. Note that the SLE curves are interfaces and that one
must select one component of the zero set that should tend
to an SLE curve.  There might exist modified conjectures
regarding CLE curves.


To determine the `critical exponents' in continuum percolation, it
is tempting to find a way to `map' the continuum problem to a
discrete percolation model. A geometric `map' from a random wave
to a graph  is to associate to the random function its Morse-Smale decomposition,  known in the physics literature as the  ``Morse
skeleton" (see \S \ref{DECOMP} or  \cite{Web} for an extensive exposition).  As
discussed in \cite{Wei}, and as illustrated in Figure 10 of
\cite{Isi},
 the Morse complex of the random function plays the
role of the lattice in lattice percolation theory. 

\subsection{Fixed frequency ensembles}

We now consider Riemannian random waves of asymtotically fixed
frequency $\lambda$, such as random spherical harmonics of fixed
degree or Euclidean random plane waves of fixed eigenvalue. In
this case the weight is a delta function at the frequency. One
would expect different behavior in the level sets since only one
frequency is involved rather than the superposition of waves of
all frequencies $\leq \lambda$.

A  recent exposition in the specific setting of random Euclidean
eigenfunctions of fixed frequency is given by \cite{EGJS}.
 The level sets play the role of open paths.
Super-level sets are compared to clusters of sites in a critical
2D percolation model, such as bond percolation on a lattice. Each
site may of the percolation model may be visualized as a disc of
area $\frac{2 \pi^2}{\lambda^2}$, i.e. as a small component.  The
nodal domains may be thought of as connected clusters of a number
$n$ such discs.  Since nodal domains are connected components in
which the eigenfunction is either positive $+$ or negative $-$,
they are analogous to clusters of `open' or `closed' vertices.

The main conjectures in this fixed frequency ensemble are due to
E. Bogolmony and C. Schmidt \cite{BS}. They conjecture that the
continuum percolation problem should belong to the same
universality class as the Potts model at a certain critical point
(where $q$ is related to a certain temperature) for a large
rectangular lattice and that the nodal lines in the $\lambda \to
\infty$ limit tend to $SLE_6$ curves. This is similar to the
predictions of \cite{KH} but for a very different ensemble where
there is little apriori reason to expect conformal invariance in
the limit. There are parallel conjectures in \cite{BBCF} for
zero-vorticity isolines in 2D turbulence, which are also
conjectured to tend to $SLE_6$ curves.  They remark (page 127)
that this limit is surprising since continuous percolation models
assume short-correlations in the height functions whereas the
vorticity field correlations decay only like $r^{-4/3}$. They
write, ``When the pair correlation function falls off slower than
$r^{-3/2}$, the system is not expected generally to belong to the
universality class of uncorrelated percolation and to be
conformally invariant". The same remarks apply to the fixed
frequency ensemble, where the correlation function is the spectral
projection $\Pi_{[\lambda, \lambda + 1]}(x, y)$ for a fixed
frequency. In this case, the correlations decay quite slowly as
$r^{-\half}$; we refer to  \cite{BS2} for this background and also
for an argument why the nodal sets should nevertheless resemble
conformally invariant $SLE_g$ curves.

 If the nodal lines in the fixed frequency model are equivalent to
 the critical percolation model, then the
`probability' of finding a nodal domain of area $s$ should decay
like $s^{- \tau}$ where $\tau = \frac{187}{91}
> 2$ (see \cite{SA}, p. 52 for the percolation
theory result).  Under some shape assumptions adopted in
\cite{EGJS}, it is equivalent that  the probability of finding
clusters consisting of $n$ discs is of order $n^{- \tau}$. For
random spherical harmonics, one may  ask for the probability that
a spherical harmonic of degree $N$ has size $n$. For a fixed $(M,
g)$ with simple eigenvalues, this notion of probability from
percolation theory  does not make sense, but we might assume that
the number of of nodal components is of order $\lambda^2$ and ask
what proportion of the nodal components has size $1$.  To obtain a
percolating nodal line, one would need a cluster with $n =
\lambda$ sites, and thus the proportion of such nodal components
to the total number would be of order $\lambda^{- \tau}$. Thus, if
there are $C \lambda^2$ total components, the number of such
components would be around $\lambda^{2 - \tau} =
\lambda^{-\frac{5}{91}} < 1$, so the model seems to predict that
such macroscopic  nodal lines are quite rare. It also predicts
that the `vast majority' of nodal components are close to the
minimal size, which does not seem so evident from the computer
graphics.


\begin{thebibliography}{MMMM}







\bibitem[AP]{AP} J. C. Alvarez Paiva and E. Fernandes,  Gelfand transforms and Crofton formulas. Selecta Math. (N.S.) 13 (2007), no. 3, 369--390.

\bibitem[AP2]{AP2}  J. C. Alvarez Paiva and G.  Berck, What is wrong with the Hausdorff measure in Finsler spaces. Adv. Math. 204 (2006), no. 2, 647–663.



\bibitem[Ar]{Ar} S. Ariturk,  
Lower bounds for nodal sets of Dirichlet and Neumann eigenfunctions, to appear in Comm. Math. Phys. ( arXiv:1110.6885).

\bibitem[Ba]{Ba} L. Bakri,  Critical set of eigenfunctions of the
Laplacian, arXiv:1008.1699.


\bibitem[Bae]{Bae} C. B\"ar,
On nodal sets for Dirac and Laplace operators. Comm. Math. Phys.
188 (1997), no. 3, 709--721.

\bibitem[Bax]{Bax} R. J. Baxter, Potts model at the critical temperature,  Journal of Physics C: Solid State Physics
6 (1973), L445.


\bibitem[BBCF]{BBCF}  D. Bernard,  G. Boffetta, A. Celani, and G.
Falkovich, Conformal invariance in two-dimensional turbulence,
nature. physics Vol. 2 (2002), p. 134.


\bibitem[Ber]{Ber} M. V. Berry,
Regular and irregular semiclassical wavefunctions. J. Phys. A 10
(1977), no. 12, 2083-2091.


\bibitem[Bers]{Bers} L. Bers,
Local behavior of solutions of general linear elliptic equations.
Comm. Pure Appl. Math. 8 (1955), 473--496.


\bibitem[BGS]{BGS} G. Blum, S. Gnutzmann and U. Smilansky, Nodal domain statistics: A Criterion
for quantum chaos,  Phys. Rev. Lett. 88, 114101 (2002).

\bibitem[BDS]{BDS} E. Bogomolny, R.  Dubertrand, and C.  Schmit,
SLE description of the nodal lines of random wavefunctions. J.
Phys. A 40 (2007), no. 3, 381-395.

\bibitem[BS]{BS} E. Bogomolny and C. Schmit, Percolation model
for nodal domains of chaotic wave functions, Phys. Rev. Letters 88
(18) (2002), 114102-114102-4.

\bibitem[BS2]{BS2} E. Bogomolny and C. Schmit,
Random wavefunctions and percolation. J. Phys. A 40 (2007), no.
47, 14033-14043.




\bibitem[Br]{Br} J. Br\"uning, \"Uber Knoten von Eigenfunktionen des Laplace-Beltrami Operators", Math.
 Z. 158 (1978), 15--21.

\bibitem[BSZ1]{BSZ1}  P. Bleher, B Shiffman, and S. Zelditch, 
 Universality and scaling of zeros on symplectic manifolds. {\it  Random matrix models and their
 applications}, 31--69, Math. Sci. Res. Inst. Publ., 40, Cambridge Univ. Press, Cambridge, 2001.



\bibitem[BSZ2]{BSZ2}  P. Bleher, B. Shiffman and S. Zelditch,  Universality
and scaling of correlations between zeros on complex manifolds,
 Invent. Math. 142 (2000),
no. 2, 351--395. http://xxx.lanl.gov/abs/math-ph/9904020.


\bibitem[Bourg]{Bour} J. Bourgain, {\em Geodesic restrictions and $L^p$-estimates
for eigenfunctions of Riemannian surfaces}, Linear and complex
analysis, 27--35, Amer. Math. Soc. Tranl. Ser. 2, 226, Amer. Math. Soc., Providence, RI, 2009.


\bibitem[BZ]{BZ}  J. Bourgain and  Z. Rudnick,
On the nodal sets of toral eigenfunctions. 
Invent. Math. 185 (2011), no. 1, 199–23.

\bibitem[Bou]{Bou} L. Boutet de Monvel,
Convergence dans le domaine complexe des s\'eries de fonctions
propres.  C. R. Acad.\ Sci.\ Paris S\'er. A-B 287 (1978), no.\ 13,
A855--A856.




\bibitem[BGT]{BGT} N. Burq, P.  G\'erard, and N. Tzvetkov,
 Restrictions of the Laplace-Beltrami eigenfunctions to submanifolds. Duke Math. J. 138 (2007), no. 3, 445--486





\bibitem[Bu]{Bu} N. Burq,
Quantum ergodicity of boundary values of eigenfunctions: A control
theory approach, to appear in Canadian Math. Bull.
(math.AP/0301349).






\bibitem[Ch1]{Ch1} S. Y. Cheng,
 Eigenfunctions and eigenvalues of Laplacian. Differential geometry
 (Proc. Sympos. Pure Math., Vol. XXVII, Stanford Univ., Stanford, Calif., 1973), Part 2, pp. 185--193.
 Amer. Math. Soc., Providence, R.I., 1975.


  \bibitem[Ch2]{Ch2} S. Y. Cheng,
Eigenfunctions and nodal sets. Comment. Math. Helv. 51 (1976), no.
1, 43--55.


\bibitem[CTZ]{CTZ} H. Christianson, J. A. Toth and S. Zelditch, Quantum ergodic restriction for Cauchy Data: Interior QUE and restricted QUE  (arXiv:1205.0286). 




\bibitem[CM]{CM}
T. H. Colding and  W. P. Minicozzi II, Lower bounds for nodal sets
of eigenfunctions. Comm. Math. Phys. 306 (2011), no. 3, 777 - 784.


\bibitem[CV]{CV} Y.Colin de Verdi\`ere, Ergodicit\'e et fonctions propres du
Laplacien, Comm.Math.Phys. 102 (1985), 497-502.

%




\bibitem[C]{C}  R. Cooper, 
The extremal values of Legendre polynomials and of certain related functions. 
Proc. Cambridge Philos. Soc. 46, (1950). 549–55.

\bibitem[Dong]{Dong} R-T. Dong,
 Nodal sets of eigenfunctions on Riemann surfaces. J. Differential Geom. 36 (1992), no. 2, 493--506.

 \bibitem[DF]{DF} H. Donnelly and C. Fefferman, Nodal sets of eigenfunctions on
Riemannian manifolds, Invent. Math. 93 (1988), 161-183.

\bibitem[DF2]{DF2} H. Donnelly and C. Fefferman,
Nodal sets of eigenfunctions: Riemannian manifolds with boundary.
Analysis, et cetera, 251--262, Academic Press, Boston, MA, 1990.


%

\bibitem[DF3]{DF3} H. Donnelly and C. Fefferman,
 Growth and geometry of eigenfunctions of the Laplacian. Analysis
  and partial differential equations, 635--655, Lecture Notes in Pure and Appl. Math., 122, Dekker, New York,
  1990.

\bibitem[DF4]{DF4} H. Donnelly and C. Fefferman,  Nodal sets for eigenfunctions of the Laplacian on surfaces. J. Amer. Math. Soc. 3 (1990), no. 2, 333--353.

\bibitem[DSZ]{DSZ}  M. R. Douglas, B. Shiffman, and S. Zelditch,
Critical points and supersymmetric vacua. II. Asymptotics and extremal metrics. 
J. Differential Geom. 72 (2006), no. 3, 381–427.

\bibitem[DZ]{DZ}   S. Dyatlov, and M. Zworski, 
Quantum ergodicity for restrictions to hypersurfaces
 (arXiv:1204.0284).

\bibitem[EK]{EK} Y. Egorov and V.  Kondratiev, {\it On spectral theory of elliptic operators}. Operator Theory:
 Advances and Applications, 89. Birkh\"auser Verlag, Basel, 1996.

\bibitem[EGJS]{EGJS} Y. Elon, S. Gnutzmann, C. Joas, and U.  Smilansky,
 Geometric characterization of nodal domains: the area-to-perimeter ratio. J. Phys. A 40 (2007), no. 11, 2689�2707.


\bibitem[EJN]{EJN} A.  Eremenko,  D. Jakobson and N.  Nadirashvili, 
On nodal sets and nodal domains on S2 and R2.
Festival Yves Colin de Verdière. 
Ann. Inst. Fourier (Grenoble) 57 (2007), no. 7, 2345–2360.



\bibitem[Fed]{Fed} H. Federer, {\it Geometric measure theory.}
 Die Grundlehren der mathematischen Wissenschaften, Band 153 Springer-Verlag New York Inc., New York
1969.



\bibitem[FGS]{FGS} G. Foltin, S.  Gnutzmann, and U.  Smilansky,
 The morphology of nodal lines---random waves versus percolation. J. Phys. A 37 (2004), no. 47, 11363--11371.

\bibitem[GaL]{GaL} N. Garofalo and F. H. Lin,
Monotonicity properties of variational integrals, $A\sb p$ weights
and unique continuation. Indiana Univ. Math. J. 35 (1986), no. 2,
245--268.

\bibitem[GaL2]{GaL2} -----------,
Unique continuation for elliptic operators: a
geometric-variational approach. Comm. Pure Appl. Math. 40 (1987),
no. 3, 347--366


\bibitem[GS]{GS} I. M. Gelfand and M. Smirnov, Lagrangians satisfying Crofton formulas, Radon transforms, and nonlocal differentials. Adv. Math. 109 (1994), 188--227.

\bibitem[GL]{GL} P.G\'erard and E.Leichtnam, Ergodic properties of
eigenfunctions for the Dirichlet problem, Duke Math J. 71 (1993),
559-607.


\bibitem[Gi]{Gi}  V. M. Gichev,  A Note on the Common Zeros of Laplace Beltrami Eigenfunctions. Ann. Global Anal. Geome. 26, 201–208 (2004).

 \bibitem[GLS]{GLS} F. Golse, E. Leichtnam, and M. Stenzel,
 Intrinsic microlocal analysis and inversion formulae for
  the heat equation on compact real-analytic Riemannian manifolds.
  Ann. Sci.\ \'Ecole Norm.\ Sup. (4) 29 (1996), no.\ 6, 669--736.



\bibitem[GSj]{GSj} A. Grigis and J. Sj\"ostrand, {\it Microlocal analysis for
 differential operators}, London Math. Soc. Lecture Notes 196 (1994).

\bibitem[GS1]{GS1}  V. Guillemin and M. Stenzel, Grauert tubes and the homogeneous Monge-Amp�re equation. J. Differential Geom. 34 (1991), no. 2, 561--570.


\bibitem[GS2]{GS2}  -----------,   Grauert tubes and the homogeneous Monge-Amp�re equation. II. J. Differential Geom. 35 (1992), no. 3, 627--641.










\bibitem[H2]{H2} Q. Han, Nodal sets of harmonic functions, Pure
and Applied Mathematics Quarterly 3 (3) (2007), 647-688.


\bibitem[HHL]{HHL} Q. Han, R. Hardt, and F. H.  Lin,
 Geometric measure of singular sets of elliptic equations. Comm. Pure Appl. Math. 51 (1998), no. 11-12, 1425--1443.

\bibitem[H]{H} Q. Han and F.H. Lin  {\it Nodal sets of solutions of elliptic
differential equations}, book in preparation (online at http://www.nd.edu/~qhan).

\bibitem[HL]{HL}  X. Han and G.  Lu, 
A geometric covering lemma and nodal sets of eigenfunctions. (English summary) 
Math. Res. Lett. 18 (2011), no. 2, 337–352



\bibitem[HHON]{HHON} R. Hardt, M.  Hoffmann-Ostenhof, T.  Hoffmann-Ostenhof and N.  Nadirashvili,
Critical sets of solutions to elliptic equations.  J. Differential
Geom. 51 (1999), no. 2, 359--373.

\bibitem[HaS]{HaS} R. Hardt and L.  Simon,  Nodal sets for solutions of elliptic equations. J. Differential Geom. 30 (1989), no. 2, 505--522.





%



\bibitem[HZ]{HZ} A. Hassell and S. Zelditch,  Quantum ergodicity of boundary values of eigenfunctions. Comm. Math. Phys. 248 (2004), no. 1, 119--168.

\bibitem[Hel]{Hel} S. Helgason, {\it Topics in harmonic analysis on homogeneous spaces.}
 Progress in Mathematics, 13. Birkh\"auser, Boston, Mass., 1981.

 \bibitem[HEJ]{HEJ} E. J. Heller, Gallery (Quantum random waves),
 http://www.ericjhellergallery.com/.

\bibitem[He]{He} H.  Hezari,  Complex zeros of eigenfunctions of 1D Schrödinger operators. Int. Math. Res. Not. IMRN 2008, no. 3, Art. ID rnm148.

\bibitem[HS]{HS} H. Hezari and C. D. Sogge, 
A natural lower bound for the size of nodal sets,
to appear in Analysis and PDE (arXiv:1107.3440).

\bibitem[HW]{HW}  H.  Hezari and Z.  Wang, 
Lower bounds for volumes of nodal sets: an improvement of a result of Sogge-Zelditch, to appear
arXiv:1107.0092. 



\bibitem[HC]{HC}   D. Hilbert and R. Courant, {\it Methods of mathematical physics},  Vol. I and   Vol.
II:. Interscience Publishers
 (John Wiley $\&$ Sons), New York-Lon don 1962.




\bibitem[HoI-IV]{HoI-IV}  L. H\"ormander, {\it Theory of
Linear Partial Differential Operators I-IV}, Springer-Verlag, New
York (1985).


\bibitem[Hu]{Hu}  R. Hu, 
$L^p$  norm estimates of eigenfunctions restricted to submanifolds. 
Forum Math. 21 (2009), no. 6, 1021 - 1052.


\bibitem[Isi]{Isi} M. B. Isichenko,
 Percolation, statistical topography, and transport in random media. Rev. Modern Phys. 64 (1992), no. 4, 961--1043.

\bibitem[IsiK]{IsiK} M. B. Isichenko and J.  Kalda,
Statistical topography. I. Fractal dimension of coastlines and
number-area rule for islands.
J. Nonlinear Sci. 1 (1991), no. 3, 255-277

 \bibitem[JN]{JN} D. Jakobson and N. Nadirashvili,  Eigenfunctions with few critical points. J. Differential Geom. 53 (1999), no. 1, 177--182.

\bibitem[JN2]{JN2} -----------,
Quasi-symmetry of $L\sp p$ norms of eigenfunctions.  Comm. Anal.
Geom. 10 (2002), no. 2, 397--408.

\bibitem[JM]{JM} D. Jakobson and D. Mangoubi,   Tubular Neighborhoods of Nodal Sets and Diophantine
Approximation,  Amer. J. Math. 131 (2009), no. 4, 1109--1135
(arXiv:0707.4045).

\bibitem[JL]{JL} D. Jerison and G. Lebeau,
Nodal sets of sums of eigenfunctions. Harmonic analysis and
partial differential equations (Chicago, IL, 1996), 223--239,
Chicago Lectures in Math., Univ. Chicago Press, Chicago, IL, 1999.

\bibitem[JJ]{JJ} J.  Jung,  Zeros of eigenfunctions on hyperbolic surfaces lying on a curve 
(arXiv: 1108.2335).








\bibitem[KH]{KH} J. Kondev and C. L. Henley, Geometrical
exponents of contour loops on random Gaussian surfaces, Phys. Rev.
Lett. 74 (1995), 4580 - 4583.

\bibitem[KHS]{KHS} J. Kondev, C. L.  Henley, and D.G.  Salinas,
 Nonlinear measures for characterizing rough surface morphologies. Phys. Rev. E, 61 (2000), 104-125.

\bibitem[Kua]{Kua} I. Kukavica, Nodal volumes for eigenfunctions of analytic regular elliptic problems.
 J. Anal. Math. 67 (1995), 269--280.

 \bibitem[Ku]{Ku} -----------,
Quantitative uniqueness for second-order elliptic operators. Duke
Math. J. 91 (1998), no. 2, 225--240.




\bibitem[LS1]{LS1} L. Lempert and R.  Sz\"oke,
 Global solutions of the homogeneous complex Monge-Amp\`ere equation and
 complex structures on the tangent bundle of Riemannian manifolds. Math. Ann. 290 (1991), no. 4, 689--712.

\bibitem[LS2]{LS2} -----------,
The tangent bundle of an almost complex manifold, Canad. Math.
Bull. 44 (2001), no. 1, 70--79.

\bibitem[Lew]{Lew} H.  Lewy,
On the minimum number of domains in which the nodal lines of
spherical harmonics divide the sphere. Comm. Partial Differential
Equations 2 (1977), no. 12, 1233-1244.

\bibitem[Ley]{Ley}  J. Leydold,  On the number of nodal domains of spherical harmonics. Topology 35 (1996), no. 2, 301–321.

\bibitem[Lin]{Lin} F.H.  Lin, Nodal sets of solutions of elliptic
and parabolic equations. Comm. Pure Appl. Math. 44 (1991), no. 3,
287--308.



\bibitem[Man]{Man} D. Mangoubi,  A Remark on Recent Lower Bounds for Nodal Sets,
 Comm. Partial Differential Equations 36 (2011), no. 12, 2208--2212
(arXiv:1010.4579.)

\bibitem[Man2]{Man2} D. Mangoubi,  The Volume of a Local Nodal Domain,  J. Topol. Anal. 2 (2010), no. 2,
259--275 ( arXiv:0806.3327).

\bibitem[Man3]{Man3} D. Mangoubi,
On the inner radius of a nodal domain. Canad. Math. Bull. 51
(2008), no. 2, 249-260.






\bibitem[Me]{Me} A. D. Melas,
On the nodal line of the second eigenfunction of the Laplacian in
${\bf R}^2$R2. J. Differential Geom. 35 (1992), no. 1, 255-263.

\bibitem[Mi]{Mi} J. Miller,  Universality for SLE(4),
arXiv:1010.1356.





\bibitem[NJT]{NJT} N. Nadirashvili, D. Jakobson, and J.A. Toth,
Geometric properties of eigenfunctions. (Russian) Uspekhi Mat. Nauk 56 (2001), no. 6(342), 67--88; translation in Russian Math. Surveys 56 (2001), no. 6, 1085--1105


\bibitem[NPS]{NPS} F. Nazarov, L.  Polterovich and M.  Sodin,
 Sign and area in nodal geometry of Laplace eigenfunctions. Amer. J. Math. 127 (2005), no. 4, 879--910.

\bibitem[NS]{NS} F. Nazarov and  M. Sodin,  On the Number of Nodal Domains of Random Spherical
Harmonics. Amer. J. Math. 131 (2009), no. 5, 1337-1357
(arXiv:0706.2409).

\bibitem[Neu]{Neu} J. Neuheisel,  "Asymptotic distribution of nodal sets on spheres", PhD Thesis,
Johns Hopkins University, Baltimore, MD 2000, 1994,
http://mathnt.mat.jhu.edu/mathnew/Thesis/joshuaneuheisel.pdf.

\bibitem[Nic]{Nic}
L.  I. Nicolaescu,  
Critical sets of random smooth functions on compact manifolds
(arXiv:1008.5085).


\bibitem[P]{P} A. Pleijel, Remarks on Courant's nodal line theorem,  Comm. Pure Appl. Math., 9, 543-550 (1956).

\bibitem[Po]{Po} I. Polterovich, Pleijel's nodal domain theorem for free membranes,  Proc. Amer. Math. Soc. 137 (2009), no. 3, 1021–1024   (arXiv:0805.1553).

\bibitem[PS]{PS}
 L. Polterovich and M.  Sodin, Nodal inequalities on surfaces. Math. Proc. Cambridge Philos. Soc. 143 (2007), no. 2, 459--467
 (arXiv:math/0604493).


\bibitem[R]{R} J.  Ralston,
Gaussian beams and the propagation of singularities. {\it Studies
in partial differential equations,} 206--248, MAA Stud. Math., 23,
Math. Assoc. America, Washington, DC, 1982.



\bibitem[Reu]{Reu} 	
M. Reuter. Hierarchical Shape Segmentation and Registration via Topological Features of Laplace-Beltrami Eigenfunctions. International Journal of Computer Vision 89 (2), pp. 287-308, 2010.

\bibitem[Reu2]{Reu2} M . Reuter, {\it Laplace Spectra for Shape Recognition}, Books on Demand  (2006). 





%


%




\bibitem[SY]{SY} R.  Schoen and S. T. Yau, Lectures on differential geometry. . Conference Proceedings and Lecture Notes in Geometry and Topology, I. International Press, Cambridge, MA, 1994.



\bibitem[Sh.1]{Sh.1} A.I.Shnirelman, Ergodic properties of eigenfunctions,
Usp.Math.Nauk. 29/6 (1974), 181-182.



\bibitem[SS]{SS} U. Smilansky and H.-J. St\"ockmann, Nodal Patterns in Physics and Mathematics,
 The European Physical Journal Special Topics
Vol. 145 (June 2007).













\bibitem[Sog]{Sog}  C. D. Sogge, Concerning the $L^p$ norm of spectral clusters for
second-order elliptic operators on compact manifolds, J. Funct.
Anal. 77 (1988), 123--138.

\bibitem[Sog2]{Sog2} 
 C. D. Sogge, {\em Kakeya-Nikodym averages and $L^p$-norms
of eigenfunctions}, (arXiv:0907.4827) to appear Tohoku Math.~J (centennial edition).

\bibitem[STZ]{STZ}  C.D. Sogge,  J. A. Toth and S. Zelditch,  About the blowup of quasimodes on Riemannian manifolds. J. Geom. Anal. 21 (2011), no. 1, 150–173.


\bibitem[SoZ]{SoZ} C. Sogge and S.  Zelditch,
 Lower bounds on the hypersurface
measure of nodal sets, Math. Research Letters 18 (2011), 27-39
(arXiv:1009.3573).

\bibitem[SoZ2]{SoZ2} 
C.D. Sogge and S.  Zelditch,  On eigenfunction restriction estimates and $L^4$-bounds for compact surfaces 
with nonpositive
curvature (arXiv:1108.2726).

\bibitem[SoZ3]{SoZ3} C.D.  Sogge and S.  Zelditch, Concerning the $L^4$ norms of
typical eigenfunctions on compact surfaces,  (arXiv:1011.0215).

\bibitem[SA]{SA} D. Stauffer and A. Aharony, {\it Introduction to
Percolation theory}, Taylor and Francis, London (1994).


\bibitem[Sz]{Sz} G. Szeg\"o,  Inequalities for the zeros of Legendre polynomials and related functions. Trans. Amer. Math. Soc. 39 (1936), no. 1, 1–17.

\bibitem[Sz2]{Sz2}  G. Szeg\"o. On the relative extrema of Legendre polynomials. Boll. Un. Mat. Ital. (3) 5, (1950). 120–121.



\bibitem[Taa]{Taa} D. Tataru, On the regularity of boundary traces for the wave equation, Ann. Scuola Norm. Sup. Pisa Cl. Sci. (4) 26 (1998), 185 -- 206.





%

\bibitem[TZ]{TZ} J. A. Toth and S. Zelditch,
Counting Nodal Lines Which Touch the Boundary of an Analytic
Domain, Jour. Diff. Geom. 81 (2009), 649- 686 (arXiv:0710.0101).

\bibitem[TZ2]{TZ2} J. A. Toth and S. Zelditch,  Quantum ergodic restriction
theorems, I: interior hypersurfaces in analytic domains,, {\it Ann. H. Poincar\'e}  13, Issue 4 (2012), Page 599-670 (arXiv:1005.1636).

\bibitem[TZ3]{TZ3} J. A. Toth and S. Zelditch,  Quantum ergodic restriction
theorems, II: manifolds without boundary (arXiv:1104.4531).

\bibitem[U]{U} K. Uhlenbeck,
Generic properties of eigenfunctions. Amer. J. Math. 98 (1976),
no. 4, 1059--1078.





\bibitem[Web]{Web} J. Weber, The Morse-Witten complex via dynamical systems. Expo. Math. 24 (2006), no. 2,
127-159.

\bibitem[Wei]{Wei} A. Weinrib,  Percolation threshold of a two-dimensional continuum system. Phys. Rev. B (3) 26 (1982), no. 3,
1352-1361.

\bibitem[Wig]{Wig} I. Wigman, On the distribution of the nodal sets of random spherical harmonics. J. Math. Phys. 50 (2009), no. 1,
013521.

\bibitem[Wu]{Wu} F. Y. Wu,  Percolation and the Potts model. J. Statist. Phys. 18 (1978), no. 2, 115-123.


 \bibitem[Y1]{Y1} S.T. Yau, Survey on partial differential equations in differential geometry.
 {\it Seminar on Differential Geometry}, pp. 3--71, Ann. of Math. Stud., 102, Princeton Univ. Press, Princeton, N.J., 1982.




\bibitem[Y2]{Y2} -----------,   Open problems in geometry.
{\it  Differential geometry: partial differential equations on
manifolds} (Los Angeles, CA, 1990), 1--28, Proc. Sympos. Pure
Math., 54, Part 1, Amer. Math. Soc., Providence, RI, 1993.

\bibitem[Y3]{Y3} -----------,  A note on the distribution of critical points of eigenfunctions, Tsing Hua Lectures in Geometry and Analysis 315--317, Internat. Press, 1997.



\bibitem[Z1]{Z1} S. Zelditch,
Uniform distribution of eigenfunctions on compact hyperbolic
surfaces. Duke Math. J. 55 (1987), no. 4, 919--941





\bibitem[Z2]{Z2} S. Zelditch,,  Complex zeros of real ergodic eigenfunctions. Invent. Math. 167 (2007), no. 2, 419--443.

\bibitem[Z3]{Z3} S. Zelditch, Ergodicity and intersections of
nodal sets and eigenfunctions on real analytic surfaces (preprint, 2012).

\bibitem[Z4]{Z4} S. Zelditch,
 Real and complex zeros of Riemannian random waves. {\it Spectral analysis in geometry and number theory}, 321-342, Contemp.
  Math., 484, Amer. Math. Soc., Providence, RI, 2009.

  \bibitem[Z5]{Z5} S. Zelditch, Local and global analysis of eigenfunctions on Riemannian manifolds.
  {\it  Handbook of geometric analysis. No. 1}, 545-658, Adv. Lect. Math. (ALM), 7, Int. Press, Somerville, MA, 2008.

  \bibitem[Z6]{Z6} S. Zelditch, New Results in Mathematics of Quantum Chaos, Current
Developments in Mathematics 2009, p. 115- 202 (arXiv:0911.4312).


 \bibitem[Z7]{Z7} S. Zelditch,  Kuznecov sum formulae and
Szego limit formulae on manifolds, Comm.\ PDE {\bf 17} (1\&2)
(1992), 221--260.

\bibitem[Z8]{Z8} 
S. Zelditch,  Pluri-potential  theory on Grauert tubes of  real
analytic Riemannian manifolds, I to appear in  Proc. Symp. Pure Math. volume for the Summer 2010 Dartmouth conference on spectral geometry.



\bibitem[Z9]{Z9} S. Zelditch, Nodal sets of
eigenfunctions in the completely integrable case, (in
preparation).


\bibitem[ZZw]{ZZw} S.Zelditch and M.Zworski,  Ergodicity of eigenfunctions
for ergodic billiards,  Comm.Math. Phys. 175 (1996), 673-682.

























\bibitem[Zw]{Zw}  M. Zworski, {\it  Semiclassical
 analysis},  to appear   in Graduate Studies in Mathematics, AMS, 2012.

\end{thebibliography}
\end{document}